\documentclass[english]{article}
\usepackage[T1]{fontenc}
\usepackage[latin9]{inputenc}
\usepackage[active]{srcltx}
\usepackage{color}
\usepackage{float}
\usepackage{mathtools}
\usepackage{amsmath}
\usepackage{amsthm}
\usepackage{amssymb}
\usepackage{graphicx}
\usepackage{geometry}
\PassOptionsToPackage{normalem}{ulem}
\usepackage{ulem}

\makeatletter
\numberwithin{equation}{section}
\numberwithin{figure}{section}
\theoremstyle{plain}
\newtheorem{thm}{\protect\theoremname}
\theoremstyle{plain}
\newtheorem{lem}[thm]{\protect\lemmaname}
\theoremstyle{definition}
\newtheorem{defn}[thm]{\protect\definitionname}
\theoremstyle{plain}
\newtheorem{cor}[thm]{\protect\corollaryname}
\theoremstyle{remark}
\newtheorem{rem}[thm]{\protect\remarkname}
\theoremstyle{plain}
\newtheorem{prop}[thm]{\protect\propositionname}

\usepackage{color}
\usepackage{matlab-prettifier}
\usepackage{amsthm}
\theoremstyle{plain}

\usepackage{hyperref}

\newtheorem{assump}[thm]{Assumption}

\usepackage{thmtools}

\makeatother

\usepackage{babel}
\usepackage{listings}
\providecommand{\corollaryname}{Corollary}
\providecommand{\definitionname}{Definition}
\providecommand{\lemmaname}{Lemma}
\providecommand{\propositionname}{Proposition}
\providecommand{\remarkname}{Remark}
\providecommand{\theoremname}{Theorem}

\begin{document}
\global\long\def\ve{\varepsilon}%
\global\long\def\R{\mathbb{R}}%
\global\long\def\Rn{\mathbb{R}^{n}}%
\global\long\def\Rd{\mathbb{R}^{d}}%
\global\long\def\E{\mathbb{E}}%
\global\long\def\P{\mathbb{P}}%
\global\long\def\bx{\mathbf{x}}%
\global\long\def\vp{\varphi}%
\global\long\def\ra{\rightarrow}%
\global\long\def\smooth{C^{\infty}}%
\global\long\def\Tr{\mathrm{Tr}}%
\global\long\def\bra#1{\left\langle #1\right|}%
\global\long\def\ket#1{\left|#1\right\rangle }%
\global\long\def\Re{\mathrm{Re}}%
\global\long\def\Im{\mathrm{Im}}%
\global\long\def\bsig{\boldsymbol{\sigma}}%
\global\long\def\btau{\boldsymbol{\tau}}%
\global\long\def\bmu{\boldsymbol{\mu}}%
\global\long\def\bx{\boldsymbol{x}}%
\global\long\def\bups{\boldsymbol{\upsilon}}%
\global\long\def\bSig{\boldsymbol{\Sigma}}%
\global\long\def\bt{\boldsymbol{t}}%
\global\long\def\bs{\boldsymbol{s}}%
\global\long\def\by{\boldsymbol{y}}%
\global\long\def\brho{\boldsymbol{\rho}}%
\global\long\def\ba{\boldsymbol{a}}%
\global\long\def\bb{\boldsymbol{b}}%
\global\long\def\bz{\boldsymbol{z}}%
\global\long\def\bc{\boldsymbol{c}}%
\global\long\def\balpha{\boldsymbol{\alpha}}%
\global\long\def\bbeta{\boldsymbol{\beta}}%
\global\long\def\blam{\boldsymbol{\lambda}}%
\global\long\def\Blam{\boldsymbol{\Lambda}}%
\global\long\def\sign{\mathrm{sign}}%
\global\long\def\diag{\mathrm{diag}}%
\global\long\def\bA{\mathbf{A}}%
\global\long\def\bB{\mathbf{B}}%
\global\long\def\bG{\mathbf{G}}%
\global\long\def\bX{\mathbf{X}}%
\global\long\def\vector{\mathrm{vec}}%
 
\global\long\def\M{\mathrm{M}}%
\global\long\def\S{\mathrm{S}}%
\global\long\def\epsilon{\varepsilon}%

\title{Non-Euclidean dual gradient ascent for entropically regularized linear
and semidefinite programming}
\author{Yuhang Cai\thanks{University of California, Berkeley} \and Michael
Lindsey\footnotemark[1] \thanks{Lawrence Berkeley National Laboratory}}
\maketitle
\begin{abstract}
We present an optimization framework that exhibits dimension-independent
convergence on a broad class of semidefinite programs (SDPs). Our
approach first regularizes the primal problem with the von Neumann
entropy, then solve the regularized problem using dual gradient ascent
with respect to a problem-adapted norm. In particular, we show that
the dual gradient norm converges to zero at a rate independent of
the ambient dimension and, via rounding arguments, construct primal-feasible
solutions in certain special cases. We also derive explicit convergence
rates for the objective. In order to achieve optimal computational
scaling, we must accommodate the use of stochastic gradients constructed
via randomized trace estimators. Throughout we illustrate the generality
of our framework via three important special cases---the Goemans-Williamson
SDP relaxation of the Max-Cut problem, the optimal transport linear
program, and several SDP relaxations of the permutation synchronization
problem. Numerical experiments confirm that our methods achieve dimension-independent
convergence in practice.
\end{abstract}

\section{Introduction}

Semidefinite programming is a powerful optimization framework with
widespread applications in combinatorial optimization \cite{goemans1995improved},
control theory \cite{boyd1994linear,parrilo2000structured}, quantum
information \cite{doherty2002distinguishing}, and machine learning
\cite{lanckriet2004learning,weinberger2009distance}. A standard semidefinite
program (SDP) seeks to minimize a linear functional over the set of
symmetric positive semidefinite matrices subject to affine constraints.
Despite their convexity, large-scale SDPs remain computationally challenging
to solve. In particular, checking the semidefinite constraint on an
$n\times n$ matrix $X$ in general requires $\Omega(n^{3})$ operations,
though the scaling of robust solvers such as interior point methods
\cite{boyd2004convex} can be far more costly still.

When the optimal solution is of rank $k\ll n$, specialized approaches
are able to avoid cubic scaling and achieve per-iteration complexity
of $O(nk^{2})$. Such approaches include the Burer-Monteiro / SDPLR
method \cite{BurerMonteiro} and alternative manifold-constrained
optimization approaches \cite{manopt}, as well as the randomized
approach SketchyCGAL \cite{SketchyCGAL}. In this work we are interested
in algorithms that make no assumption on the solution rank.

Meanwhile, the celebrated matrix multiplicative weights update (MMWU)
algorithm \cite{arora2012multiplicative}, enjoys strong theoretical
guarantees for certain SDPs such as the Goemans-Williamson relaxation
of the Max-Cut problem \cite{goemans1995improved}, achieving nearly-linear
scaling for graphs of bounded degree. However, MMWU has not been widely
adopted as a practical general-purpose tool for semidefinite programming,
partly due to its reliance on problem-specific oracles, and to the
best of our knowledge no implementations are publicly available. In
fact, MMWU is similar to our approach in that it implicitly represents
of the primal variable as a matrix exponential, and the approximate
Cholesky factorization referenced in MMWU can be understood in our
language as a choice of a randomized trace estimator. Our approach,
like MMWU, achieves nearly-linear scaling on the Max-Cut SDP for graphs
of bounded degree, but we also demonstrate the flexibility of the
approach for more general semidefinite programming. Moreover, our
algorithm is conceptually quite simple and can be derived simply as
dual gradient ascent with respect to a carefully chosen norm, after
a suitable entropic regularization of the primal problem.

Indeed, the success of entropic regularization in accelerating algorithms
for problems like optimal transport (OT) \cite{peyre2019computational}
provides a compelling inspiration. Although OT is a linear program
(LP) problem, its importance has inspired specialized algorithms beyond
standard LP solvers. Specifically, the celebrated Sinkhorn scaling
algorithm was proposed \cite{cuturi2013sinkhorn} as a method for
solving an entropically regularized variant of the OT problem. This
approach and a variant known as the Greenkhorn algorithm \cite{altschuler2017near}
have been shown to achieve nearly-linear complexity \cite{altschuler2017near,dvurechensky2018computational,lin2022efficiency},
attaining $\epsilon$ additive accuracy with a computational complexity
of $\tilde{O}(n^{2}/\epsilon^{2})$ for $n\times n$ transport plans.
Although more sophisticated approaches not based on entropic regularization
have been shown to achieve $\tilde{O}(n^{2}/\epsilon)$ scaling for
this problem (see, e.g., \cite{li2024fastcomputationoptimaltransport}
for a recent approach as well as a comprehensive review), the conceptual
and algorithmic simplicity of the entropic regularization remain extremely
compelling in practice. Moreover the regularization itself can be
viewed as beneficial for certain tasks \cite{pooladian2024entropicestimationoptimaltransport}.

Our work builds on the recent work \cite{lindsey2023fast}, which
introduces the general framework of entropically regularized semidefinite
programming, together with fast dual optimizers that achieve empirically
optimal scaling with respect to the problem dimension by using a variant
of Hutchinson's randomized trace estimator \cite{hutchinson1989stochastic,meyer2021hutch++}.
We comment that entropic regularization has appeared in several other
contexts of linear and semidefinite programming \cite{krechetov2018entropy,li2024fastcomputationoptimaltransport,lin2022variational,pavlov2024gibbs,pavlov2024logarithmically},
though to our knowledge \cite{lindsey2023fast} is the first to advance
it as a framework allowing for fast general semidefinite programming.
In \cite{lindsey2023fast}, numerical experiments are presented for
the Max-Cut SDP as well as an application in spectral clustering,
exhibiting dimension-independent convergence, but no rigorous convergence
theory is offered.

In this paper, we address this gap by considering a somewhat modified
optimization approach within this entropically regularized framework,
specifically dual gradient ascent or stochastic gradient ascent \cite{beck2017first}
with respect to a problem-adapted non-Euclidean norm, developing a
rigorous dimension-independent convergence theory for a broad class
of entropically regularized semidefinite programs. Our theoretical
contributions are summarized as follows.
\begin{itemize}
\item \textbf{Norm derivation: }Our algorithms and analysis define and make
use of a problem-adapted norm on the dual variable space, rather than
the standard Euclidean norm. The correct choice of norm for a given
problem depends on the structure of the constraints, but the procedure
for establishing the norm ultimately derives from the fact that the
von Neumann entropy is strongly convex with respect to the trace or
nuclear norm \cite{carlen2014remainder}.
\item \textbf{Gradient convergence and rounding: }We show that the dual
gradients converge to zero under quite general conditions, implying
the existence of a nearly primal-feasible approximate solution. We
further develop rounding arguments for specific SDPs to construct
a primal-feasible solution. Combining the gradient convergence and
the rounding arguments, we get an explicit bound on the primal objective.
Note that this strategy bears a resemblance at a high level to the
original nearly-linear scaling analysis \cite{altschuler2017near}
of the Sinkhorn and Greenkhorn algorithms, though this analysis has
been subsequently improved \cite{dvurechensky2018computational,lin2022efficiency}.
\item \textbf{Objective convergence: }We establish an explicit objective
convergence via an analysis of (possibly stochastic) gradient descent
on a smooth objective with respect to a general norm. Specifically,
we show that stochastic gradient descent converges exponentially fast
up to a bias, where the exponential rate and the bias are determined
by the gradient estimation error. To the best of our knowledge, such
an analysis of SGD is novel, though we point out that prior work \cite{beck2003mirror,nemirovskij1983problem}
has considered stochastic mirror descent under general norms.
\end{itemize}
In addition to the general theory, we study three concrete instances
of entropically regularized SDPs: the Goemans-Williamson SDP relaxation
of the Max-Cut problem \cite{goemans1995improved}, the optimal transport
linear program \cite{cuturi2013sinkhorn}, and several SDP relaxations
of the permutation synchronization problem \cite{pachauri2013solving}.
For each problem, we derive the dual of its entropic regularization
and develop the corresponding optimization algorithm. We summarize
some specific findings as follows.
\begin{itemize}
\item \textbf{OT LP: }We prove that our approach achieves $\tilde{O}(n^{2}/\epsilon^{2})$
complexity, matching that of the celebrated Sinkhorn algorithm \cite{cuturi2013sinkhorn}
and Greenkhorn algorithm \cite{altschuler2017near}. More specifically,
to solve the regularized problem with inverse temperature or regularization
parameter $\beta$, the complexity is $\tilde{O}(\beta n^{2}/\epsilon)$.
The optimal value in turn achieves $\epsilon$ additive error for
$\beta=\tilde{O}(\epsilon^{-1})$. Compared to other standard approaches
to the regularized problem, our algorithm is more flexible in that
it allows for the incorporation of additional constraints through
our more general framework. Moreover, the results for general $\beta$
(i.e., not necessarily in the $\beta=\tilde{O}(\epsilon^{-1})$ regime)
are new to the best of our knowledge, and we also obtain \emph{a priori}
bounds on the dual-optimal Kantorovich potentials that may be of general
interest.
\item \textbf{Max-Cut SDP} and \textbf{permutation synchronization SDPs:
}Inspired by the rounding argument for the OT LP \cite{altschuler2017near},
we develop rounding arguments for the Max-Cut SDP and a permutation
synchronization SDP. Then we can prove that additive accuracy $\epsilon$
is obtained in $\tilde{O}(\epsilon^{-3})$ iterations of our algorithm,
in which each iteration requires $\tilde{O}(\epsilon^{-2})$ samples
in the gradient estimator. Notably, the overall dimension-dependence
is near-optimal. To our knowledge, no such optimal-scaling algorithm
has been established for the SDP in the permutation synchronization
case.
\end{itemize}
Finally, we complement our theoretical results with numerical experiments
to validate the practical performance of our methods.

We conclude the introduction with an outline of this paper. In Section
\ref{sec:Preliminaries}, we present background and basic results
on the entropic regularization of general SDPs, examining linear programming
as a special case. We also introduce our specific problems of interest.
In Section \ref{sec:dual}, we derive the general dual problem and
its smoothness properties with respect to a carefully constructed
norm, working out further details for our specific problems of interest.
In Section \ref{sec:algo}, we describe the (stochastic) gradient
ascent algorithm for the dual problem induced by the choice of norm,
again working out further details for our case studies. In Section
\ref{sec:conv}, we provide the convergence analysis for the general
algorithm. In Section \ref{sec:ot-conv}, we provide a refined analysis
for OT. In Sections \ref{sec:max-cut-rounding} and \ref{sec:perm-synch-rounding},
we present rounding arguments for the Max-Cut SDP and an SDP relaxation
of the permutation synchronization problem. In Section \ref{sec:exp},
we present numerical experiments for our problems of interest.

\subsection*{Acknowledgments}

This material is based on work supported by the U.S. Department of
Energy, Office of Science, Accelerated Research in Quantum Computing
Centers, Quantum Utility through Advanced Computational Quantum Algorithms,
grant no. DE-SC0025572 and by the Applied Mathematics Program of the
US Department of Energy (DOE) Office of Advanced Scientific Computing
Research under contract number DE-AC02-05CH11231. M.L. was also partially
supported by a Sloan Research Fellowship. The authors gratefully acknowledge
conversations with Nikhil Srivastava and Aaron Sidford.

\section{Preliminaries\label{sec:Preliminaries}}

We begin with some background on semidefinite and linear programming
and their entropic regularizations.

Throughout this manuscript, we denote the set of non-negative real
numbers by $\R_{+}$, the set of integers $\{1,...,n\}$ by $[n]$,
and the $n$-dimensional simplex by $\Delta_{n}\coloneqq\{x\in\R_{+}^{n}:\sum_{i=1}^{n}x_{i}=1\}$.
We also use $\mathbf{1}_{n}$ to denote the vector of all ones in
$\R^{n}$.

\subsection{Unregularized SDP / LP}

We consider the semidefinite program (SDP) 
\begin{align}
\underset{X\in\R^{n\times n}}{\text{minimize}}\ \  & \Tr[CX]\label{eq:unreg_sdp}\\
\text{subject to}\ \  & \Tr[\mathbf{A}X]=\mathbf{b},\nonumber \\
 & \Tr[X]=1,\nonumber \\
 & X\succeq0.\nonumber 
\end{align}
 Here the second line is a shorthand for the constraints $\Tr[A_{i}X]=b_{i}$
for $i=1,\ldots,m$, where the $A_{i}$ are symmetric. We also assume
that $C$ is symmetric. (Both $C$ and the $A_{i}$ can always be
symmetrized, so the symmetry assumption loses no generality.) The
SDP can be viewed as an optimization problem over \emph{density operators},
in the sense of quantum information \cite{NielsenChuang}, i.e., operators
$X$ satisfying $X\succeq0$ and $\Tr[X]=1$.

In the `diagonal' case where $C=\mathrm{diag(c)}$ and all the $A_{i}=\mathrm{diag}(a_{i})$
are diagonal matrices, this problem descends to an optimization problem
(in fact, a \emph{linear program} or LP) over \emph{probability vectors}
$x\in\R^{n}$, satisfying $x\geq0$ and $\mathbf{1}^{\top}x=1$: 
\begin{align}
\underset{x\in\R^{n}}{\text{minimize}}\ \  & c^{\top}x\label{eq:unreg_lp}\\
\text{subject to}\ \  & Ax=b,\nonumber \\
 & \mathbf{1}_{n}^{\top}x=1,\nonumber \\
 & x\geq0,\nonumber 
\end{align}
 where $A=(a_{1},\ldots,a_{m})^{\top}\in\R^{m\times n}$.

\subsection{Key cases of interest \label{subsec:cases}}

Throughout this article we will highlight several important special
cases.

\subsubsection{Max-Cut SDP}

The first is the \emph{Max-Cut SDP}, which arises as the celebrated
SDP relaxation due to Goemans and Williamson of the Max-Cut problem
\cite{goemans1995improved}: 
\begin{align}
\underset{X\in\R^{n\times n}}{\text{minimize}}\ \  & \Tr[CX]\label{eq:maxcut_sdp}\\
\text{subject to}\ \  & \mathrm{diag}(X)=\mathbf{b},\nonumber \\
 & \Tr[X]=1,\nonumber \\
 & X\succeq0.\nonumber 
\end{align}
 Typically $\mathbf{b}=\mathbf{1}_{n}/n$, but more generally we may
consider any $\mathbf{b}$ such that $\mathbf{1}_{n}^{\top}\mathbf{b}=1$.
Here the trace constraint is redundant, but it is useful to maintain
it separately to derive a particular dual formulation.

\subsubsection{Optimal transport LP}

The second special case is the\emph{ discrete (Kantorovich) optimal
transport (OT) problem}. This problem is specified by marginal probability
vectors $\mu\in\Delta^{m}$ and $\nu\in\Delta^{n}$, as well a cost
matrix $c\in\R^{m\times n}$. (Although $c$ is viewed here as a matrix,
we keep the notation lower-case to emphasize its distinctness from
$C$ as defined above.) Then the \emph{OT LP} is written: 
\begin{align}
\underset{\pi\in\R^{m\times n}}{\text{minimize}}\ \  & \left\langle c,\pi\right\rangle \label{eq:ot_lp}\\
\text{subject to}\ \  & \pi\mathbf{1}_{n}=\mu,\nonumber \\
 & \pi^{\top}\mathbf{1}_{m}=\nu,\nonumber \\
 & \left\langle \pi,\mathbf{1}_{m\times n}\right\rangle =1,\nonumber \\
 & \pi\ge0.
\end{align}
 Here $\left\langle \,\cdot\,,\,\cdot\,\right\rangle $ is the usual
Frobenius inner product. Note carefully that in our notation for this
special case, we are over-riding some of our notation in the general
case (namely: $m$ and $n$) in order to maintain accepted notation
in the OT setting, so we pause here to explain the connection. To
view (\ref{eq:ot_lp}) as a special case of (\ref{eq:unreg_lp}),
we view the \emph{vectorization} of $\pi$ as corresponding to a probability
vector $x\in\R^{mn}$. Then the $m+n$ marginal equality constraints
can be phrased as $Ax=b$ for a suitable choice of $A\in\R^{(m+n)\times mn}$
and $b^{\top}=(\mu^{\top},\nu^{\top})\in\R^{m+n}$, and the cost vector
$c\in\R^{mn}$ is obtained as the vectorization of $c$ (for which
we maintain the same notation, abusing notation slightly). 

\subsubsection{SDPs for permutation synchronization}

Finally, we consider SDPs which are motivated by the permutation synchronization
(henceforward Perm-Synch) problem \cite{pachauri2013solving}. Recently,
the entropic regularization of this SDP was considered in \cite{lindseyshi2025}
where it was observed that the regularization enjoys favorable qualitative
properties from the point of view of the Perm-Synch problem. We will
develop an alternative optimization framework for the regularized
Perm-Synch SDPs considered in that work.

In this setting we consider a cost matrix $C$ which is $(NK)\times(NK)$
with $K\times K$ blocks denoted $C^{(i,j)}$ for $i,j=1,\ldots,N$.
(More generally it is possible to consider blocks of different sizes,
but we ignore this possibility for notational simplicity.) In the
Perm-Synch problem, one may think of $N$ as a number of images and
$K$ as a number of keypoints identified in each image. Then $-C^{(i,j)}$
is a one-hot encoding of tentative pairwise correspondences identified
between the keypoints in images $i$ and $j$, which formally can
be viewed as a submatrix of a permutation matrix. In the Perm-Synch
problem, we aim to find a cycle-consistent matching between keypoints
in all image pairs that is close to the tentative (not necessarily
cycle-consistent) matching encoded in $C$. 

Then the \emph{strong} Perm-Synch SDP that we consider is the following
problem: 

\begin{align}
\underset{X\in\R^{n\times n}}{\text{minimize}}\ \  & \Tr[CX]\label{eq:strongpermsynch_sdp}\\
\text{subject to}\ \  & X^{(i,i)}=\frac{\mathbf{I}_{K}}{n},\quad i=1,\ldots,N\nonumber \\
 & \Tr[X]=1,\nonumber \\
 & X\succeq0.\nonumber 
\end{align}
 Here $n=NK$, and we have normalized the diagonal block constraints
so that any feasible $X$ satisfies $\Tr[X]=1$.

We also consider the following \emph{weak }Perm-Synch SDP which slightly
relaxes the block-diagonal constraints in (\ref{eq:strongpermsynch_sdp}): 

\begin{align}
\underset{X\in\R^{n\times n}}{\text{minimize}}\ \  & \Tr[CX]\label{eq:weakpermsynch_sdp}\\
\text{subject to}\ \  & \Tr\left[\frac{\mathbf{1}_{K}\mathbf{1}_{K}^{\top}}{K}X^{(i,i)}\right]=1/n,\quad i=1,\ldots,N\nonumber \\
 & \mathrm{diag}(X)=\mathbf{1}_{n}/n,\nonumber \\
 & \Tr[X]=1,\nonumber \\
 & X\succeq0.\nonumber 
\end{align}

See \cite{lindseyshi2025} for more details and background, as well
as a description of how to recover a cycle-consistent matching from
the Perm-Synch SDPs.

\subsection{Entropic regularization}

We can regularize the SDP (\ref{eq:unreg_sdp}) using the von Neumann
entropy \cite{vonneumann1932mathematische}, defined by 
\begin{equation}
S(X):=\Tr[X\log X]\label{eq:SX}
\end{equation}
 for $X\succeq0$ (and $S(X)=+\infty$ otherwise), to obtain the regularized
SDP: 
\begin{align}
\underset{X\in\R^{n\times n}}{\text{minimize}}\ \  & \Tr[CX]+\beta^{-1}S(X)\label{eq:reg_sdp}\\
\text{subject to}\ \  & \Tr[\mathbf{A}X]=\mathbf{b},\nonumber \\
 & \Tr[X]=1.\nonumber 
\end{align}
 Here $\beta\in(0,\infty)$ can be interpreted physically as an inverse
temperature.

Note that the domain 
\begin{equation}
\mathcal{P}_{1}:=\{X\in\R^{n\times n}\,:\,X\succeq0,\,\Tr[X]=1\}\label{eq:P1}
\end{equation}
 of all density operators is compact, so all of the problems defined
above admit minimizers, provided that (\ref{eq:unreg_sdp}) (equivalently,
(\ref{eq:reg_sdp})) is feasible. We will always assume feasibility.
Then by strict convexity of the entropy, there is a unique minimizer
$X_{\star,\beta}$ of the regularized problem (\ref{eq:reg_sdp}),
attaining the optimal value $p_{\star,\beta}$. Meanwhile, let $X_{\star,\infty}$
be any solution of the unregularized problem, attaining the optimal
value $p_{\star,\infty}$. When the context is clear, we may simply
omit $\beta\in(0,\infty]$ from the subscript and write $X_{\star}$,
$p_{\star}$.

By Sion's minimax theorem, \cite{Komiya1988}, strong duality holds.
Moreover, if we assume that Slater's condition \cite{boyd2004convex}
holds, i.e., that there exists $X\succ0$ which is feasible, then
the dual problem admits a maximizer. (Note that all of our specific
problems of interest satisfy Slater's condition.) However, we will
not explicitly assume that Slater's condition holds. Indeed, Theorem
\ref{thm:gradconv} does not require the existence of a dual optimizer.
On the other hand, Theorem \ref{thm:objconv} is only vacuous if there
does not exist a dual optimizer.

Note that over $\mathcal{P}_{1}$, the entropy $S$ is maximized by
$X=uu^{\top}$ for unit vectors $u$, for which $S(X)=0$. Meanwhile,
$S$ is minimized over this domain by $X=\mathbf{I}_{n}/n$, for which
$S(X)=-\log n$. Therefore the `entropic diameter' of $\mathcal{P}_{1}$
is $\log n$. The following simple lemma based on this fact (analogous
to \cite{pmlr-v75-weed18a}), guarantees that regularization by inverse
temperature $\beta$ only changes the optimal value by $O(\beta^{-1}\log n)$.
\begin{lem}
\label{lem:reg-bound}With definitions and notation as in the preceding,
\[
\Tr[CX_{\star,\infty}]\leq\Tr[CX_{\star,\beta}]\leq\Tr[CX_{\star,\infty}]+\beta^{-1}\log n
\]
\end{lem}

\begin{proof}
By optimality for the unregularized problem $\Tr[CX_{\star,\infty}]\leq\Tr[CX_{\star,\beta}]$.
Meanwhile, by optimality for the regularized problem, 
\[
\Tr[CX_{\star,\beta}]+\beta^{-1}S(X_{\star,\beta})\leq\Tr[CX_{\star,\infty}]+\beta^{-1}S(X_{\star,\infty}),
\]
 which implies 
\[
\Tr[CX_{\star,\beta}]\leq\Tr[CX_{\star,\infty}]+\beta^{-1}\left[S(X_{\star,\infty})-S(X_{\star,\beta})\right]\leq\Tr[CX_{\star,\infty}]+\beta^{-1}\log n.
\]
\end{proof}
For concreteness, we point out that in the diagonal case, the regularized
SDP (\ref{eq:maxcut_sdp}) descends to a regularized LP 
\begin{align}
\underset{x\in\R^{n}}{\text{minimize}}\ \  & c^{\top}x+\beta^{-1}S(x)\label{eq:reg_lp}\\
\text{subject to}\ \  & Ax=b,\nonumber \\
 & \mathbf{1}_{n}^{\top}x=1,\nonumber \\
 & x\geq0,\nonumber 
\end{align}
 where by some abuse of notation, here we view $S$ as the Shannon
entropy 
\[
S(x)=\sum_{i=1}^{n}\left[x_{i}\log x_{i}-x_{i}\right],
\]
 defined on $x\geq0$. In the diagonal case, we can view (\ref{eq:reg_lp})
either as a simplification of (\ref{eq:maxcut_sdp}), or more directly
as an entropic regularization of (\ref{eq:unreg_lp}). These interpretations
coincide because the minimal-entropy density operator $X\in\mathcal{P}_{1}$
satisfying $\mathrm{diag}(X)=x$, for any given probability vector
$x$, is the diagonal density operator $X=\mathrm{diag}(x)$. (This
is a consequence of the Gibbs variational principle \cite{georgii2011gibbs}.)

Moreover, observe that the regularized LP (\ref{eq:reg_lp}) in the
case of the OT problem (\ref{eq:ot_lp}) coincides with the celebrated
entropic regularization of OT \cite{cuturi2013sinkhorn}.

\section{Dual problem \label{sec:dual}}

\subsection{General form}

Within (\ref{eq:reg_sdp}), we dualize only the $\Tr[\mathbf{A}X]=\mathbf{b}$
constraint (with dual variable $\blam$), leaving the $\Tr[X]=1$
constraint over the primal domain intact. Then we form the Lagrangian
\begin{align*}
\mathcal{L}_{\beta}(X,\blam) & =\Tr[CX]+\beta^{-1}S(X)-\blam\cdot\left(\Tr[\mathbf{A}X]-\mathbf{b}\right)\\
 & =\mathbf{b}\cdot\blam+\Tr\left[\left(C-\blam\cdot\mathbf{A}\right)X\right]+\beta^{-1}S(X)
\end{align*}
 and define the unconstrained dual objective 
\begin{align*}
g_{\beta}(\blam) & :=\min_{X\in\mathcal{P}_{1}}\mathcal{L}_{\beta}(X,\blam)\\
 & =\mathbf{b}\cdot\blam+\min_{X\in\mathcal{P}_{1}}\left\{ \Tr\left[\left(C-\blam\cdot\mathbf{A}\right)X\right]+\beta^{-1}S(X)\right\} .
\end{align*}
 The minimum is attained by 
\begin{equation}
X_{\beta,\blam}:=\frac{e^{-\beta(C-\blam\cdot\mathbf{A})}}{\Tr\left[e^{-\beta(C-\blam\cdot\mathbf{A})}\right]}\in\mathcal{P}_{1},\label{eq:Xbetalambda}
\end{equation}
 and it follows that 
\[
g_{\beta}(\blam)=\mathbf{b}\cdot\blam-\beta^{-1}\log\Tr\left[e^{-\beta(C-\blam\cdot\mathbf{A})}\right].
\]
 For later purposes it is useful to define as well the unnormalized
density matrix 
\begin{equation}
Y_{\beta,\blam}:=e^{-\beta(C-\blam\cdot\mathbf{A})},\label{eq:Ybetalambda}
\end{equation}
 so that $X_{\beta,\blam}=Y_{\beta,\blam}/\Tr[Y_{\beta,\blam}].$

Then we seek to solve the dual optimization problem.
\[
\underset{\blam\in\R^{m}}{\text{maximize}}\ g_{\beta}(\blam).
\]
 Indeed, note that the first-order optimality condition is that 
\[
\Tr[\mathbf{A}X_{\beta,\blam}]=\mathbf{b},
\]
 which is the primal feasibility constraint. Letting $\blam_{\star,\beta}$
denote a dual solution (assuming that it exists), observe then that
$X_{\beta,\blam_{\star,\beta}}=X_{\star,\beta}$ is the primal solution.
As with $X_{\star,\beta}$, when the context is clear, we may simply
omit $\beta$ from the subscript and write $\blam_{\star}$ for a
dual optimizer.

Instead of maximizing $g$, we will minimize 
\[
f_{\beta}:=-g_{\beta}
\]
 to clarify the connection with conventional optimization theory.

Note that in the LP case (\ref{eq:reg_lp}), the dual objective can
be written more simply 
\[
g_{\beta}(\blam)=\mathbf{b}\cdot\blam-\beta^{-1}\log\sum_{i=1}^{n}\left[e^{-\beta(c_{i}-a_{i}\cdot\blam)}\right].
\]

\subsection{Strong smoothness}

In order to explain the smoothness properties of the dual problem,
we first recall the definitions for strong convexity and smoothness.
These definitions were first introduced in \cite{nemirovskij1983problem}.
\begin{defn}
[Strong convexity]

A function $f:\R^{p}\to\R$, is $\mu$-strongly convex with respect
to a norm $\|\cdot\|$ on $\R^{p}$ if
\[
f(t\cdot y+(1-t)\cdot x)\le t\,f(y)+(1-t)\,f(x)-\frac{\mu}{2}\,t\,(1-t)\,\|x-y\|^{2},
\]
for all $x,y\in\R^{p}$ and $t\in[0,1]$.
\end{defn}

\begin{defn}
[Smoothness]

A $f:\R^{p}\to\R$ is $L$-smooth with respect to a norm $\|\cdot\|$
on $\R^{p}$ if $f$ is everywhere differentiable and 
\[
\|\nabla f(x)-\nabla f(y)\|_{*}\le L\,\|x-y\|,
\]
for all $x,y\in\R^{p}$. 
\end{defn}

Throughout we will use $\|\cdot\|_{*}$ to denote the dual norm corresponding
to any choice of norm $\|\cdot\|.$ Now we state a well-known fact
\cite{carlen2014remainder} and derive some implications.
\begin{thm}
$S$ is 1-strongly convex on $\mathcal{P}_{1}$ with respect to the
nuclear norm $\Vert\,\cdot\,\Vert_{\Tr}$.
\end{thm}

It is is useful to extend the definition of the von Neumann entropy
$S(X)$ (\ref{eq:SX}) to the Euclidean vector space of symmetric
$n\times n$ matrices via the convention $S(X)=+\infty$ for $X\notin\mathcal{P}_{1}$.
Then we can define the suitable convex conjugate
\[
S^{*}(Y):=\sup_{X\in\R^{n\times n},\,X=X^{\top}}\left\{ \left\langle X,Y\right\rangle -S(X)\right\} ,\quad\text{for all }Y\in\R^{n\times n},\,Y=Y^{\top},
\]
 where $\left\langle \,\cdot\,,\,\cdot\,\right\rangle $ denotes the
Frobenius inner product and the explicit expression on the right-hand
side can be verified by direct computation. 

The next result follows from the relation \cite{kakade2009duality}
between strong convexity of the primal and strong smoothness of the
dual.
\begin{cor}
The convex conjugate $S^{*}$ is 1-strongly smooth with respect to
the spectral norm $\Vert\,\cdot\,\Vert_{2}$.
\end{cor}

\begin{proof}
The result follows directly from Theorem 6 in \cite{kakade2009duality},
via the observation that the spectral norm is the dual norm of the
nuclear norm.
\end{proof}
In order to define the appropriate notion of smoothness for our objective
$f(\blam)$, we need to take the structure of the constraints into
consideration.

\begin{assump}

\label{assump:norm}Assume that $\Vert\,\cdot\,\Vert$ is a norm on
$\R^{m}$ such that $\Vert\blam\cdot\mathbf{A}\Vert_{2}\leq\Vert\blam\Vert$.
In the LP case, this descends to an assumption that $\Vert A^{\top}\blam\Vert_{\infty}\leq\Vert\blam\Vert$.

\end{assump}
\begin{rem}
Note that $\Vert\blam\cdot\mathbf{A}\Vert_{2}$ itself defines a norm
on $\blam$, but for algorithmic reasons, it may be useful to upper-bound
this norm with a simpler norm.
\end{rem}

\begin{cor}
$g_{\beta}$ is $\beta$-strongly smooth with respect to $\Vert\,\cdot\,\Vert$.
\end{cor}

\begin{proof}
It suffices to show that the `free energy' 
\[
\Omega(\blam):=\beta^{-1}\log\Tr\left[e^{-\beta(C-\blam\cdot\mathbf{A})}\right]=\beta^{-1}S^{*}\left(-\beta\left[C-\blam\cdot\mathbf{A}\right]\right)
\]
 is strongly smooth.

Now by the strong smoothness of $S^{*}$ we have
\begin{align*}
\Omega(\blam+\blam') & =\beta^{-1}S^{*}(-\beta\left[C-\blam\cdot\mathbf{A}\right]+\beta\left[\blam'\cdot\mathbf{A}\right])\\
 & \leq\beta^{-1}S^{*}(-\beta\left[C-\blam\cdot\mathbf{A}\right])+\left\langle \nabla S^{*}(-\beta C+\blam\cdot\mathbf{A}),\blam'\cdot\mathbf{A}\right\rangle +\frac{\beta}{2}\Vert\blam'\cdot\mathbf{A}\Vert_{2}^{2}\\
 & =\Omega(\blam)+\left\langle X_{\beta,\blam},\blam'\cdot\mathbf{A}\right\rangle +\frac{\beta}{2}\Vert\blam'\cdot\mathbf{A}\Vert_{2}^{2}.
\end{align*}
 Note that $X_{\beta,\blam}=\nabla\Omega(\blam)$ and $\Vert\blam'\cdot\mathbf{A}\Vert_{2}^{2}\leq\Vert\blam'\Vert^{2}$,
so we have precisely shown $\beta$-strong smoothness with respect
to $\Vert\,\cdot\,\Vert$.
\end{proof}

\subsection{Case studies \label{subsec:normcases}}

Now we discuss how the norm $\Vert\,\cdot\,\Vert$ can be chosen to
guarantee Assumption \ref{assump:norm} in each of our special cases.

\subsubsection{Max-Cut SDP \label{subsec:normMaxCut}}

In this case, we have $\mathbf{A}=[e_{1}e_{1}^{\top},\ldots,e_{n}e_{n}^{\top}]$,
and therefore $\blam\cdot\mathbf{A}=\mathrm{diag}(\blam)$. This means
that we can simply take 
\[
\Vert\blam\Vert=\Vert\blam\Vert_{\infty}
\]
 to guarantee Assumption \ref{assump:norm}. Observe that the dual
norm is the ordinary vector 1-norm.

Note that in this context, concretely we have 
\begin{equation}
X_{\beta,\blam}\propto e^{-\beta(C-\mathrm{diag}(\blam))},\label{eq:Xmaxcut}
\end{equation}
 suitably normalized.

\subsubsection{OT LP\label{subsec:normOT} }

In this case the dual problem takes the form 
\[
g_{\beta}(\phi,\psi)=\mu^{\top}\phi+\nu^{\top}\psi-\beta^{-1}\log\sum_{i,j}e^{-\beta(c_{ij}-\phi_{i}-\psi_{j})}
\]
 where we identify the dual variable as $\blam=(\phi,\nu)$. Translating
Assumption \ref{assump:norm} to this setting, we need our norm to
satisfy 
\[
\Vert(\phi,\nu)\Vert\geq\max_{i,j}\vert\phi_{i}+\psi_{j}\vert.
\]
 Therefore we can take 
\[
\Vert(\phi,\psi)\Vert=\sqrt{2\left(\Vert\phi\Vert_{\infty}^{2}+\Vert\psi\Vert_{\infty}^{2}\right)}.
\]
Observe that the dual norm is the norm is given by 
\[
\Vert(a,b)\Vert_{*}=\sqrt{\frac{1}{2}\left(\Vert a\Vert_{1}^{2}+\Vert b\Vert_{1}^{2}\right)},
\]
 as we verify in Appendix \ref{app:dual-norm:ot}.

\subsubsection{Strong Perm-Synch SDP \label{subsec:normStrongPermSynch} }

For the Perm-Synch SDP (\ref{eq:strongpermsynch_sdp}), it is more
natural to identify the the dual variable with an object $\Blam$
consisting of symmetric blocks $\Lambda^{(1)},\ldots,\Lambda^{(N)}\in\R^{K\times K}$.
Then the matrix $\blam\cdot\mathbf{A}$ is identified with the block
diagonal matrix $\bigoplus_{i=1}^{N}\Lambda^{(i)}$, so we may take
\[
\Vert\Blam\Vert=\max_{i=1,\ldots,N}\Vert\Lambda^{(i)}\Vert_{2}.
\]
 Observe that the dual norm, evaluated on a vector $\mathbf{B}=(B^{(1)},\ldots,B^{(N)})$
of $K\times K$ symmetric matrices is given by 
\[
\Vert\mathbf{B}\Vert_{*}=\sum_{i=1}^{N}\Vert B^{(i)}\Vert_{\Tr},
\]
 where $\Vert\,\cdot\,\Vert_{\Tr}$ denotes the trace or nuclear norm
(i.e., the sum of the singular values). We defer the detailed derivation
of the dual norm to Appendix \ref{app:dual-norm:strong-ps}.

Note that in this context, concretely we have 
\begin{equation}
X_{\beta,\Blam}\propto e^{-\beta(C-\bigoplus_{i=1}^{N}\Lambda^{(i)})},\label{eq:Xstrong}
\end{equation}
 suitably normalized.

\subsubsection{Weak Perm-Synch SDP \label{subsec:normWeakPermSynch} }

For the weak permutation synchronization SDP (\ref{eq:weakpermsynch_sdp})
it is natural to label the dual variables as $(\blam,\bmu)$. Then
the matrix $\blam\cdot\mathbf{A}$ (noting that we have overloaded
the notation for $\blam$ temporarily) corresponds to 
\[
\mathrm{diag}(\blam)+\bigoplus_{i=1}^{N}\mu^{(i)}\frac{\mathbf{1}_{K}\mathbf{1}_{K}^{\top}}{K},
\]
 so it suffices to take 
\[
\Vert(\blam,\bmu)\Vert=\sqrt{2\left(\Vert\blam\Vert_{\infty}^{2}+\Vert\bmu\Vert_{\infty}^{2}\right)}.
\]
 Observe that the dual norm is given by 
\[
\Vert(a,b)\Vert_{*}=\sqrt{\frac{1}{2}\left(\Vert a\Vert_{1}^{2}+\Vert b\Vert_{1}^{2}\right)}.
\]
 The verification of the dual norm is identical to the one offered
for the OT LP norm.

Note that in this context, concretely we have 
\begin{equation}
X_{\beta,(\blam,\bmu)}\propto\exp\left[-\beta\left(C-\mathrm{diag}(\blam)-\bigoplus_{i=1}^{N}\mu^{(i)}\frac{\mathbf{1}_{K}\mathbf{1}_{K}^{\top}}{K}\right)\right],\label{eq:Xweak}
\end{equation}
 suitably normalized.

\section{Algorithm \label{sec:algo}}

\subsection{General form}

Consider gradient descent on $f_{\beta}$ with respect to the norm
$\Vert\,\cdot\,\Vert$ satisfying Assumption \ref{assump:norm} (cf.
Section 10.9 in \cite{beck2017first}), with step size $\eta$. Given
an initial guess $\blam_{0}\in\R^{m}$, this algorithm is defined
by setting 
\[
\blam_{t+1}=\underset{\blam\in\R^{m}}{\text{argmin}}\left\{ f_{\beta}(\blam_{t})+\nabla f_{\beta}(\blam_{t})\cdot(\blam-\blam_{t})+\frac{1}{2\eta}\Vert\blam-\blam_{t}\Vert^{2}\right\} .
\]
 We let $T$ denote the time horizon for the algorithm, which furnishes
iterates $\blam_{1},\ldots,\blam_{T}$.

Observe that the dual gradient can be evaluated as 
\begin{equation}
\nabla f_{\beta}(\blam)=\Tr\left[\mathbf{A}X_{\beta,\blam}\right]-\mathbf{b},\label{eq:dualgrad}
\end{equation}
 which is the primal feasibility gap for $X_{\beta,\blam}$. We may
not be able to compute the trace efficiently, if we want to avoid
forming the full matrix exponential $e^{-\beta(C-\blam\cdot\mathbf{A})}\propto X_{\beta,\blam}$.

We can proceed by randomized trace estimation. If we let $z\in\R^{n}$
have i.i.d. Rademacher entries (so that $\E[z]=0$, $\E[zz^{\top}]=\mathbf{I}_{n}$)
and define 
\[
w=e^{-\frac{1}{2}\beta(C-\blam\cdot\mathbf{A})}z,
\]
 then 
\[
\E\left[w^{\top}\mathbf{A}w\right]=\Tr\left[\mathbf{A}e^{-\beta(C-\blam\cdot\mathbf{A})}\right],\quad\E\left[w^{\top}w\right]=\Tr\left[e^{-\beta(C-\blam\cdot\mathbf{A})}\right],
\]
 and therefore 
\[
\Tr\left[\mathbf{A}X_{\beta,\blam}\right]=\frac{\E\left[w^{\top}\mathbf{A}w\right]}{\E\left[w^{\top}w\right]}.
\]
 Then by drawing many independent Rademacher sample vectors $z_{t,s}$,
$s=1,\ldots,S$, at each optimization step $t=0,\ldots,T-1$ and taking
empirical estimates of the numerator and the denominator, we may form
an empirical estimate of the gradient 
\begin{equation}
g_{t}:=\frac{\sum_{s=1}^{S}w_{t,s}^{\top}\mathbf{A}w_{t,s}}{\sum_{s=1}^{S}w_{t,s}^{\top}w_{t,s}}-b\approx\nabla f_{\beta}(\blam_{t}),\label{eq:gt}
\end{equation}
 where 
\begin{equation}
w_{t,s}:=e^{-\frac{1}{2}\beta(C-\blam_{t}\cdot\mathbf{A})}z_{t,s}.\label{eq:wts}
\end{equation}
 These matrix-vector multiplications can be computed efficiently without
forming the full matrix exponential, using Chebyshev expansion \cite{golub2013matrix}.
For simplicity, we will ignore this source of error, but standard
theory \cite{trefethen2019approximation} guarantees that the error
$\epsilon_{\mathrm{cheb}}$ can be made exponentially small in the
uniform norm by considering a polynomial expansion of order $O\left(\log(n/\epsilon_{\mathrm{cheb}})\sqrt{\Vert\beta(C-\blam_{t}\cdot\mathbf{A})\Vert}\right)$.
If $\Vert\blam_{t}\Vert\leq B$ remains bounded independent of $T$
throughout the optimization procedure, then by Assumption \ref{assump:norm},
the cost of forming each $w_{t,s}$ (with entrywise error bounded
by $\epsilon_{\mathrm{cheb}}$) is bounded by the cost of computing
$O\left(\log(n/\epsilon_{\mathrm{cheb}})\sqrt{\beta(\Vert C\Vert+B)}\right)$
matvecs by $C-\blam_{t}\cdot\mathbf{A}$.

Alternatively, it is useful to view the gradient estimator as follows.
In terms of the stochastic vectors $w_{t,s}$ (\ref{eq:wts}), we
can define 
\begin{equation}
\hat{X}_{t}:=\frac{\sum_{s=1}^{S}w_{t,s}w_{t,s}^{\top}}{\sum_{s=1}^{S}w_{t,s}^{\top}w_{t,s}},\label{eq:Xhatt}
\end{equation}
 which is a stochastic estimate of the matrix $X_{\beta,\blam_{t}}$.
(Note that $\Tr[\hat{X}_{t}]=1$ holds automatically by construction.)
Then our gradient estimate can be viewed as 
\begin{equation}
g_{t}=\Tr[\mathbf{A}\hat{X}_{t}]-\mathbf{b},\label{eq:dualgrad2}
\end{equation}
 by analogy to the formula (\ref{eq:dualgrad}) for the exact gradient
$\nabla f_{\beta}(\blam)$, with the stochastic estimate $\hat{X}_{t}$
in place of $X_{\beta,\blam_{t}}$. For later purposes it is useful
to define as well the unnormalized estimate 
\begin{equation}
\hat{Y}_{t}=\frac{1}{S}\sum_{s=1}^{S}w_{t,s}w_{t,s}^{\top},\label{eq:Yhatt}
\end{equation}
 by analogy to (\ref{eq:Ybetalambda}), so that $\hat{X}_{t}=\hat{Y}_{t}/\Tr[\hat{Y}_{t}].$

We will only make use of a very weak approximation assumption: 

\begin{assump}

\label{assump:gradientestimator}Let $\gamma\geq0$. We assume that
for $t=0,\ldots,T-1$, our gradient estimates $g_{t}$ satisfy
\[
\Vert g_{t}-\nabla f_{\beta}(\blam_{t})\Vert_{*}\leq\gamma.
\]
\end{assump}

We will see how this assumption can be validated in our special SDP
cases of interest.

Observe that in the LP case, we can form the probability vector 
\[
\left[x_{\beta,\blam}\right]_{i}=\frac{e^{-\beta(c_{i}-a_{i}\cdot\blam)}}{\sum_{j=1}^{n}e^{-\beta(c_{j}-a_{j}\cdot\blam)}}
\]
 with linear scaling cost in the size of $x$. Then the gradient 
\[
\nabla f_{\beta}(\blam)=Ax_{\beta,\blam}-\mathbf{b}
\]
 can be computed exactly without any need for stochastic estimation.
Thus in this case we will automatically assume that Assumption \ref{assump:gradientestimator}
holds with $\gamma=0$.

In general, then, we write our optimization step as 
\begin{equation}
\blam_{t+1}=\underset{\blam\in\R^{m}}{\text{argmin}}\left\{ f_{\beta}(\blam_{t})+g_{t}\cdot(\blam-\blam_{t})+\frac{1}{2\eta}\Vert\blam-\blam_{t}\Vert^{2}\right\} ,\quad t=0,\ldots,T-1.\label{eq:optstep}
\end{equation}
 For simplicity, we assume that the initial initial condition is 
\[
\blam_{0}=0.
\]

\subsection{Case studies \label{subsec:updatecases}}

\subsubsection{Max-Cut SDP}

For the Max-Cut SDP, the gradient descent norm is simply $\Vert\,\cdot\,\Vert_{\infty}$.
Therefore, standard treatments (Section 10.9 in \cite{beck2017first})
establish that 
\[
\blam_{t+1}=\blam_{t}-\eta\Vert g_{t}\Vert_{1}\,\mathrm{sign}(g_{t}).
\]
 We offer a derivation in Appendix \ref{app:algo-max-cut}. Moreover,
in this case note that for the implementation of the gradient estimator,
we have $C-\blam_{t}\cdot\mathbf{A}=C-\mathrm{diag}(\blam_{t})$,
and $g_{t}$ is a stochastic estimate of 
\[
\nabla f_{\beta}(\blam_{t})=\mathrm{diag}[X_{\beta,\blam_{t}}]-\mathbf{b}.
\]
 Specifically, we can write 
\[
g_{t}=\mathrm{diag}[\hat{X}_{t}]-\mathbf{b},
\]
 where $\hat{X}_{t}$ is defined as in (\ref{eq:Xhatt}).

We want to guarantee that Assumption \ref{assump:gradientestimator}
holds for all of our gradient estimates $g_{t}$ by taking $S$ sufficiently
large.

To this end, the following lemma, reproduced with suitable notation
from \cite{lindseyshi2025}, is simple consequence of the fact that
the Johnson-Lindenstrauss random projection is a subspace embedding
\cite{Martinsson_Tropp_2020}: 
\begin{lem}
\label{lem:concentration1} Let $\ve,\delta\in(0,1)$ and $t\in\{0,\ldots,T-1\}$.
With notation as in the preceding discussion, there exists a constant
$C>0$ such that if $S\geq\frac{C\,\log(n/\delta)}{\ve^{2}}$, then
with probability at least $1-\delta$, we have that 
\[
(1-\ve)\,\mathrm{diag}[Y_{\beta,\blam_{t}}]\leq\mathrm{diag}[\hat{Y}_{t}]\leq(1+\ve)\,\mathrm{diag}[Y_{\beta,\blam_{t}}].
\]
\end{lem}

As a consequence it follows that Assumption \ref{assump:gradientestimator}
can be guaranteed with high probability by taking $S=\tilde{O}(\gamma^{-2})$:
\begin{cor}
\label{cor:maxcutgradest}Let $\gamma\in(0,2]$ and $\delta\in(0,1)$.
Take 
\[
S\geq\frac{16\,C\,\log(Tn/\delta)}{\gamma^{2}},
\]
where $C$ as in Lemma \ref{lem:concentration1}. Then Assumption
\ref{assump:gradientestimator} holds with probability at least $1-\delta$,
where we choose $\Vert\,\cdot\,\Vert=\Vert\,\cdot\,\Vert_{\infty}$,
hence $\Vert\,\cdot\,\Vert_{*}=\Vert\,\cdot\,\Vert_{1}$.
\end{cor}

The proof is given in Appendix \ref{app:gradest}.

\subsubsection{OT LP}

For the OT LP, our choice of norm $\Vert(\phi,\psi)\Vert=\sqrt{2(\Vert\phi\Vert_{\infty}^{2}+\Vert\psi\Vert_{\infty}^{2})}$
induces the update 
\begin{align*}
\phi_{t+1} & =\phi_{t}-\frac{\eta}{2}\Vert\mu_{t}-\mu\Vert_{1}\,\mathrm{sign}\left(\mu_{t}-\mu\right),\\
\psi_{t+1} & =\psi_{t}-\frac{\eta}{2}\Vert\nu_{t}-\nu\Vert_{1}\,\mathrm{sign}\left(\nu_{t}-\nu\right),
\end{align*}
 where $\mu_{t}=\pi_{t}\mathbf{1}_{n}$ and $\nu_{t}=\pi_{t}^{\top}\mathbf{1}_{m}$
are the current marginal estimates, defined in terms of 
\[
\pi_{t}=\pi_{\beta,\blam_{t}}\propto\left(e^{-\beta(c_{ij}-[\phi_{t}]_{i}-[\psi_{t}]_{j})}\right),
\]
 suitably normalized. The derivation of this exact formula is deferred
to Appendix \ref{app:algo-ot}.

In our implementation, we include an extra step in the algorithm in
which the dual variables $\phi_{t}$ and $\psi_{t}$ are shifted by
a constant for numerical stability. Indeed, note that adding a constant
shift to either dual variable affects neither $\pi_{t}$ nor the dual
objective, so we can always ensure that $\mathbf{1}_{m}^{\top}\phi_{t}=0$
and $\mathbf{1}_{n}^{\top}\psi_{t}=0$. This inclusion leads to the
following algorithm: 
\begin{align}
\tilde{\phi}_{t} & =\phi_{t}-\frac{\eta}{2}\Vert\mu_{t}-\mu\Vert_{1}\,\mathrm{sign}\left(\mu_{t}-\mu\right),\nonumber \\
\tilde{\psi}_{t} & =\psi_{t}-\frac{\eta}{2}\Vert\nu_{t}-\nu\Vert_{1}\,\mathrm{sign}\left(\nu_{t}-\nu\right),\label{eq:otupdate}\\
\phi_{t+1} & =\tilde{\phi}_{t}-\frac{1}{m}\mathbf{1}_{m}^{\top}\tilde{\phi}_{t},\nonumber \\
\psi_{t+1} & =\tilde{\psi}_{t}-\frac{1}{n}\mathbf{1}_{n}^{\top}\tilde{\psi_{t}}.\nonumber 
\end{align}
This convention also helps to control the norm $\|(\phi_{t},\psi_{t})\|$
over the optimization trajectory, as we shall elaborate in Section
\ref{sec:ot-conv}.

\subsubsection{Strong Perm-Synch SDP}

For the strong Perm-Synch SDP, the gradient descent norm is 
\[
\Vert\Blam\Vert=\max_{i=1,\ldots,N}\Vert\Lambda^{(i)}\Vert_{2},
\]
 cf. Section \ref{subsec:normStrongPermSynch}.

One can show that this norm induces the update
\[
\Lambda_{t+1}^{(i)}=\Lambda_{t}^{(i)}-\eta\left(\sum_{j=1}^{N}\Vert G_{t}^{(j)}\Vert_{\Tr}\right)\,\mathrm{sign}[G_{t}^{(i)}],\quad i=1,\ldots,N.
\]
 Here, `$\mathrm{sign}$' denotes the sign function viewed as a matrix
function. Moreover we use the notation $G_{t}^{(i)}$ to denote the
$i$-th block of the gradient estimate 
\[
G_{t}^{(i)}=\hat{X}_{t}^{(i,i)}-\mathbf{I}_{K},
\]
 where $\hat{X}_{t}$ is defined as in (\ref{eq:Xhatt}). A detailed
derivation of the update formula is deferred to Appendix \ref{app:algo-strong-ps}. 

Note that once the vectors $w_{t,s}$ are formed (for $s=1,\ldots,S$)
the cost of forming all of the blocks $G_{t}^{(i)}$ is $O(SK^{2}N)$.
The cost of performing the update is then $O(SK^{2}N+K^{3}N)$ due
to the cubic cost in the block size $K$ of evaluating the matrix
sign function, which requires a diagonalization. Later we will argue
that we can take $S=\tilde{O}(\gamma^{-2}K)$ to satisfy Assumption
\ref{assump:gradientestimator}, so that the total cost is $\tilde{O}(\gamma^{-2}K^{3}N)$,
once the vectors $w_{t,s}$ are formed.

Now we turn to the aforementioned choice of $S$. The following lemma,
reproduced with suitable notation from \cite{lindseyshi2025}, is
(like Lemma \ref{lem:concentration1}) a simple consequence of the
fact that the Johnson-Lindenstrauss random projection is a subspace
embedding \cite{Martinsson_Tropp_2020}: 
\begin{lem}
\label{lem:concentration2} Let $\ve,\delta\in(0,1)$ and $t\in\{0,\ldots,T-1\}$.
With notation as in the preceding discussion, there exists a constant
$C>0$ such that if $S\geq\frac{C\,\log(N/\delta)}{\ve^{2}}K\log K$,
then with probability at least $1-\delta$, we have that 
\[
(1-\ve)\,Y_{\beta,\blam_{t}}^{(i,i)}\preceq\hat{Y}_{t}^{(i,i)}\preceq(1+\ve)\,Y_{\beta,\blam_{t}}^{(i,i)}
\]
 for all $i=1,\ldots,N$.
\end{lem}

As a consequence it follows that Assumption \ref{assump:gradientestimator}
can be guaranteed with high probability by taking $S=\tilde{O}(\gamma^{-2}K)$:
\begin{cor}
\label{cor:strongpermsynchgradest}Let $\gamma\in(0,4]$ and $\delta\in(0,1)$.
Take 
\[
S\geq\frac{64\,C\,\log(TN/\delta)}{\gamma^{2}}K\log K,
\]
where $C$ as in Lemma \ref{lem:concentration2}. Then Assumption
\ref{assump:gradientestimator} holds with probability at least $1-\delta$,
where we choose $\Vert\,\cdot\,\Vert$ following Section \ref{subsec:normStrongPermSynch}.
\end{cor}

The proof is given in Appendix \ref{app:gradest}.

\subsubsection{Weak Perm-Synch SDP}

For the strong Perm-Synch SDP, the gradient descent norm is 
\[
\Vert(\blam,\bmu)\Vert=\sqrt{2\left(\Vert\blam\Vert_{\infty}^{2}+\Vert\bmu\Vert_{\infty}^{2}\right)}.
\]
 cf. Section \ref{subsec:normWeakPermSynch}. 

One can show that this norm induces the following update. First we
compute the gradient estimate, which we separate into two components
$g_{t}\in\R^{n}$ and $h_{t}\in\R^{N}$ (abusing notation slightly
in this context by overloading the general notation $g_{t}$ for the
gradient estimate, just as we have done for $\blam$). Concretely:
\[
g_{t}=\mathrm{diag}[\hat{X}_{t}]-\frac{\mathbf{1}_{n}}{n};\quad h_{t}^{(i)}=\frac{\mathbf{1}_{K}^{\top}\hat{X}_{t}^{(i,i)}\mathbf{1}_{K}}{K}-\frac{1}{n},\ \ i=1,\ldots,N.
\]
 Then, we update 
\[
\blam_{t+1}=\blam_{t}-\frac{\eta}{2}\Vert g_{t}\Vert_{1}\,\mathrm{sign}(g_{t}),
\]
\[
\bmu_{t+1}=\bmu_{t}-\frac{\eta}{2}\Vert h_{t}\Vert_{1}\,\mathrm{sign}(h_{t}).
\]
 Note that the cost of performing this update, once the vectors $w_{t,s}$
are formed (for $s=1,\ldots,S$), is only $O(SNK)=O(Sn)$. A detailed
derivation of the update formula is deferred to Appendix \ref{app:algo-weak-ps}.

In fact we will see that we can take $S=\tilde{O}(\gamma^{-2})$.
To see this, we reproduce the following lemma (with suitable notation)
from \cite{lindseyshi2025}, which is once again a simple consequence
of the fact that the Johnson-Lindenstrauss random projection is a
subspace embedding \cite{Martinsson_Tropp_2020}: 
\begin{lem}
\label{lem:concentration3} Let $\ve,\delta\in(0,1)$ and $t\in\{0,\ldots,T-1\}$.
With notation as in the preceding discussion, there exists a constant
$C>0$ such that if $S\geq\frac{C\,\log(2n/\delta)}{\ve^{2}}$, then
with probability at least $1-\delta$, we have that both 
\[
(1-\ve)\,\mathrm{diag}[Y_{\beta,(\blam_{t},\bmu_{t})}]\leq\mathrm{diag}[\hat{Y}_{t}]\leq(1+\ve)\,\mathrm{diag}[Y_{\beta,(\blam_{t},\bmu_{t})}]
\]
 and 
\[
(1-\ve)\,\mathbf{1}_{K}^{\top}Y_{\beta,(\blam_{t},\bmu_{t})}^{(i,i)}\mathbf{1}_{K}\leq\mathbf{1}_{K}^{\top}\hat{Y}_{t}^{(i,i)}\mathbf{1}_{K}\leq(1+\ve)\,\mathbf{1}_{K}^{\top}Y_{\beta,(\blam_{t},\bmu_{t})}^{(i,i)}\mathbf{1}_{K}
\]
 for all $i=1,\ldots,N$.
\end{lem}

As a consequence it follows that Assumption \ref{assump:gradientestimator}
can be guaranteed with high probability by taking $S=\tilde{O}(\gamma^{-2})$:
\begin{cor}
\label{cor:weakpermsynchgradest}Let $\gamma\in(0,2]$ and $\delta\in(0,1)$.
Take 
\[
S\geq\frac{16\,C\,\log(2Tn/\delta)}{\gamma^{2}},
\]
where $C$ as in Lemma \ref{lem:concentration3}. Then Assumption
\ref{assump:gradientestimator} holds with probability at least $1-\delta$,
where we choose $\Vert\,\cdot\,\Vert$ following Section \ref{subsec:normWeakPermSynch}.
\end{cor}

\section{General convergence theory \label{sec:conv}}

We will prove as much convergence theory as can be tractably achieved
in a general setting. There are two approaches to consider.

The first of these (Section \ref{subsec:grad-cov}) uses strong smoothness
to show that the dual gradients converge to zero. The argument is
quite general, but it does not directly yield an error bound on the
approximation of the optimal value. Since the dual gradient at $\blam$
correspond to the primal feasibility gap for $X_{\beta,\blam}$, the
dual gradient convergence may be combined with a specialized \emph{rounding
}argument which explains how $X_{\beta,\blam}$ can be rounded to
a primal-feasible matrix $X$ whose distance from $X_{\beta,\blam}$
is bounded in terms of the feasibility gap. This overall strategy
was previously used in \cite{altschuler2017near} in the context of
entropically regularized OT. In Sections \ref{sec:max-cut-rounding}
and \ref{sec:perm-synch-rounding} below, we will provide a rounding
argument for the Max-Cut SDP and the strong Perm-Synch SDP, respectively.
A rounding argument for the weak Perm-Synch SDP remains out of our
reach in this work. However, for the OT problem, we will not need
a rounding argument to establish convergence of the objective, as
we now elaborate.

Indeed, the second general argument (Section \ref{subsec:obj-conv})
uses strong smoothness to guarantee a convergence rate for the objective
function, but bounding the preconstant requires an \emph{a priori}
bound over the optimization trajectory of $\Vert\blam_{t}\Vert$.
Such a bound appears to be quite hard to guarantee in general. However,
in the special case of the OT LP, we can provide such a bound, as
we shall show below in Section \ref{sec:ot-conv} (Theorem \ref{thm:otconv}).
The argument here is key for avoiding the rounding argument introduced
by \cite{altschuler2017near}, which allows us to prove a sharp convergence
rate. In turn, we will show in Section \ref{sec:ot-conv} how the
objective convergence result for OT can be \emph{combined} with
the gradient convergence strategy, as well as rounding, to obtain
a good convergence rate equipped with a primal-feasible solution (Theorem
\ref{thm:otcomplexity}).

The two general arguments are presented below in Sections \ref{subsec:grad-cov}
and \ref{subsec:obj-conv} respectively. In the ensuing sections we
will offer the specialized arguments building on the general theory
as indicated in the preceding discussion.

\subsection{Gradient convergence \label{subsec:grad-cov}}

Now we state a result guaranteeing gradient convergence. Below we
always use the notation $e_{\max}(A)$ and $e_{\min}(A)$ to denote
the largest and smallest eigenvalues, respectively, of a symmetric
matrix $A$. We also use the notation $\Delta e(A)=e_{\max}(A)-e_{\min}(A)$.
Note that $\Delta e(A)\leq2\Vert A\Vert_{2}$ always.

\begin{restatable}[Gradient convergence]{restatethm}{gradconvthm}

\label{thm:gradconv}Let $\blam_{t}$, $t=0,\ldots,T$, be furnished
by the algorithm (\ref{eq:optstep}) with gradient estimates satisfying
Assumption \ref{assump:gradientestimator} and step size $\eta=1/\beta$,
for some $\gamma>0$. Then there exists $t\in\{0,\ldots,T-1\}$ such
that 
\[
\Vert\nabla f_{\beta}(\blam_{t})\Vert_{*}\leq3\gamma+2\sqrt{\frac{\beta\,\Delta e(C)}{T}}+2\sqrt{\frac{\log n}{T}}.
\]

\end{restatable}
\begin{rem}
If we want to guarantee that $\Vert\nabla f_{\beta}(\blam_{t})\Vert_{*}=O(\ve)$,
we can take $\gamma=O(\ve)$, $\beta=O(\ve^{-1})$, $T=\tilde{O}(\ve^{-3})$.
Note that for the Max-Cut SDP and this choice of $\gamma$, Lemma
\ref{cor:maxcutgradest} says that Assumption \ref{assump:gradientestimator}
can be guaranteed to hold with high probability by taking $S=\tilde{O}(\ve^{-2})$
samples in our gradient estimator (\ref{eq:gt}).

In the LP setting, $\Delta e(C)$ corresponds to $\max_{i}c_{i}-\min_{i}c_{i}$.
\end{rem}

The proof of Theorem \ref{thm:gradconv} is given in Appendix \ref{app:gradconv}.

\subsection{Objective convergence \label{subsec:obj-conv}}

Below we present two objective-convergence results: one using exact
gradients ($\gamma=0$ in Assumption \ref{assump:gradientestimator})
and allowing for stochastic gradient estimation ($\gamma>0$). The
exact-gradient result (Theorem \ref{thm:objconv}) can be recovered
as a special case of known results---see, for example, Fact B.1 in
\cite{allen2014linear} and Remark 10.72 in \cite{beck2017first}.
In contrast, our stochastic-gradient result (Theorem \ref{thm:objconv-1})
establishes an exponential convergence rate (with rate determined
by $\gamma$) up to a bias (also determined by $\gamma$). Such a
result, to the best of our knowledge, has not appeared previously
in the literature.

\begin{restatable}[GD objective convergence]{restatethm}{objconvthm}

\label{thm:objconv} Let $\blam_{t}$, $t=0,\ldots,T$, be furnished
by the algorithm (\ref{eq:optstep}) with exact gradient estimation
(i.e., satisfying Assumption \ref{assump:gradientestimator} with
$\gamma=0$) and step size $\eta\le1/\beta$. We assume that 
\begin{equation}
D\coloneqq\sup_{t=0,1,2,\ldots}\left\{ \|\blam_{t}-\blam_{\star}\|^{2}\right\} <+\infty\label{eq:D}
\end{equation}
 where $\blam_{\star}$ is a minimizer of $f_{\beta}$. Then for any
$t\ge1,$
\[
f_{\beta}(\blam_{t})-f_{\beta}(\blam_{\star})\le\frac{2D}{t\eta}.
\]

\end{restatable}
\begin{rem}
Here the constant $D$ is difficult to bound in general. It is well-known
(cf. Theorem 3.4 of \cite{garrigos2023handbook}) that if the norm
is Euclidean norm $\|\cdot\|_{2}$, then the descent lemma implies
that $\|\blam_{t+1}-\blam_{\star}\|_{2}\le\|\blam_{t}-\blam_{\star}\|_{2}$,
so that $D=\Vert\blam_{\star}\Vert_{2}^{2}$, since we take $\blam_{0}=0$.
However, the descent lemma does not hold for general norms. Besides,
for the case of $\|\cdot\|_{\infty}$, \cite{sidford2018coordinate}
showed that in a general setting, $D$ can scale with $n$. In many
practical settings, however, $D$ appears to be dimension-independent,
as we confirm in Section \ref{sec:ot-conv} in the case of optimal
transport.
\end{rem}

The proof of Theorem \ref{thm:objconv} is given in Appendix \ref{app:lossconv}.

\begin{restatable}[SGD objective convergence]{restatethm}{noiseobjconvthm}

\label{thm:objconv-1} Let $\blam_{t}$, $t=0,\ldots,T$, be furnished
by the algorithm (\ref{eq:optstep}) with gradient estimates satisfying
Assumption \ref{assump:gradientestimator} and step size $\eta\le1/(2\beta)$.
Further assume that $D$ as defined in (\ref{eq:D}) is bounded and
that the gradient error satisfies $\gamma\le\sqrt{3\beta^{2}D/2}$.
Then for any $t\ge1,$
\[
f_{\beta}(\blam_{t})-f_{\beta}(\blam_{\star})\le\bigg(1-\frac{2\gamma}{\sqrt{2(\eta^{-1}-\beta)(2\eta^{-1}-\beta)D}}\bigg)^{t}\big[f_{\beta}(\blam_{0})-f_{\beta}(\blam_{\star})\big]+\frac{\gamma\sqrt{2(2\eta^{-1}-\beta)D}}{\sqrt{\eta^{-1}-\beta}}.
\]
Specifically, if we set $\eta=1/(2\beta)$, we have
\[
f_{\beta}(\blam_{t})-f_{\beta}(\blam_{\star})\le\bigg(1-\frac{2\gamma}{\sqrt{6\beta^{2}D}}\bigg)^{t}\big[f_{\beta}(\blam_{0})-f_{\beta}(\blam_{\star})\big]+\gamma\sqrt{6D}.
\]

\end{restatable}
\begin{rem}
Suppose that we want to reach an additive error tolerance of $\epsilon$.
The theorem establishes an exponential rate of convergence up to a
bias determined by $\gamma$. Since the bias term is $\gamma\sqrt{6D}$,
we can take $\gamma=\Omega(\epsilon/\sqrt{D})$ to reach the desired
accuracy. Meanwhile the exponential decay factor is 
\[
\left[1-\Omega\left(\frac{\epsilon}{\beta D}\right)\right]^{t}\leq e^{-\Omega\left(\frac{t\epsilon}{\beta D}\right)},
\]
 so in turn we can take $t=O\left(\beta D\frac{\log(1/\ve)}{\ve}\right)$
to ensure that the decay factor is $O(\ve)$. By Lemma \ref{lem:guessobj},
moreover, $f_{\beta}(\blam_{0})-f_{\beta}(\blam_{\star})=O(1+\beta^{-1}\log n)$,
assuming that the spectral range of the cost matrix is $O(1)$. Therefore,
we require $T=\tilde{O}(\beta D\ve^{-1})$ iterations with $\gamma=O(\epsilon/\sqrt{D})$.
Note that up to logarithmic factors, we recover the complexity of
$T=O(\beta D\ve^{-1})$ iterations implied by Theorem \ref{thm:objconv}
in the exact-gradient case.

Typically, if the gradients are estimated with stochastic trace estimation,
we would require $\tilde{O}(D\ve^{-2})$ shots of the trace estimator
to ensure that Assumption \ref{assump:gradientestimator} holds for
such $\gamma$ with high probability, yielding $\tilde{O}(\beta D^{2}\epsilon^{-3})$
shots in total. Therefore, in the $\beta=\Omega(\ve^{-1})$ regime,
which by Lemma \ref{lem:reg-bound} ensures $\tilde{O}(\ve)$-approximation
of the unregularized problem, the overall complexity is therefore
$\tilde{O}(D^{2}\ve^{-4})$ shots, which scales nearly independently
of dimension assuming that $D$ is $\tilde{O}(1)$.

We refer the reader to Theorem 5.4 of \cite{garrigos2023handbook}
for a typical analysis of SGD with respect to the Euclidean norm.
Although the setting is somewhat different (and in particular assumes
unbiased gradient estimation), it is interesting to compare with our
result in the analogous context in which the Euclidean norm everywhere
takes the place of our general norm $\Vert\,\cdot\,\Vert$, the objective
function is $\beta$-smooth, the gradient estimator has variance $\gamma$,
and the step size is taken at least small enough such that $\eta\leq1/(4\beta)$.
In this setting, Theorem 5.4 of \cite{garrigos2023handbook} implies
that the expected objective error is $O\left(\frac{D^{2}}{\eta t}+\eta\gamma\right)$.
Similar to our result, there is a decaying term and a bias term in
the error estimate. However, the bias term (unlike ours) is modulated
by the step size, while the decaying term decays only algebraically
with $t$.

Taking $\eta=1/(4\beta)$, for simplicity, to ensure convergence,
the expected error of the objective is $O\left(\frac{\beta D^{2}}{t}+\frac{\gamma}{2\beta}\right)$,
we see that we require $T=O\left(\beta D^{2}\ve^{-1}\right)$ iterations
to ensure that the first term is $O(\ve)$. In the $\beta=\Omega(\ve^{-1})$
regime, we see that we only need to ensure $\gamma=O(1)$ to control
the second. This seems only to require a nearly-constant number of
shots in the trace estimator, but due to the use of the Euclidean
norm, we would in fact require a number that grows with the dimension
$n$, in our settings of interest. Note also that in this hypothetical
context, the meaning of $D$ is different since it is defined with
respect to the Euclidean norm. In our problems of interest, such a
`Euclidean $D$' is dimension-dependent.
\end{rem}

The proof of Theorem \ref{thm:objconv-1} is given in Appendix \ref{app:lossconv}.

\section{Convergence for OT \label{sec:ot-conv}}

Now we present the objective convergence result for the OT LP. See
the discussion at the beginning of Section \ref{sec:conv} for further
context.

\begin{restatable}[OT objective convergence]{restatethm}{otconvthm}

\label{thm:otconv} Let $\blam_{t}=(\phi_{t},\psi_{t})$, $t=0,\ldots,T$,
be furnished by the algorithm (\ref{eq:otupdate}) with step size
$\eta\le1/\beta$. Let $M,s>0$ such that $\min_{i\in[m]}\mu_{i}\geq s$,
$\min_{j\in[n]}\nu_{j}\geq s$ and $\max_{i,j}|c_{ij}|\leq M$. Then
for all $t=1,\ldots,T$: 
\[
f_{\beta}(\blam_{t})-f_{\beta}(\blam_{\star})\le\frac{32\left[2M+\frac{\log(1/s)+1}{\beta}\right]^{2}}{t\eta}.
\]

\end{restatable}
\begin{rem}
To the best of our knowledge, this result is the first dimension-independent
bound on convergence of the objective for the entropically regularized
OT problem for general $\beta$ (i.e., not in the large-$\beta$ regime).
Previous results \cite{altschuler2017near,dvurechensky2018computational,lin2022efficiency}
focus on obtaining a dual gradient bound, which in turn yields an
objective bound via a rounding argument in the large-$\beta$ regime.
\end{rem}

The proof of Theorem \ref{thm:otconv} is given in Appendix \ref{app:pf-ot}.

Based on the objective bound, we can also establish the following
gradient bound. 

\begin{restatable}[OT gradient convergence]{restatethm}{otgradthm}

\label{thm:otgrad} With notation and assumptions as in Theorem \ref{thm:otconv},
for any integer $T\ge2$ there exists $t\in\{0,\ldots,T\}$ such that
\[
\|\nabla f_{\beta}(\blam_{t})\|_{\star}\le\frac{16\left[2M+\frac{\log(1/s)+1}{\beta}\right]}{(T-1)\eta}.
\]

\end{restatable}
\begin{rem}
The proof is motivated by the switch argument in \cite{dvurechensky2018computational,lin2022efficiency}.
In short, we split the $T$ iterations into two chunks. In the first
chunk we apply the objective bound in Theorem \ref{thm:otconv}. For
the second chunk, we use the same trick as in the proof of Theorem
\ref{thm:gradconv}.
\end{rem}

The proof of Theorem \ref{thm:otgrad} is also given in Appendix \ref{app:pf-ot}.

Based on the gradient bound, we propose the following scheme to get
a primal feasible solution:
\begin{itemize}
\item Run the algorithm (\ref{eq:otupdate}) for $T$ steps to obtain $\tilde{\pi}:=\pi_{\beta,\blam_{\tilde{t}}}$
where $\tilde{t}=\arg\min_{t=1,\ldots,T}\|\nabla f_{\beta}(\blam_{t})\|_{*}$.
\item Round \textbf{$\tilde{\pi}$} to a primal-feasible solution $\hat{\pi}$
of (\ref{eq:ot_lp}) using Algorithm 2 of \cite{altschuler2017near}.
\end{itemize}
Then we output $\hat{\pi}$ as our final result. For this scheme,
we have the following Theorem:

\begin{restatable}[OT complexity bound]{restatethm}{otcomthm}

\label{thm:otcomplexity} Let $p_{\star}$ denote the optimal value
of the unregularized OT LP (\ref{eq:ot_lp}). Adopt the same assumptions
as in Theorem \ref{thm:otconv}, and in turn let $\hat{\pi}$ be furnished
by the preceding discussion. Then 
\[
\langle c,\hat{\pi}\rangle\le p_{\star}+\frac{128M\left[2M+\frac{\log(1/s)+1}{\beta}\right]}{(T-1)\eta}+\beta^{-1}\log(mn).
\]
 In particular, to achieve 
\[
\langle c,\hat{\pi}\rangle\le p_{\star}+\ve
\]
 for a given target $\ve>0$, we can take $\beta=2\ve^{-1}\log(mn)$,
$\eta=1/\beta$, and 
\[
T=O\left((M/\ve)^{2}\log(mn)+(M/\ve)[1+\log(1/s)]\right).
\]
 In the regime $\ve=O\left(\frac{M\log(mn)}{\log(1/s)+1}\right)$,
ruling out large values of $\ve$, the expression simplifies to 
\[
T=O\left((M/\ve)^{2}\log(mn)\right),
\]
 and the total computational complexity is 
\[
O\left((M/\ve)^{2}mn\log(mn)\right).
\]

\end{restatable}
\begin{rem}
This complexity bound matches the best known bound for the celebrated
Sinkhorn scaling algorithm \cite{lin2022efficiency,dvurechensky2018computational}.
Compared to the Sinkhorn algorithm, our algorithm can straightforwardly
adapt to extra constraints in the primal optimization problem.

The proof of Theorem \ref{thm:otcomplexity} is also given in Appendix
\ref{app:pf-ot}.
\end{rem}

\section{Rounding for Max-Cut SDP \label{sec:max-cut-rounding}}

Now we give the rounding argument for the Max-Cut SDP (\ref{eq:maxcut_sdp}).
Combined with the gradient convergence result (Theorem \ref{thm:gradconv}),
this yields a complete guarantee for approximating the optimal value
of (\ref{eq:maxcut_sdp}). In fact the rounding argument yields a
primal-feasible approximate solution, which is important for producing
an upper bound to the Max-Cut problem, using the celebrated randomized
rounding procedure of \cite{goemans1995improved}. (The latter notion
of rounding is different.)

Throughout we let $\Vert A\Vert_{\infty}=\max_{j}\sum_{i}\vert A_{ij}\vert$
denote the usual induced $\ell^{\infty}$ operator norm. Note that
$\Vert A\Vert_{2}\leq\Vert A\Vert_{\infty}$ always holds, so assuming
that $\Vert C\Vert_{\infty}=O(1)$ is a slightly stronger assumption
than assuming $\Vert C\Vert_{2}=O(1)$. However, when the Max-Cut
SDP is applied to the Max-Cut problem on graphs with bounded (weighted)
degree, we naturally have $\Vert C\Vert_{\infty}=O(1)$. In fact,
the weights can even be signed, as in the case of spin glass optimization.

The key lemma, restated and proved in Appendix \ref{app:maxcutrounding},
is the following:

\begin{restatable}[Max-Cut rounding]{restatelem}{maxcutroundinglem}

\label{lem:maxcutround} Suppose that $X\succeq0$ with $\Vert\mathrm{diag}(X)-\mathbf{1}_{n}\Vert_{1}\leq\delta$.
Then there exists $X'\succeq0$ with $\mathrm{diag}(X')=\mathbf{1}_{n}$
such that for any symmetric $A$, we have 
\[
\Tr[A(X-X')]\leq3\delta\Vert A\Vert_{\infty}.
\]

\end{restatable}

This lemma can be used to prove the following main theorem for the
Max-Cut SDP. The proof is also supplied in Appendix \ref{app:maxcutrounding}.

\begin{restatable}[Max-Cut rounding]{restatethm}{maxcutroundingthm}

\label{thm:maxcut} Let $p_{\star}$ denote the optimal value of the
unregularized Max-Cut SDP (\ref{eq:maxcut_sdp}), and let $\kappa=\frac{\max_{i}b_{i}}{\min_{i}b_{i}}$.
Suppose we have $\blam\in\R^{m}$ such that $X_{\beta,\blam}$ as
defined by (\ref{eq:Xmaxcut}) satisfies $\Vert\mathrm{diag}(X_{\beta,\blam})-\mathbf{b}\Vert_{1}\leq\delta$.
Then 
\[
\left|\Tr[CX_{\beta,\blam}]-p_{\star}\right|\leq3\kappa\delta\Vert C\Vert_{\infty}+\beta^{-1}\log n.
\]
 Moreover we can efficiently construct $X'$ which is feasible for
(\ref{eq:maxcut_sdp}) and satisfies 
\[
\Tr[CX']-p_{\star}\leq6\kappa\delta\Vert C\Vert_{\infty}+\beta^{-1}\log n.
\]

\end{restatable}
\begin{rem}
Let $\ve>0$ be a target for the additive error of our estimate of
the optimal value. Now Theorem \ref{thm:gradconv} implies that $\blam$
as in the statement of Theorem \ref{thm:maxcut} with $\delta=O(\ve)$
can be obtained within the first $T$ iterations of algorithm (\ref{eq:optstep})
with gradient estimates satisfying Assumption \ref{assump:gradientestimator},
as long as we take $\beta=O(\ve^{-1})$, $\gamma=O(\ve)$, and $T=\tilde{O}(\ve^{-3})$.
To guarantee that Assumption \ref{assump:gradientestimator} is satisfied
with high probability for $\gamma=O(\ve)$, Lemma \ref{cor:maxcutgradest}
says that we need only take $S=\tilde{O}(\ve^{-2})$ samples per iteration
within the gradient estimator (\ref{eq:gt}). In summary, assuming
that $\kappa=\tilde{O}(1)$ and $\Vert C\Vert_{\infty}=\tilde{O}(1)$,
we achieve our target error $\ve$ within $\tilde{O}(\ve^{-3})$ iterations,
each of which requires $\tilde{O}(\ve^{-2})$ samples in the gradient
estimator. Notably, the overall dimension-dependence is only logarithmic.
\end{rem}

\section{Rounding for Strong Perm-Synch SDP \label{sec:perm-synch-rounding}}

Next we give the rounding argument for the strong Perm-Synch SDP (\ref{eq:strongpermsynch_sdp}).
The gradient convergence result, Theorem \ref{thm:gradconv}, allows
us to construct a PSD matrix $X$ of size $NK\times NK$ satisfying
\[
\sum_{i=1}^{N}\Vert X^{(i,i)}-\mathbf{I}_{K}/n\Vert_{\Tr}\leq\delta
\]
 for arbitrary $\delta>0$, i.e., satisfying approximate primal feasibility.
We want to construct $X'$ which is close to $X$ and primal-feasible.
The following key lemma is proved in Appendix \ref{app:strongpermsynchrounding}.

\begin{restatable}[Strong Perm-Synch rounding]{restatelem}{strongpermsynchroundinglem}

\label{lem:strongpermsynchround} Suppose that $X\succeq0$ with $\sum_{i=1}^{N}\Vert X^{(i,i)}-\mathbf{I}_{K}\Vert_{\Tr}\leq\delta$.
Then there exists $X'\succeq0$ with $[X']^{(i,i)}=\mathbf{I}_{K}$
for all $i=1,\ldots,N$, such that for any symmetric $A$, we have
\[
\Tr[A(X-X')]\leq(2K+1)\,\delta\,\vert\vert\vert A\vert\vert\vert,
\]
 where 
\[
\vert\vert\vert A\vert\vert\vert:=\max_{i=1,\ldots,N}\sum_{j=1}^{N}\Vert A^{(i,j)}\Vert_{2}.
\]

\end{restatable}

This lemma in turn enables us to prove the main theorem for the Strong
Perm-Synch SDP. The proof is also deferred to Appendix \ref{app:strongpermsynchrounding}.

\begin{restatable}[Strong Perm-Synch rounding]{restatethm}{strongpermsynchroundingthm}

\label{thm:strongpermsynch} Let $p_{\star}$ denote the optimal value
of the unregularized Strong Perm-Synch SDP (\ref{eq:strongpermsynch_sdp}).
Suppose we have $\Blam=(\Lambda^{(1)},\ldots,\Lambda^{(N)})$ such
that $X_{\beta,\Blam}$ as defined by (\ref{eq:Xstrong}) satisfies
$\sum_{i=1}^{N}\Vert X_{\beta,\Blam}^{(i,i)}-\mathbf{I}_{K}/n\Vert_{\Tr}\leq\delta$.
Then 
\[
\left|\Tr[CX_{\beta,\Blam}]-p_{\star}\right|\leq(2K+1)\delta\,\vert\vert\vert C\vert\vert\vert+\beta^{-1}\log n.
\]
 Moreover we can efficiently construct $X'$ which is feasible for
(\ref{eq:strongpermsynch_sdp}) and satisfies 
\[
\Tr[CX']-p_{\star}\leq(4K+2)\delta\,\vert\vert\vert C\vert\vert\vert+\beta^{-1}\log n.
\]

\end{restatable}
\begin{rem}
First we discuss the norm $\vert\vert\vert C\vert\vert\vert$. In
the permutation synchronization context outlined in Section \ref{subsec:cases}
(cf. \cite{lindseyshi2025} for further details), the blocks $C^{(i,j)}$
should be viewed as submatrices of permutation matrices, so in particular
$\Vert C^{(i,j)}\Vert_{2}\leq1$ for all $i,j=1,\ldots,N$. Moreover,
some blocks might satisfy $C^{(i,j)}=0$ if we do not have any tentative
correspondences between keypoints in images $i,j$. We can view our
images as defining a graph, in which $(i,j)$ is an edge if and only
if $C^{(i,j)}\neq0$. Then we can roughly estimate $\vert\vert\vert C\vert\vert\vert\leq d$,
where $d$ is the maximum degree of this graph. Provided $d=O(1)$,
then $\vert\vert\vert C\vert\vert\vert=O(1)$ as well, analogously
to the situation for the Max-Cut SDP.

Let $K\ve>0$ be a target for the additive error of our estimate of
the optimal value. Now provided that $\vert\vert\vert C\vert\vert\vert=O(1)$,
Theorem \ref{thm:gradconv} implies that $\Blam$ as in the statement
of Theorem \ref{thm:maxcut} with $\delta=O(\ve)$ can be obtained
within the first $T$ iterations of algorithm (\ref{eq:optstep})
with gradient estimates satisfying Assumption \ref{assump:gradientestimator},
as long as we take $\beta=O(\ve^{-1})$, $\gamma=O(\ve)$, and $T=\tilde{O}(\ve^{-3})$.
To guarantee that Assumption \ref{assump:gradientestimator} is satisfied
with high probability for $\gamma=O(\ve)$, Lemma \ref{cor:strongpermsynchgradest}
says that we need only take $S=\tilde{O}(\ve^{-2})$ samples per iteration
within the gradient estimator (\ref{eq:gt}).

Note the somewhat annoying prefactor of $K$ appearing in the error
tolerance. If $K$ (the number of keypoints per image) is viewed as
constant while $N$ (the number of images) is taken to be large, this
prefactor does not affect the asymptotic scaling of the algorithm.
(Indeed, the Strong Perm-Synch SDP rapidly becomes expensive to implement
for large $K$ anyway due to the cost of diagonalizing $K\times K$
matrices.) Still, the prefactor of $K$ arising from the proof of
Lemma \ref{lem:strongpermsynchround} may be overly pessimistic in
practice. Note that in this proof, if the blocks $A^{(i,j)}$ are
submatrices of permutation matrices and the singular vector blocks
$V^{(j)}$ are permutation matrices (as should be the case in exact
permutation recovery), then the factor of $K$ appearing in (\ref{eq:annoying})
could be omitted.
\end{rem}

\section{Numerical experiments \label{sec:exp}}

We test our algorithms on all of the problems of interest described
above. In all cases we fix the step size in terms of the regularization
parameter as $\eta=1/\beta$. The code is available online at 
\[
\texttt{https://github.com/willcai7/NonEuclidean-GD}
\]

\subsection{Max-Cut SDP}

We validate our theory of Section \ref{sec:conv} concerning gradient
convergence, and we will not concern ourselves with downstream processing
of the SDP solution, which can be used as in \cite{lindsey2023fast}
to produce an upper bound for the Max-Cut problem. We use the algorithm
and code of \cite{expmv} to perform the required exponential matvecs.

We consider the Max-Cut problem on Erd\H{o}s-Renyi random graphs
of size $n$ with probability $p=3/n$ of including an edge. Thus
the expected degree of each vertex is 3. For such a model, the rank
of the optimal solution $X$ of (\ref{eq:maxcut_sdp}) grows with
the problem size, meaning that low-rank approaches to SDP cannot achieve
$\tilde{O}(n)$ scaling.

We consider $n=200,400,600,800,1000$ and plot in Figure \ref{fig:maxcut_feas}
the primal feasibility error 
\begin{equation}
\Vert\mathrm{diag}(\hat{X}_{t})-\mathbf{1}_{n}/n\Vert_{1}\label{eq:maxcut_feas}
\end{equation}
 of the estimate $\hat{X}_{t}$ (\ref{eq:Xhatt}) for $t=1,\ldots,200$,
which can also be viewed as the dual norm of the dual gradient estimate.
We consider two possible choices for the number of stochastic vectors:
$S=\lceil25\times\log n\rceil$ and $S=\lceil100\times\log n\rceil$.
Accordingly we expect the feasibility error to saturate at roughly
half as large of a value for the latter choice. This claim is borne
out empirically, as is the dimension-independent convergence profile
of the feasibility error. We also consider two choices for the inverse
temperature: $\beta=10$ and $\beta=20$. We expect that the latter
choice should require about twice as many iterations to achieve the
same error, and this claim is also confirmed numerically.

\begin{figure}[H]
\begin{centering}
\includegraphics[scale=0.4]{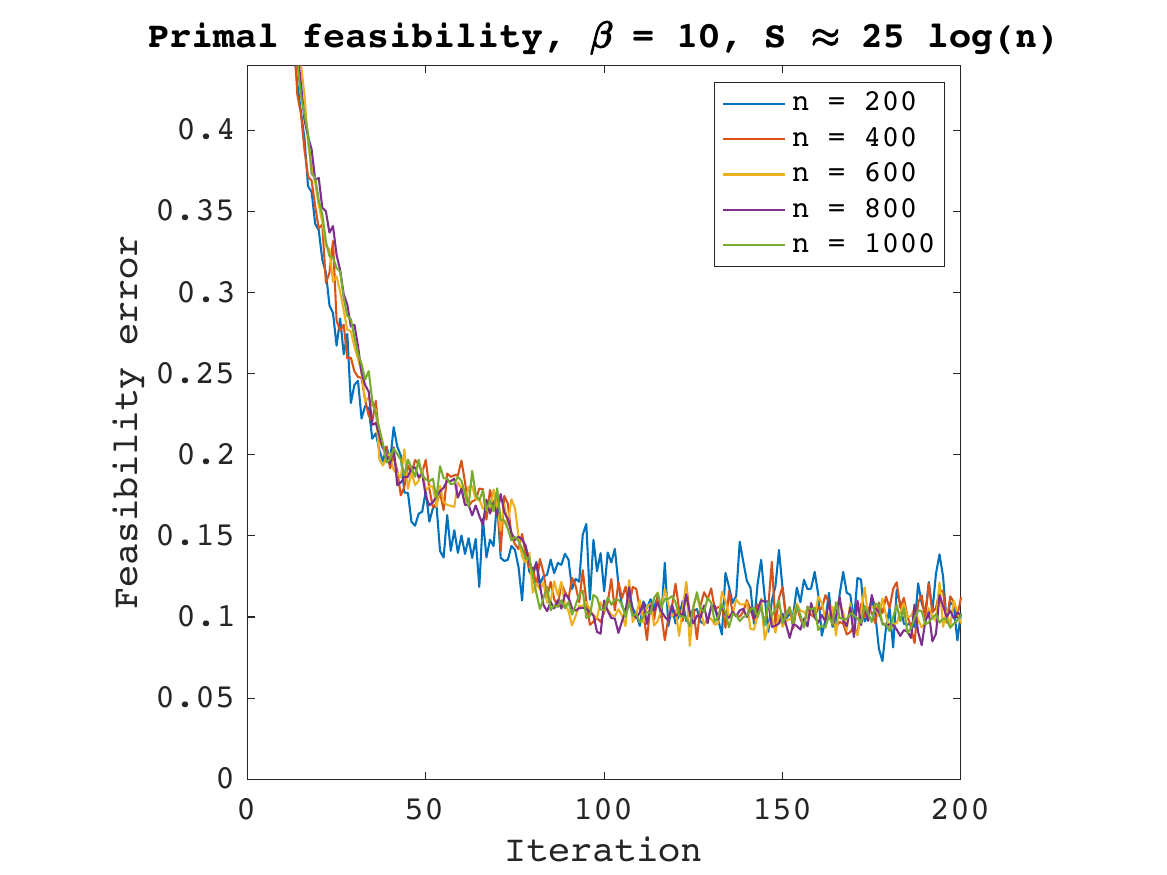}\includegraphics[scale=0.4]{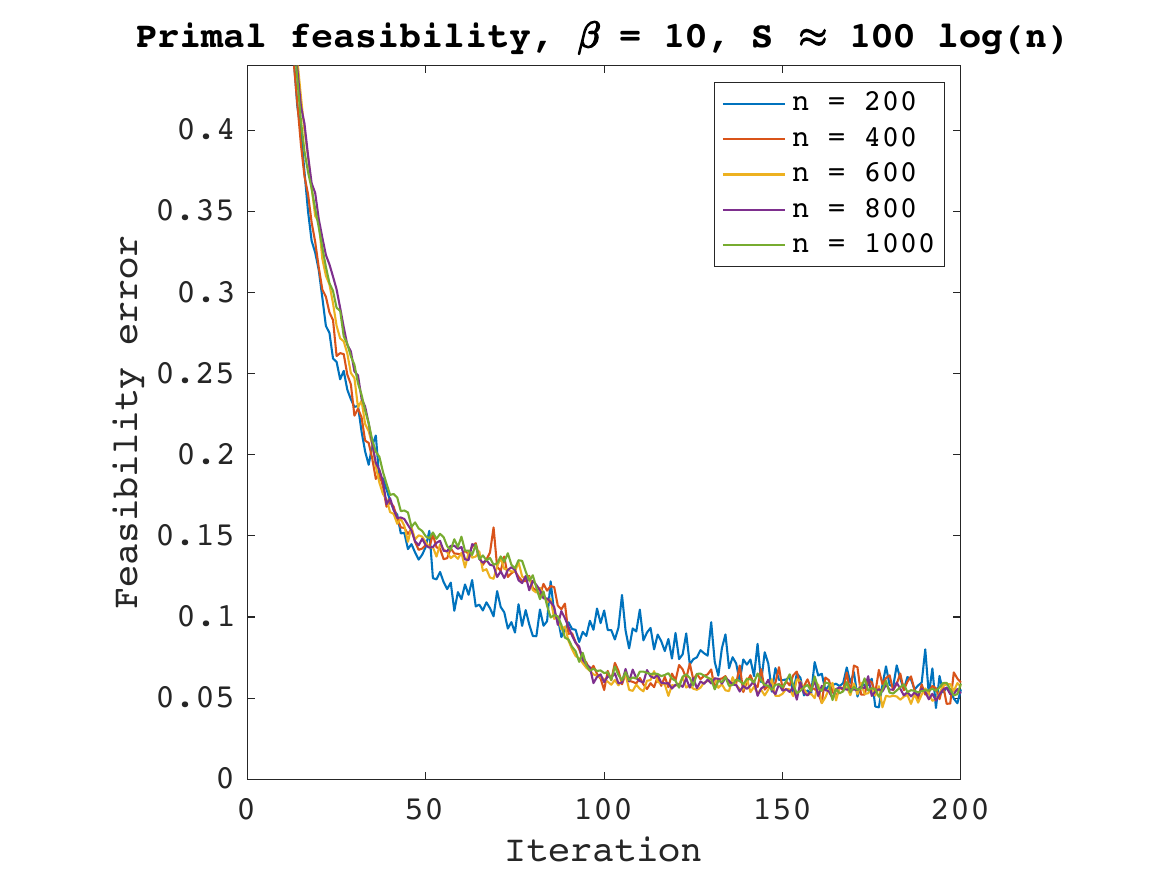}\vspace{3mm}
\par\end{centering}
\begin{centering}
\includegraphics[scale=0.4]{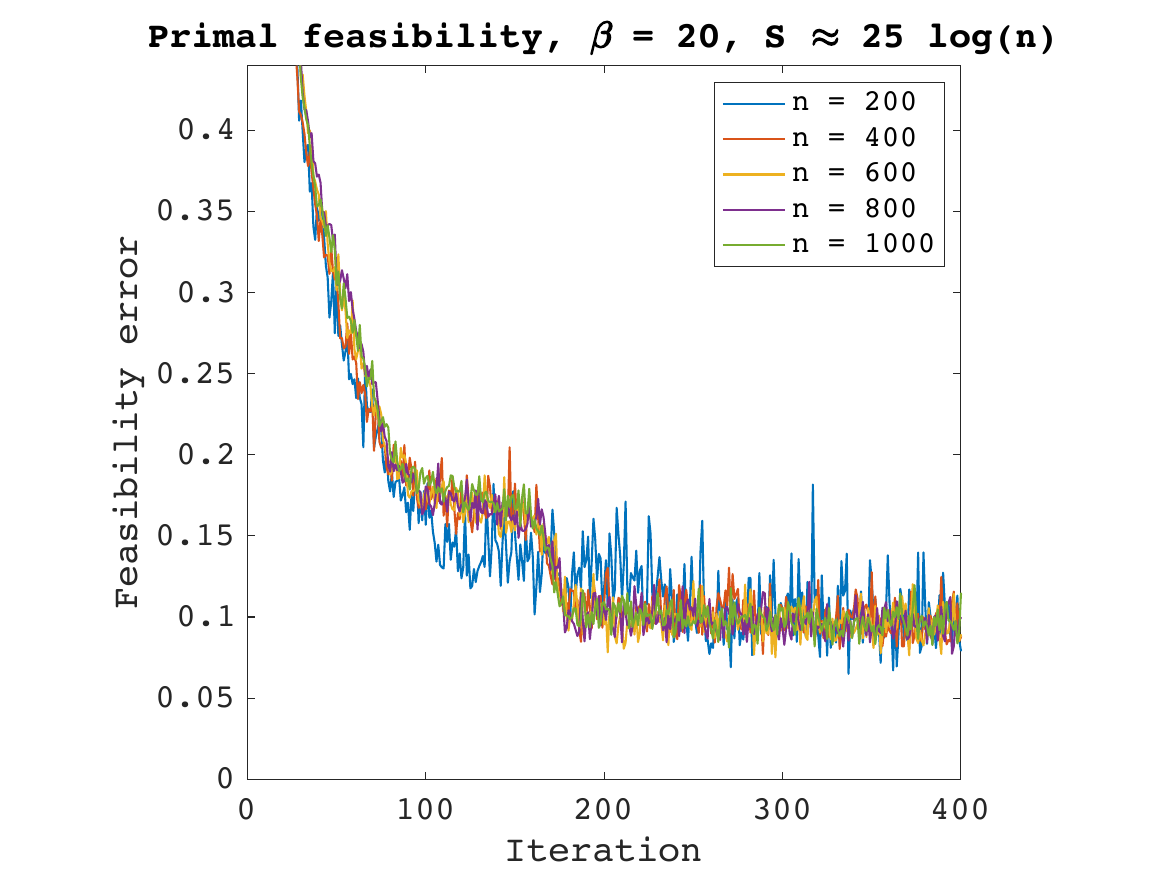}\includegraphics[scale=0.4]{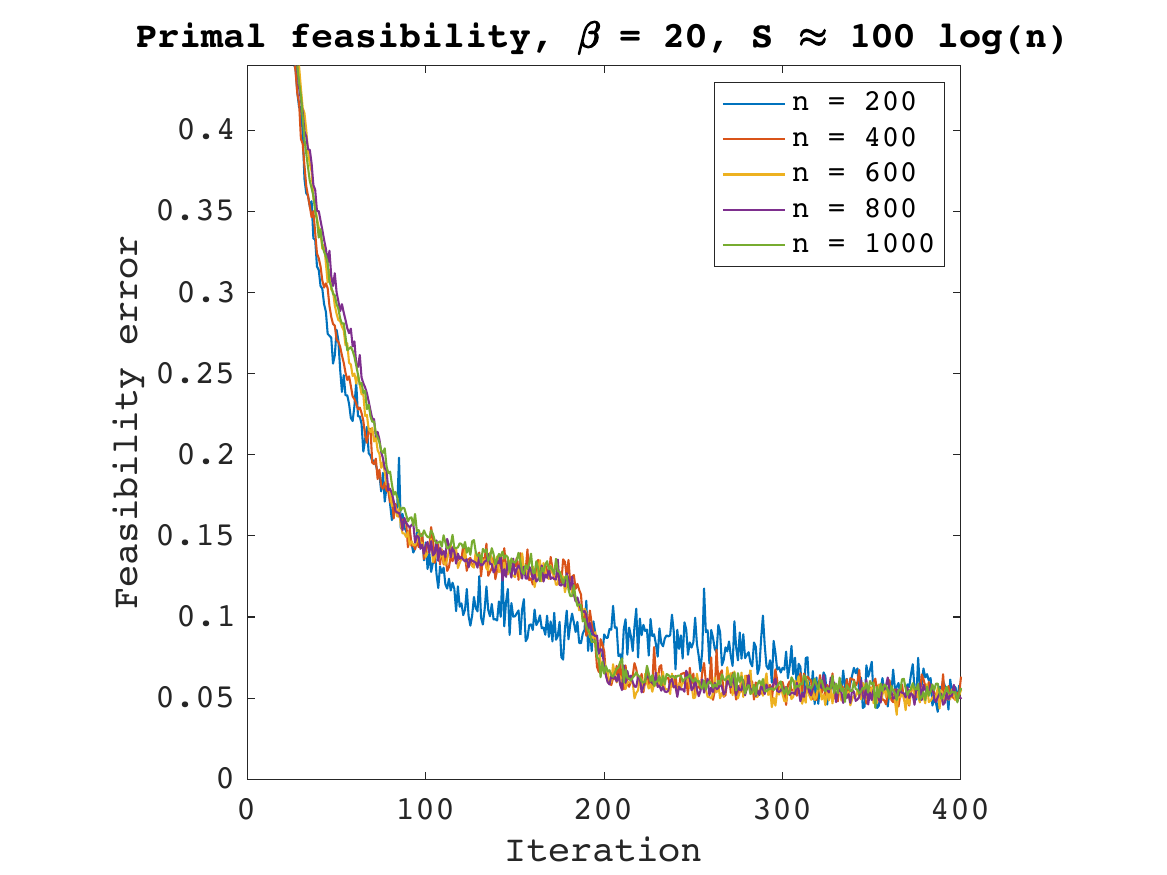}
\par\end{centering}
\caption{Convergence profile of the primal feasibility error (\ref{eq:maxcut_feas})
for the Max-Cut SDP. Experiments are described in the main text. In
the top and bottom rows we consider $\beta=10$ and $\beta=20$, respectively.
(Note the difference in the horizontal axis scale.) In the left and
right columns we consider $S=\lceil25\times\log n\rceil$ and $S=\lceil100\times\log n\rceil$,
respectively. \label{fig:maxcut_feas}}
\end{figure}

\subsection{OT LP}

Following the experimental setting in \cite{li2024fastcomputationoptimaltransport},
we consider two test cases for OT.
\begin{itemize}
\item \textbf{Synthetic} \textbf{dataset.} For an integer $k,$ we produce
two $k\times k$ images as follows. Each image has a randomly placed
square foreground that accounts for 50\% of the pixels; the foreground
and background have pixel values uniformly sampled from $[0,10]$
and $[0,1]$, respectively. We flatten the two images and normalize
them to obtain $\mu,\nu\in\R^{n}$ where $n=k^{2}.$ In addition,
we construct the cost matrix $c\in\R^{n\times n}$ using the $\ell_{2}$
distance between pairs of points in each $k\times k$ image. We normalize
the cost matrix such that the maximum value is $10$.
\item \textbf{MNIST dataset.} First, a pair of images is randomly selected
from the MNIST dataset and downsampled to size $k\times k$. Then,
a value of $0.01$ is added to all pixels. The remaining steps for
obtaining $\mu$, $\nu$, and the cost matrix are the same as in the
previous setting.
\end{itemize}
For each test case, we compare our algorithm \ref{eq:otupdate} with
the Sinkhorn scaling algorithm in two sets of experiments. In the
first set, we fix the dimension $n=784$, which is the original size
of MNIST dataset samples, and consider different inverse temperatures,
$\beta=5,10,20$. For the second set, we fix the inverse temperature
$\beta=10$ and consider different sizes, $n=100,400,900$. We measure
convergence using the dual objective error 
\begin{equation}
f_{\beta}(\phi_{t},\psi_{t})-f_{\beta}(\phi_{\star},\psi_{\star})\label{eq:ot-obj-error}
\end{equation}
 and the primal feasibility error 
\begin{equation}
\|\mu_{t}-\mu\|_{1}+\|\nu_{t}-\nu\|_{1}.\label{eq:ot-feas}
\end{equation}
 For each set of experiments, we consider $20$ pairs of images and
run for $50,000$ iterations. We report averaged convergence metrics
in Fig \ref{fig:ot-beta} and \ref{fig:ot-n} for the respective sets
of experiments. The results confirm our theoretical findings that
the objective error and primal feasibility decay as $\tilde{O}(\beta/t)$,
nearly independently of the dimension.

\begin{figure}
\begin{centering}
\includegraphics[bb=0bp 0bp 440bp 420bp,scale=0.4]{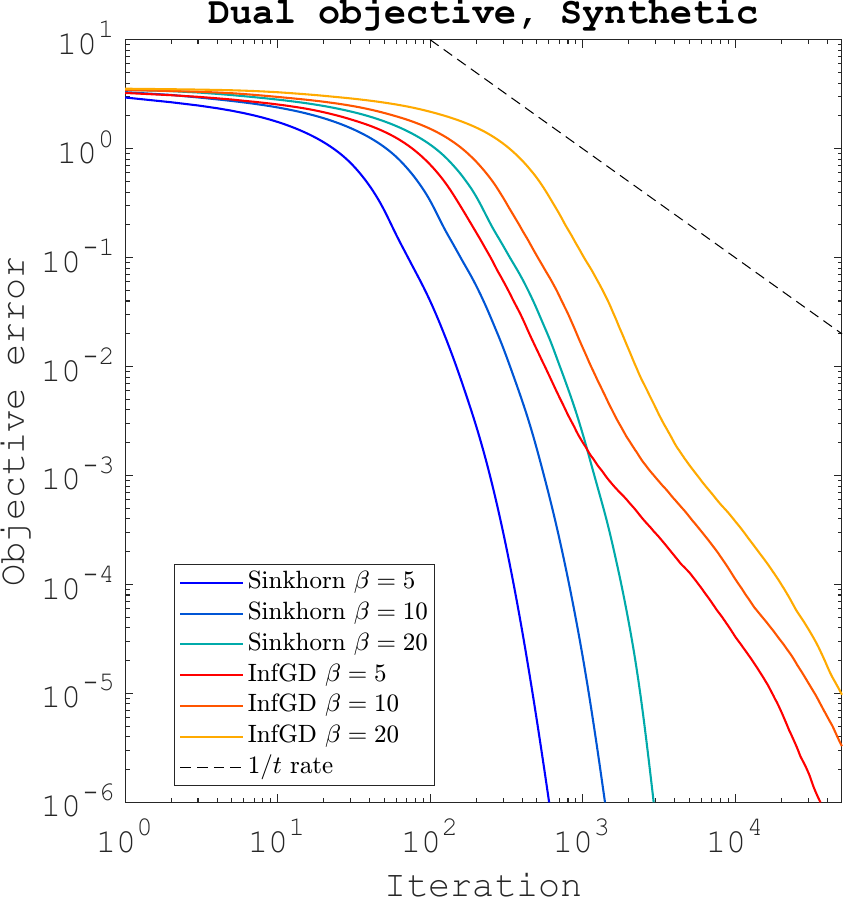}
$\qquad$ \includegraphics[bb=0bp 0bp 440bp 440bp,clip,scale=0.4]{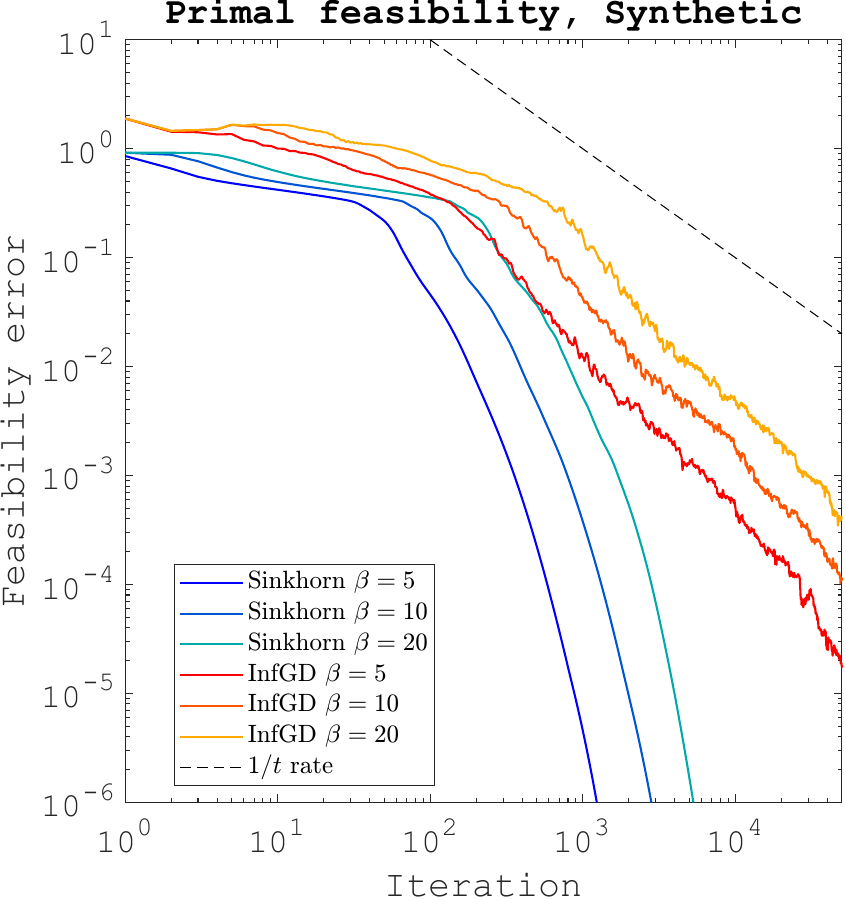}\vspace{5mm}
\par\end{centering}
\begin{centering}
\includegraphics[bb=0bp 0bp 440bp 440bp,scale=0.4]{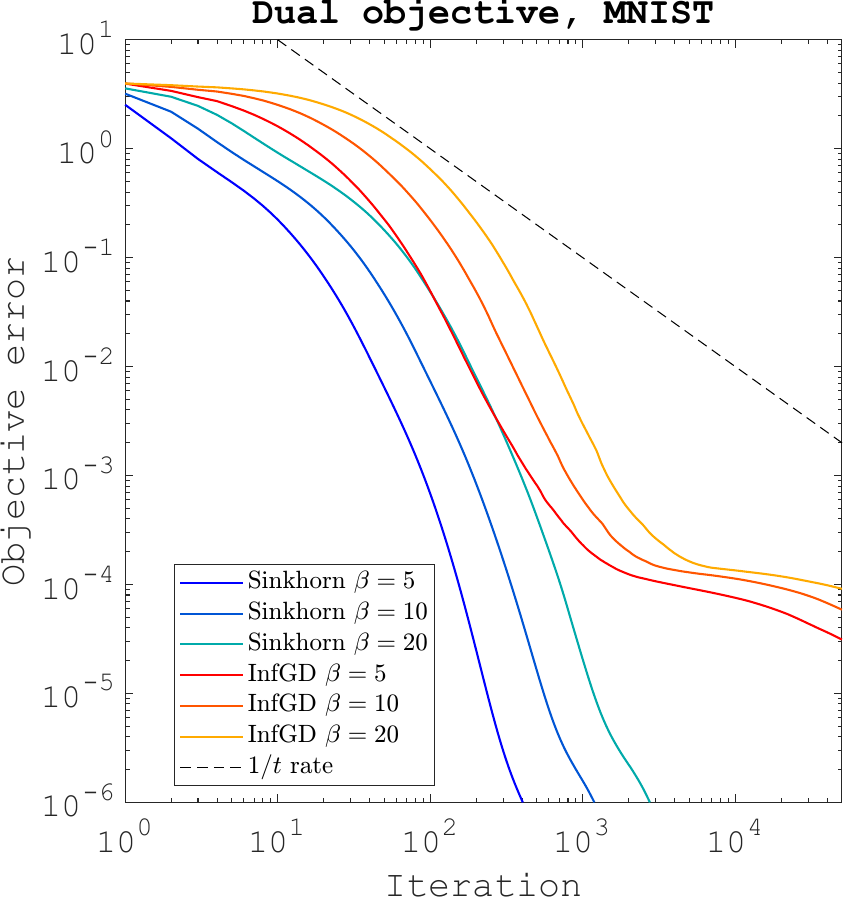}
$\qquad$ \includegraphics[bb=0bp 0bp 440bp 440bp,clip,scale=0.4]{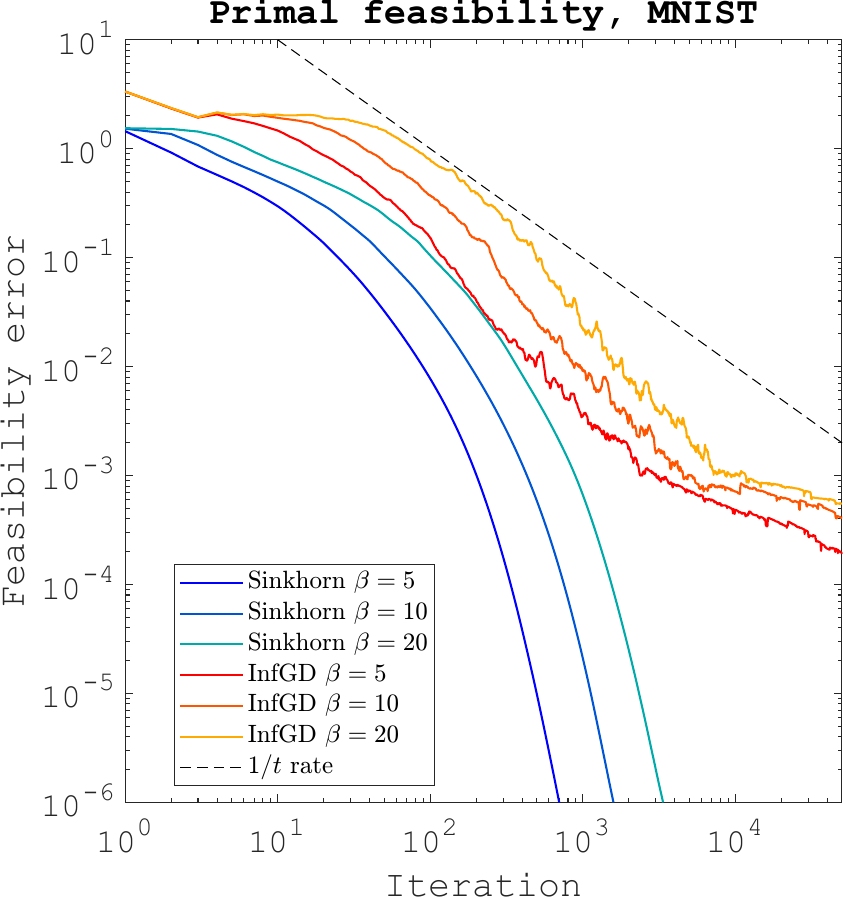}
\par\end{centering}
\begin{centering}
\vspace{3mm}
\par\end{centering}
\centering{}\caption{Convergence profile of the dual objective error (\ref{eq:ot-obj-error})
and the primal feasibility error (\ref{eq:ot-feas}) for various $\beta$
in the OT LP. Experiments are described in the main text. In the top
and bottom rows we consider the MNIST and Synthetic datasets, respectively.
\label{fig:ot-beta}}
\end{figure}

\begin{figure}
\begin{centering}
\includegraphics[bb=0bp 0bp 440bp 440bp,clip,scale=0.4]{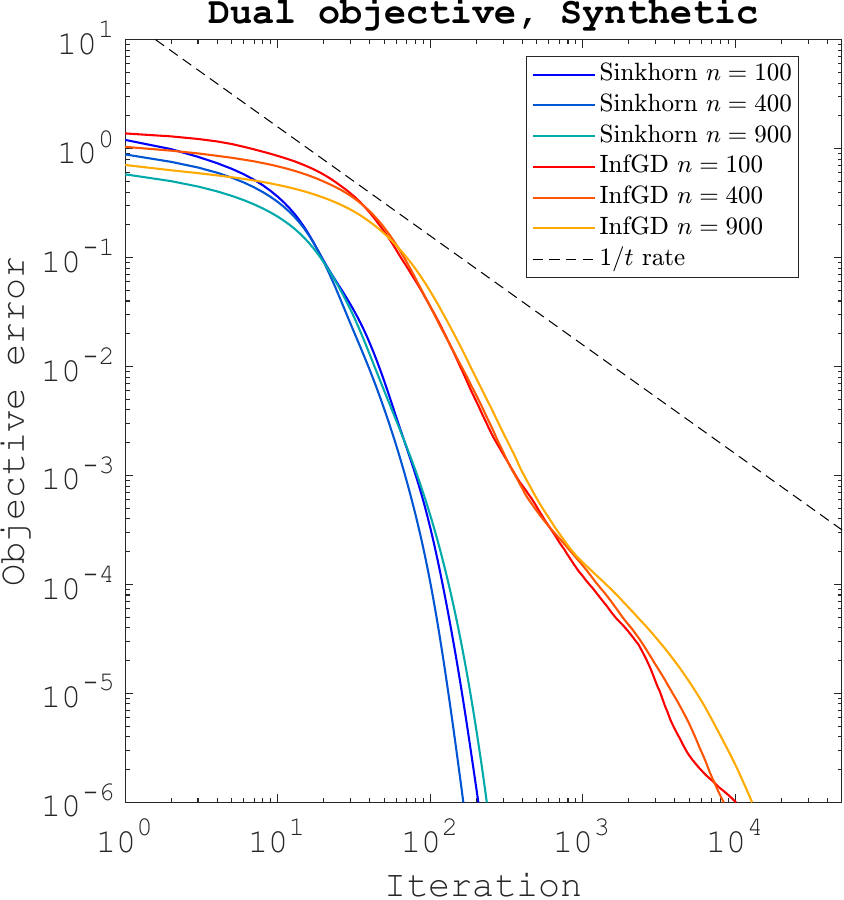}
$\qquad$ \includegraphics[bb=0bp 0bp 440bp 440bp,clip,scale=0.4]{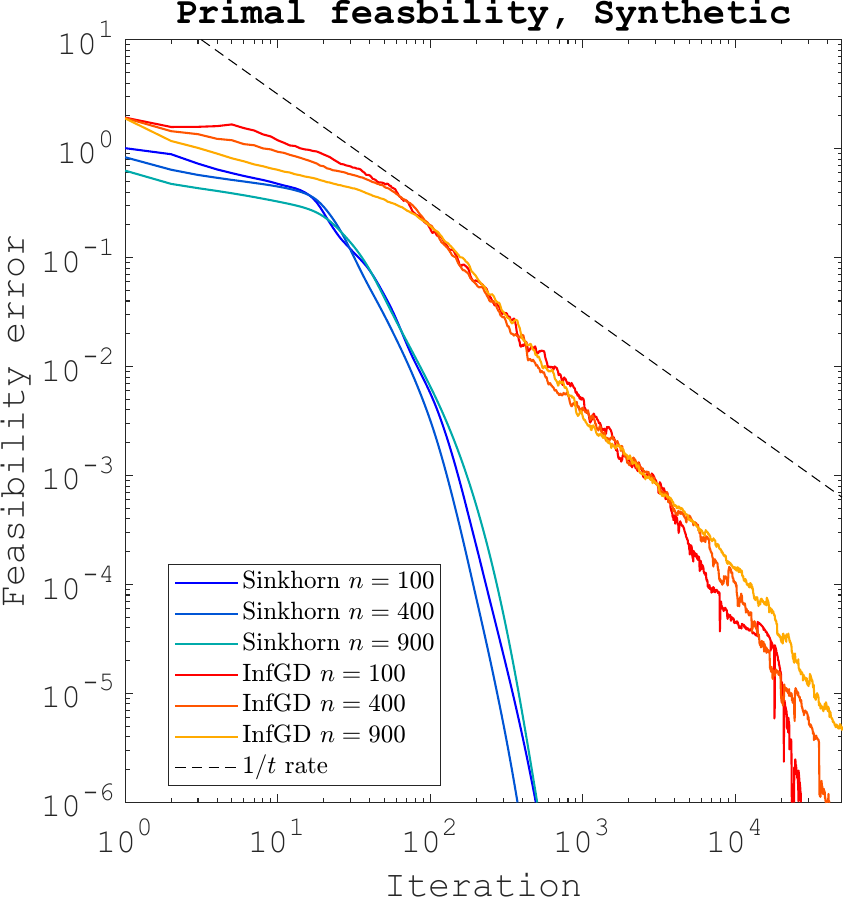}\vspace{5mm}
\par\end{centering}
\begin{centering}
\includegraphics[bb=0bp 0bp 440bp 440bp,clip,scale=0.4]{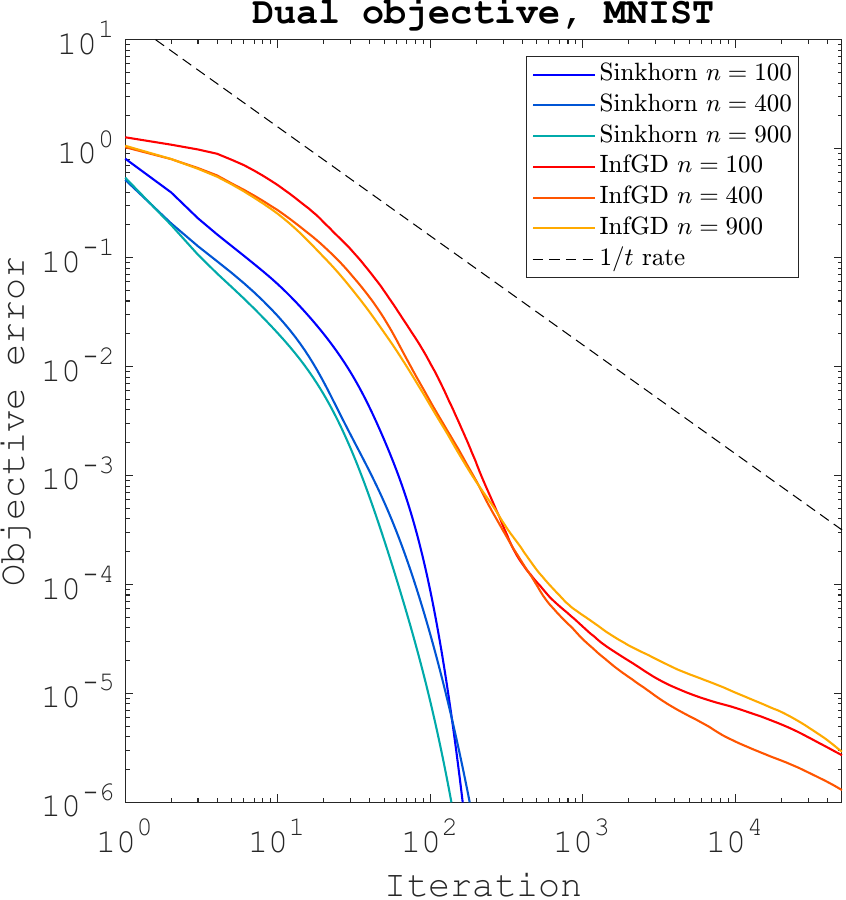}
$\qquad$ \includegraphics[bb=0bp 0bp 440bp 440bp,clip,scale=0.4]{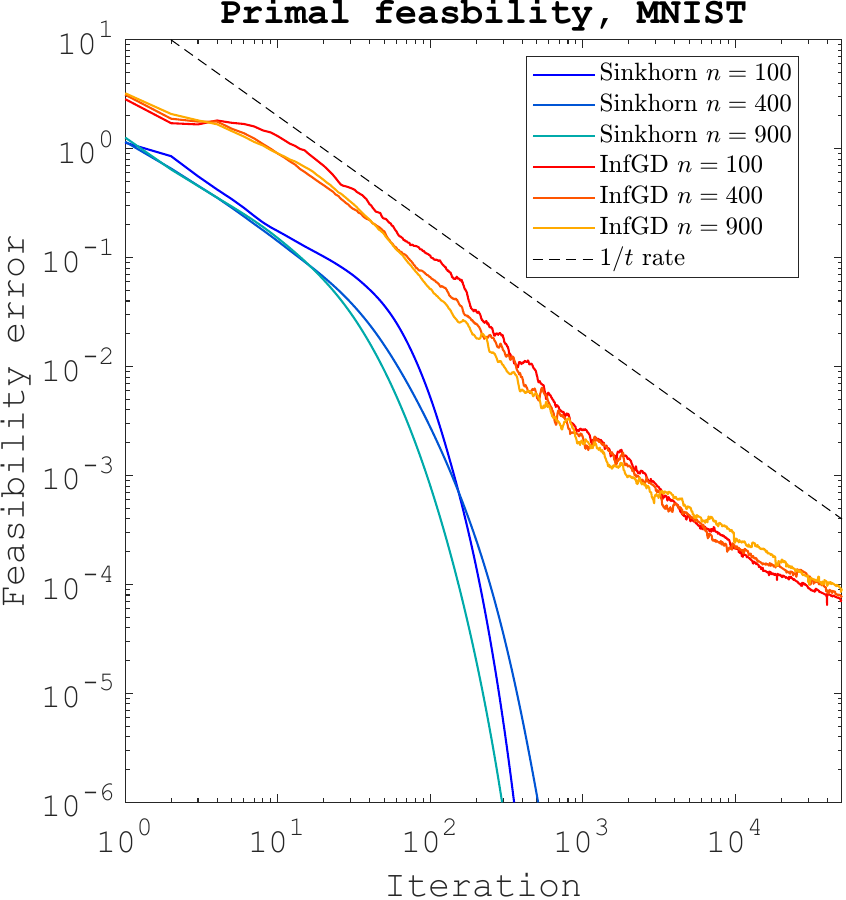}
\par\end{centering}
\begin{centering}
\vspace{3mm}\caption{Convergence profile of the dual objective error (\ref{eq:ot-obj-error})
and the primal feasibility error (\ref{eq:ot-feas}) for various problem
dimensions $n$ in the OT LP. Experiments are described in the main
text. In the top and bottom rows we consider MNIST and Synthetic datasets,
respectively. \label{fig:ot-n}}
\par\end{centering}
\end{figure}

\subsection{Strong Perm-Synch SDP \label{subsec:Numerics-Strong-Perm-Synch-SDP}}

We set up our permutation synchronization problems as follows. Let
$M$ be the size of a full global registry of keypoints. Let $N$
and $K\leq M$ be the number of images and number of keypoints per
image as in the discussion of Section \ref{subsec:cases}. For each
image, we uniformly randomly sample $K$ keypoints from the registry.
For each pair $(i,j)$ of images, we let $-C^{(i,j)}$ encode the
pairwise correspondences between the keypoints in images $i$ and
$j$. However, for each pair $i\neq j$, with probability $p$ we
corrupt the cost matrix by uniformly randomly mislabeling the keypoints
in both images before encoding the correspondences. To implement the
necessary exponential matvecs, we use the Chebfun library \cite{driscoll2014chebfun}
to perform Chebyshev fitting for the exponential on a bounding interval
for the spectrum obtained via power iteration, then implement a Chebyshev
expansion.

For the strong Perm-Synch SDP, we fix the number of keypoints per
image to be $K=10$ and the corruption probability to be $p=0.15$.
We peg $M=\lceil N/2\rceil$ and $\beta=10\times(\log N)/N$, and
we vary $N=20,40,60,80,100$, plotting the primal feasibility error
\begin{equation}
\sum_{i=1}^{N}\left\Vert \hat{X}_{t}^{(i,i)}-\mathbf{I}_{K}/n\right\Vert _{\Tr}\label{eq:strongpermsynch_feas}
\end{equation}
 in Figure \ref{fig:strong_feas}. In the figure, we consider two
choices for the number of stochastic vectors: $S=\lceil8\times K\log N\rceil$
and $S=\lceil32\times K\log N\rceil$. Accordingly we expect the feasibility
error to saturate at roughly half as large of a value for the latter
choice. This claim is borne out empirically, as is the dimension-independent
convergence profile of the feasibility error. For all of these tests,
the recovery procedure of \cite{lindseyshi2025} always achieves exact
recovery of the underlying correspondences.

Note (cf. \cite{lindseyshi2025}) that one expects a spectral method
or low-rank SDP solver to require $\tilde{O}(N^{2}KM+NKM^{2})=\tilde{O}(N^{3})$
operations in terms of the varying quantity $N$, while our solver
for the strong Perm-Synch SDP requires only $\tilde{O}(N^{2}K^{2}+NK^{3})=\tilde{O}(N^{2})$
operations (ignoring error dependence).

\begin{figure}[H]
\begin{centering}
\includegraphics[scale=0.4]{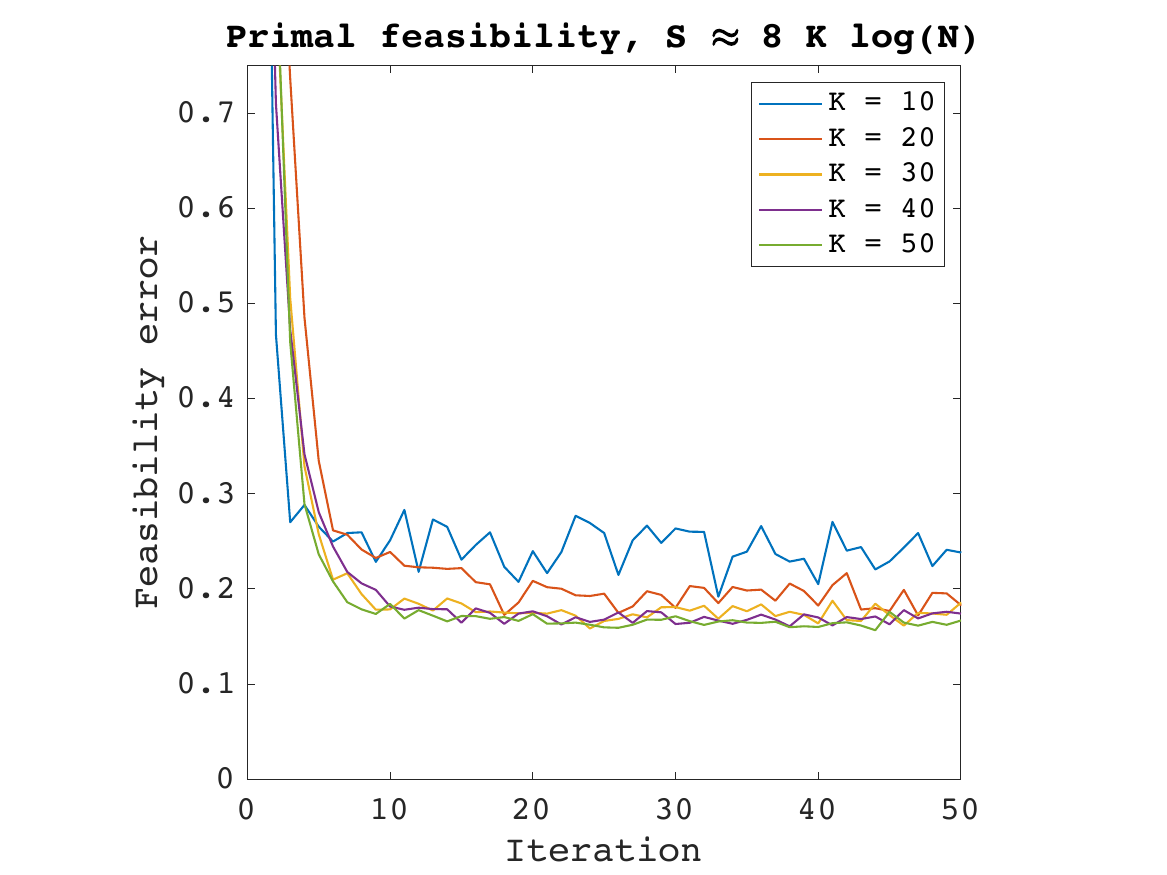}\includegraphics[scale=0.4]{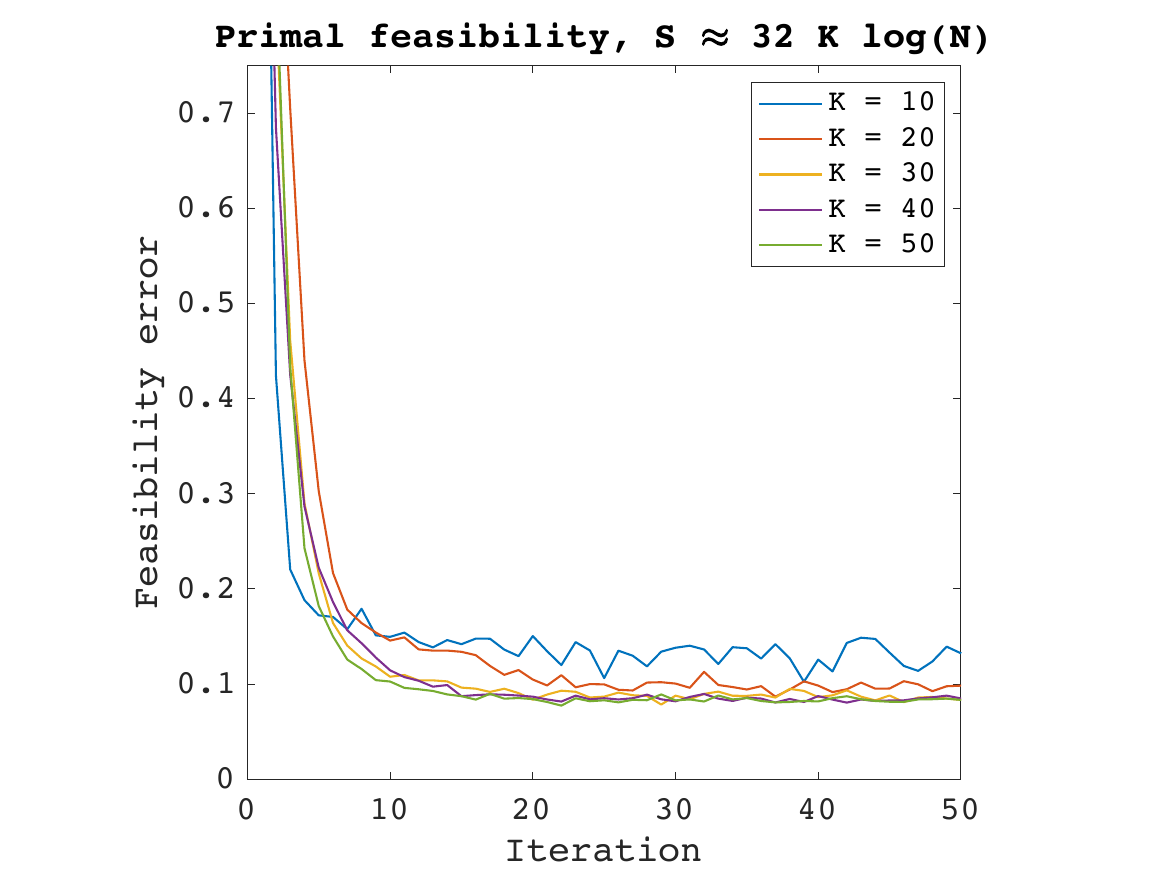}
\par\end{centering}
\caption{Convergence profile of the primal feasibility error (\ref{eq:strongpermsynch_feas})
for the strong Perm-Synch SDP. Experiments are described in the main
text. At left and right we consider $S=\lceil8\times K\log N\rceil$
and $S=\lceil32\times K\log N\rceil$, respectively. \label{fig:strong_feas}}
\end{figure}

\subsection{Weak Perm-Synch SDP}

For the weak Perm-Synch SDP, we maintain the same model as described
at the beginning of Section \ref{subsec:Numerics-Strong-Perm-Synch-SDP}.
However, we now peg $M=K$, $p=0.1$, $N=100$, and $\beta=10\times(\log N)/N$,
and we vary $K=10,20,30,40,50$, plotting the primal feasibility error
\begin{equation}
\sqrt{\Vert\mathrm{diag}(\hat{X}_{t})-\mathbf{1}_{n}/n\Vert_{1}^{2}+\left(\sum_{i=1}^{N}\left|\frac{\mathbf{1}_{K}^{\top}\hat{X}_{t}^{(i,i)}\mathbf{1}_{K}}{K}-\frac{1}{n}\right|\right)^{2}}\label{eq:weakpermsynch_feas}
\end{equation}
in Figure \ref{fig:weak_feas}. Again, we consider two choices for
the number of stochastic vectors: $S=\lceil8\times K\log N\rceil$
and $S=\lceil32\times K\log N\rceil$, and the error saturates at
a value about half as large for the latter choice. The experiments
confirm the dimension-independent convergence profile of the feasibility
error. For all of these tests, the recovery procedure of \cite{lindseyshi2025}
always achieves exact recovery of the underlying correspondences.

Note (cf. \cite{lindseyshi2025}) that one expects a spectral method
or low-rank SDP solver to require $\tilde{O}(N^{2}KM+NKM^{2})=\tilde{O}(K^{3})$
operations in terms of the varying quantity $K$, while our solver
for the strong Perm-Synch SDP requires only $\tilde{O}(N^{2}K)=\tilde{O}(K)$
operations (ignoring error dependence).

\begin{figure}[H]
\begin{centering}
\includegraphics[scale=0.4]{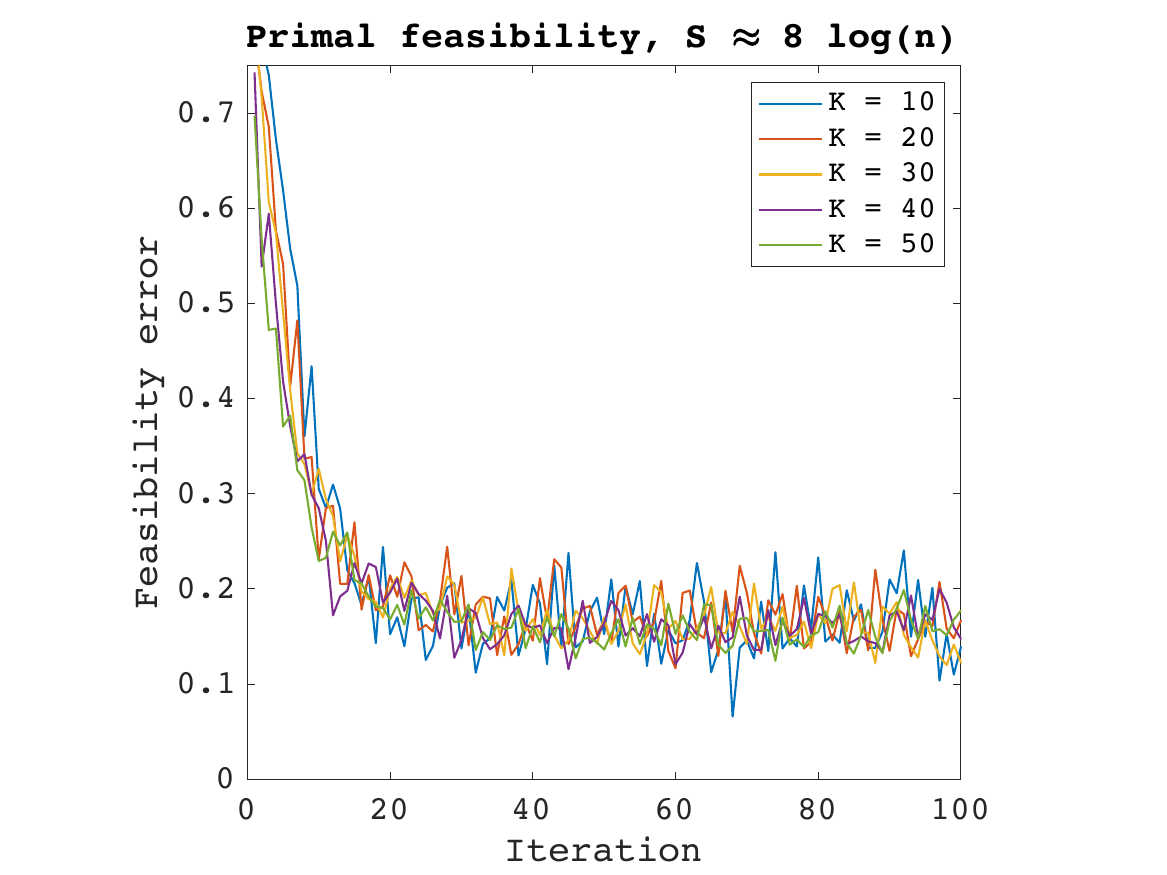}\includegraphics[scale=0.4]{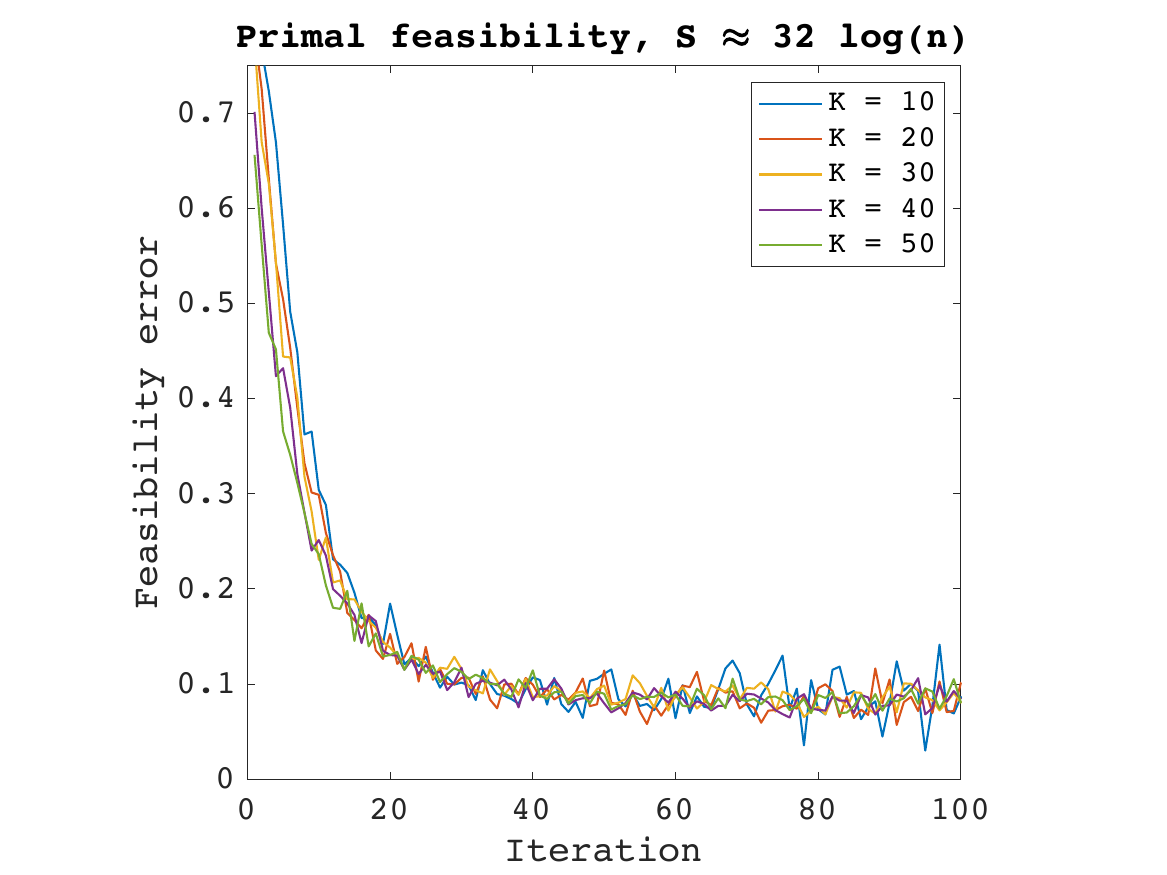}
\par\end{centering}
\caption{Convergence profile of the primal feasibility error (\ref{eq:weakpermsynch_feas})
for the weak Perm-Synch SDP. Experiments are described in the main
text. At left and right we consider $S=\lceil8\times\log n\rceil$
and $S=\lceil32\times\log n\rceil$, respectively. \label{fig:weak_feas} }
\end{figure}
\bibliographystyle{plain}
\bibliography{sdp_gd}

\begin{thebibliography}{10}

\bibitem{expmv}
Awad~H. Al-Mohy and Nicholas~J. Higham.
\newblock Computing the action of the matrix exponential, with an application to exponential integrators.
\newblock {\em SIAM Journal on Scientific Computing}, 33(2):488--511, 2011.

\bibitem{allen2014linear}
Zeyuan Allen-Zhu and Lorenzo Orecchia.
\newblock Linear coupling: An ultimate unification of gradient and mirror descent.
\newblock {\em arXiv preprint arXiv:1407.1537}, 2014.

\bibitem{altschuler2017near}
Jason Altschuler, Jonathan Weed, and Philippe Rigollet.
\newblock Near-linear time approximation algorithms for optimal transport via sinkhorn iteration.
\newblock In {\em Advances in Neural Information Processing Systems (NeurIPS)}, volume~30, 2017.

\bibitem{arora2012multiplicative}
Sanjeev Arora, Elad Hazan, and Satyen Kale.
\newblock The multiplicative weights update method: a meta-algorithm and applications.
\newblock {\em Theory of computing}, 8(1):121--164, 2012.

\bibitem{beck2017first}
Amir Beck.
\newblock {\em First-order methods in optimization}.
\newblock SIAM, 2017.

\bibitem{beck2003mirror}
Amir Beck and Marc Teboulle.
\newblock Mirror descent and nonlinear projected subgradient methods for convex optimization.
\newblock {\em Operations Research Letters}, 31(3):167--175, 2003.

\bibitem{manopt}
N.~Boumal, B.~Mishra, P.-A. Absil, and R.~Sepulchre.
\newblock {M}anopt, a {M}atlab toolbox for optimization on manifolds.
\newblock {\em Journal of Machine Learning Research}, 15(42):1455--1459, 2014.

\bibitem{boyd2004convex}
S.P. Boyd and L.~Vandenberghe.
\newblock {\em Convex Optimization}.
\newblock Number pt. 1 in Berichte {\"u}ber verteilte messysteme. Cambridge University Press, 2004.

\bibitem{boyd1994linear}
Stephen Boyd, Laurent El~Ghaoui, Eric Feron, and Venkataramanan Balakrishnan.
\newblock {\em Linear Matrix Inequalities in System and Control Theory}.
\newblock SIAM Studies in Applied Mathematics. SIAM, 1994.

\bibitem{BurerMonteiro}
Samuel Burer and Renato D.~C. Monteiro.
\newblock A nonlinear programming algorithm for solving semidefinite programs via low-rank factorization.
\newblock {\em Mathematical Programming}, 95(2):329--357, 2003.

\bibitem{carlen2014remainder}
Eric~A Carlen and Elliott~H Lieb.
\newblock Remainder terms for some quantum entropy inequalities.
\newblock {\em Journal of Mathematical Physics}, 55(4), 2014.

\bibitem{cuturi2013sinkhorn}
Marco Cuturi.
\newblock Sinkhorn distances: Lightspeed computation of optimal transport.
\newblock {\em Advances in neural information processing systems}, 26, 2013.

\bibitem{doherty2002distinguishing}
Andrew~C Doherty, Pablo~A Parrilo, and Federico~M Spedalieri.
\newblock Distinguishing separable and entangled states.
\newblock {\em Physical Review Letters}, 88(18):187904, 2002.

\bibitem{driscoll2014chebfun}
Tobin~A. Driscoll, Nicholas Hale, and Lloyd~N. Trefethen.
\newblock Chebfun: A new kind of numerical computing.
\newblock {\em SIAM Review}, 56(4):561--588, 2014.

\bibitem{dvurechensky2018computational}
Pavel Dvurechensky, Alexander Gasnikov, and Alexey Kroshnin.
\newblock Computational optimal transport: Complexity by accelerated gradient descent is better than by sinkhorn algorithm.
\newblock In {\em International conference on machine learning}, pages 1367--1376. PMLR, 2018.

\bibitem{garrigos2023handbook}
Guillaume Garrigos and Robert~M Gower.
\newblock Handbook of convergence theorems for (stochastic) gradient methods.
\newblock {\em arXiv preprint arXiv:2301.11235}, 2023.

\bibitem{georgii2011gibbs}
Hans-Otto Georgii.
\newblock {\em Gibbs Measures and Phase Transitions}.
\newblock De Gruyter, 2nd edition, 2011.

\bibitem{goemans1995improved}
Michel~X Goemans and David~P Williamson.
\newblock Improved approximation algorithms for maximum cut and satisfiability problems using semidefinite programming.
\newblock {\em Journal of the ACM (JACM)}, 42(6):1115--1145, 1995.

\bibitem{golub2013matrix}
Gene~H. Golub and Charles~F. Van~Loan.
\newblock {\em Matrix Computations}.
\newblock Johns Hopkins University Press, 4th edition, 2013.

\bibitem{hutchinson1989stochastic}
Michael~F Hutchinson.
\newblock A stochastic estimator of the trace of the influence matrix for laplacian smoothing splines.
\newblock {\em Communications in Statistics-Simulation and Computation}, 18(3):1059--1076, 1989.

\bibitem{kakade2009duality}
Sham Kakade, Shai Shalev-Shwartz, Ambuj Tewari, et~al.
\newblock On the duality of strong convexity and strong smoothness: Learning applications and matrix regularization.
\newblock {\em Unpublished Manuscript, http://ttic. uchicago. edu/shai/papers/KakadeShalevTewari09. pdf}, 2(1):35, 2009.

\bibitem{Komiya1988}
H.~Komiya.
\newblock Elementary proof for {Sion's} minimax theorem.
\newblock {\em Kodai Math. J.}, 11(1):5--7, 1988.

\bibitem{krechetov2018entropy}
Mikhail Krechetov, Jakub Marecek, Yury Maximov, and Martin Takac.
\newblock Entropy penalized semidefinite programming.
\newblock {\em arXiv preprint arXiv:1802.04332}, 2018.

\bibitem{lanckriet2004learning}
Gert~RG Lanckriet, Nello Cristianini, Peter Bartlett, Laurent El~Ghaoui, and Michael~I Jordan.
\newblock Learning the kernel matrix with semidefinite programming.
\newblock {\em Journal of Machine Learning Research}, 5:27--72, 2004.

\bibitem{li2024fastcomputationoptimaltransport}
Gen Li, Yanxi Chen, Yu~Huang, Yuejie Chi, H.~Vincent Poor, and Yuxin Chen.
\newblock Fast computation of optimal transport via entropy-regularized extragradient methods, 2024.

\bibitem{lin2022variational}
Lin Lin and Michael Lindsey.
\newblock Variational embedding for quantum many-body problems.
\newblock {\em Communications on Pure and Applied Mathematics}, 75(9):2033--2068, 2022.

\bibitem{lin2022efficiency}
Tianyi Lin, Nhat Ho, and Michael~I Jordan.
\newblock On the efficiency of entropic regularized algorithms for optimal transport.
\newblock {\em Journal of Machine Learning Research}, 23(137):1--42, 2022.

\bibitem{lindsey2023fast}
Michael Lindsey.
\newblock Fast randomized entropically regularized semidefinite programming.
\newblock {\em arXiv preprint arXiv:2303.12133}, 2023.

\bibitem{lindseyshi2025}
Michael Lindsey and Yunpeng Shi.
\newblock Fast entropy-regularized {SDP} relaxations for permutation synchronization.
\newblock {\em In preparation}, 2025.

\bibitem{Martinsson_Tropp_2020}
Per-Gunnar Martinsson and Joel~A. Tropp.
\newblock Randomized numerical linear algebra: {F}oundations and algorithms.
\newblock {\em Acta Numerica}, 29:403--572, 2020.

\bibitem{meyer2021hutch++}
Raphael~A Meyer, Cameron Musco, Christopher Musco, and David~P Woodruff.
\newblock Hutch++: Optimal stochastic trace estimation.
\newblock In {\em Symposium on Simplicity in Algorithms (SOSA)}, pages 142--155. SIAM, 2021.

\bibitem{nemirovskij1983problem}
Arkadij~Semenovi{\v{c}} Nemirovskij and David~Borisovich Yudin.
\newblock Problem complexity and method efficiency in optimization.
\newblock 1983.

\bibitem{NielsenChuang}
M.~Nielsen and I.~Chuang.
\newblock {\em Quantum computation and quantum information}.
\newblock Cambridge University Press, 2001.

\bibitem{pachauri2013solving}
Deepti Pachauri, Risi Kondor, and Vikas Singh.
\newblock Solving the multi-way matching problem by permutation synchronization.
\newblock {\em Advances in neural information processing systems}, 26, 2013.

\bibitem{parrilo2000structured}
Pablo~A Parrilo.
\newblock {\em Structured Semidefinite Programs and Semialgebraic Geometry Methods in Robustness and Optimization}.
\newblock PhD thesis, California Institute of Technology, 2000.

\bibitem{pavlov2024logarithmically}
Dmitrii Pavlov.
\newblock Logarithmically sparse symmetric matrices.
\newblock {\em Beitr{\"a}ge zur Algebra und Geometrie/Contributions to Algebra and Geometry}, pages 1--16, 2024.

\bibitem{pavlov2024gibbs}
Dmitrii Pavlov, Bernd Sturmfels, and Simon Telen.
\newblock Gibbs manifolds.
\newblock {\em Information Geometry}, 7(Suppl 2):691--717, 2024.

\bibitem{peyre2019computational}
Gabriel Peyr{\'e}, Marco Cuturi, et~al.
\newblock Computational optimal transport: With applications to data science.
\newblock {\em Foundations and Trends{\textregistered} in Machine Learning}, 11(5-6):355--607, 2019.

\bibitem{pooladian2024entropicestimationoptimaltransport}
Aram-Alexandre Pooladian and Jonathan Niles-Weed.
\newblock Entropic estimation of optimal transport maps, 2024.

\bibitem{sidford2018coordinate}
Aaron Sidford and Kevin Tian.
\newblock Coordinate methods for accelerating $\ell_\infty$ regression and faster approximate maximum flow.
\newblock In {\em 2018 IEEE 59th Annual Symposium on Foundations of Computer Science (FOCS)}, pages 922--933. IEEE, 2018.

\bibitem{trefethen2019approximation}
Lloyd~N Trefethen.
\newblock {\em Approximation theory and approximation practice, extended edition}.
\newblock SIAM, 2019.

\bibitem{vonneumann1932mathematische}
John von Neumann.
\newblock Mathematische grundlagen der quantenmechanik.
\newblock 1932.
\newblock Translated as *Mathematical Foundations of Quantum Mechanics*.

\bibitem{pmlr-v75-weed18a}
Jonathan Weed.
\newblock An explicit analysis of the entropic penalty in linear programming.
\newblock In S\'ebastien Bubeck, Vianney Perchet, and Philippe Rigollet, editors, {\em Proceedings of the 31st Conference On Learning Theory}, volume~75 of {\em Proceedings of Machine Learning Research}, pages 1841--1855. PMLR, 2018.

\bibitem{weinberger2009distance}
Kilian~Q Weinberger and Lawrence~K Saul.
\newblock Distance metric learning for large margin nearest neighbor classification.
\newblock {\em Journal of Machine Learning Research}, 10:207--244, 2009.

\bibitem{SketchyCGAL}
Alp Yurtsever, Joel~A. Tropp, Olivier Fercoq, Madeleine Udell, and Volkan Cevher.
\newblock Scalable semidefinite programming.
\newblock {\em SIAM Journal on Mathematics of Data Science}, 3(1):171--200, 2021.

\end{thebibliography}

\pagebreak{}

\part*{Appendices}

\appendix

\section{Proofs for norms \label{app:dual-norm}}

In this section, we offer proofs deferred from Section \ref{subsec:normcases}.

\subsection{OT LP \label{app:dual-norm:ot}}
\begin{prop}
For the norm 
\[
\|\blam\|=\|(\phi,\psi)\|=\sqrt{2(\|\phi\|_{\infty}^{2}+\|\psi\|_{\infty}^{2})}
\]
 satisfying Assumption \ref{assump:norm} for the OT LP, the corresponding
dual norm is given by 
\[
\|(a,b)\|_{*}=\sqrt{\frac{1}{2}(\|a\|_{1}^{2}+\|b\|_{1}^{2})}
\]
 for $a\in\R^{m}$, $b\in\R^{n}$.
\end{prop}

\begin{proof}
Compute: 
\begin{align*}
\|(a,b)\|_{*} & \coloneqq\max_{\|(\phi,\psi)\|\le1}\left\{ \langle\phi,a\rangle+\langle\psi,b\rangle\right\} \\
 & \le\max_{\|(\phi,\psi)\|\le1}\left\{ \|\phi\|_{\infty}\|a\|_{1}+\|\psi\|_{\infty}\|b\|_{1}\right\} \\
 & \le\max_{\|(\phi,\psi)\|\le1}\left\{ \sqrt{\|\phi\|_{\infty}^{2}+\|\psi\|_{\infty}^{2}}\right\} \cdot\sqrt{\|a\|_{1}^{2}+\|b\|_{1}^{2}}\\
 & =\sqrt{\frac{1}{2}(\|a\|_{1}^{2}+\|b\|_{1}^{2})}.
\end{align*}
Here the first inequality follows from H\"{o}lder's inequality, and
the second inequality follows from the Cauchy-Schwarz inequality.

To make the equality hold, we can take 
\[
\phi=\frac{\|a\|_{1}\sign(a)}{\sqrt{2(\|a\|_{1}^{2}+\|b\|_{1}^{2})}},\qquad\psi=\frac{\|b\|_{1}\sign(b)}{\sqrt{2(\|a\|_{1}^{2}+\|b\|_{1}^{2})}}.
\]
 This concludes the proof.
\end{proof}

\subsection{Strong Perm-Synch SDP \label{app:dual-norm:strong-ps}}
\begin{prop}
For the norm 
\[
\|\Blam\|=\max_{i=1,...,N}\|\Lambda^{(i)}\|_{2}
\]
 satisfying Assumption \ref{assump:norm} for the strong Perm-Synch
SDP, the corresponding dual norm for a vector $\bB=(B^{(1)},\ldots,B^{(N)}),$where
$B^{(i)}\in\R^{K\times K}$ for $i=1,\ldots,N$, is given by 
\[
\|\bB\|_{*}=\sum_{i=1}^{n}\|B^{(i)}\|_{\Tr}.
\]
\end{prop}

\begin{proof}
In the following we use the classical fact that $\Tr[AB]\leq\Vert A\Vert_{2}\Vert B\Vert_{\Tr}$
for arbitrary matrices $A,B$, which can also be viewed as a special
case of H\"{o}lder's inequality for the Schatten $p$-norms.

Then compute: 
\begin{align*}
\|\bB\|_{*} & :=\max_{\|\Blam\|\le1}\left\{ \sum_{i=1}^{N}\Tr[(\Lambda^{(i)})^{\top}B^{(i)}]\right\} \\
 & \le\max_{\|\Blam\|\le1}\left\{ \sum_{i=1}^{N}\|\Lambda^{(i)}\|_{2}\cdot\|B^{(i)}\|_{\Tr}\right\} \\
 & \le\max_{\|\Blam\|\le1}\left\{ \max_{i=1,\ldots,N}\|\Lambda^{(i)}\|_{2}\right\} \cdot\sum_{i=1}^{N}\|B^{(i)}\|_{\Tr}\\
 & =\sum_{i=1}^{N}\|B^{(i)}\|_{\Tr}.
\end{align*}
Here the first inequality follows from the aforementioned fact.

Note that the inequalities hold with equality holds by taking $\Lambda^{(i)}=U^{(i)}V^{(i)\top}$,
where if $B^{(i)}=U^{(i)}\Sigma^{(i)}V^{(i)\top}$ is an SVD. This
concludes the proof.
\end{proof}

\section{Proofs for algorithms \label{app:algo}}

In this section, we will show how to figure out the exact form of
the update (\ref{eq:optstep}) for each problem of interest. Before
we proceed to specifics, we prove a useful lemma here:
\begin{lem}
\label{lem:infty-minimizer}Given $g\in\R^{n}$, we have 
\[
\underset{\blam\in\R^{n}}{\mathrm{argmin}}\Big\{\langle g,\blam\rangle+\frac{1}{2\eta}\|\blam\|_{\infty}^{2}\Big\}=-\eta\|g\|_{1}\,\sign(g).
\]
\end{lem}

\begin{proof}
For convenience, define 
\[
f(\blam)\coloneqq\langle g,\blam\rangle+\frac{1}{2\eta}\|\blam\|_{\infty}^{2}.
\]
Note that for any $\blam$ if we let $\blam'=-\|\blam\|_{\infty}\,\sign(g)$,
we have 
\[
\langle g,\blam\rangle\ge-\|\blam\|_{\infty}\,\|g\|_{1}=\langle g,\blam'\rangle
\]
 with equality if and only if $\blam$ is entrywise proportional to
$\mathrm{sign}(g)$. This in turn implies that 
\begin{align*}
f(\blam) & =\langle g,\blam\rangle+\frac{1}{2\eta}\|\blam\|_{\infty}^{2}\\
 & \ge-\|\blam\|_{\infty}\,\|g\|_{1}+\frac{1}{2\eta}\|\blam\|_{\infty}^{2}\\
 & =\langle g,\blam'\rangle+\frac{1}{2\eta}\|\blam'\|_{\infty}^{2}\\
 & =f(\blam'),
\end{align*}
 with equality if and only if $\blam$ is entrywise proportional to
$\mathrm{sign}(g)$. Therefore, a minimizer $\blam_{*}$ in the statement
of the lemma must be of the form $-a\cdot\sign(g)$ for some scalar
$a$. Note that
\begin{align*}
f(-a\cdot\sign(g)) & =-\langle g,a\cdot\sign(g)\rangle+\frac{1}{2\eta}\|a\cdot\sign(g)\|_{\infty}^{2}\\
 & =-a\|g\|_{1}+\frac{1}{2\eta}\cdot a^{2}\\
 & =\frac{1}{2\eta}(a-\eta\|g\|_{1})^{2}-\frac{\eta}{2}\|g\|_{1}^{2}.
\end{align*}
Hence, the minimizer is $a=\eta\|g\|_{1}$ and $\blam_{*}=-\eta\|g\|_{1}\,\sign(g)$.
This concludes the proof.
\end{proof}

\subsection{Max-Cut SDP \label{app:algo-max-cut}}

For the Max-Cut SDP, the update is the solution to the following optimization
problem:
\[
\underset{\blam\in\R^{m}}{\text{min}}\Big\{ f_{\beta}(\blam_{t})+\langle g_{t},\blam-\blam_{t}\rangle+\frac{1}{2\eta}\Vert\blam-\blam_{t}\Vert_{\infty}^{2}\Big\}.
\]
 We can let $\bx=\blam-\blam_{t}$ and apply Lemma \ref{lem:infty-minimizer}
directly to obtain: 
\begin{prop}
[Max-Cut SDP update]The exact update of (\ref{eq:optstep}) for the
Max-Cut SDP is
\[
\blam_{t+1}=\blam_{t}-\eta\|g_{t}\|_{1}\,\sign(g_{t}).
\]
\end{prop}

\subsection{OT LP\label{app:algo-ot}}

For the OT LP, we have two dual variables $\phi,\psi$. The dual gradients
are: 
\begin{align*}
\nabla_{\phi}f_{\beta}(\phi,\psi) & =-\mu+\pi_{\beta,(\phi,\psi)}\mathbf{1}_{n},\\
\nabla_{\psi}f_{\beta}(\phi,\psi) & =-\nu+\pi_{\beta,(\phi,\psi)}^{\top}\mathbf{1}_{m}.
\end{align*}
Here $\pi_{\beta,(\phi,\psi)}\propto\left(e^{-\beta(c_{ij}-\phi_{i}-\psi_{j})}\right)$.
Let $g_{t}^{(1)}=\nabla_{\phi}f_{\beta}(\phi_{t},\psi_{t})$ and $g_{t}^{(2)}=\nabla_{\psi}f_{\beta}(\phi_{t},\psi_{t})$.
Then we want to solve the following optimization problem (omitting
the irrelevant constant term):
\[
\underset{\phi\in\R^{m},\,\psi\in\R^{n}}{\text{min}}\Big\{\langle g_{t}^{(1)},\phi-\phi_{t}\rangle+\langle g_{t}^{(2)},\psi-\psi_{t}\rangle+\frac{1}{\eta}\big[\Vert\phi-\phi_{t}\Vert_{\infty}^{2}+\|\psi-\psi_{t}\|_{\infty}^{2}\big]\Big\}.
\]
 Note that the optimization decouples for $\phi$ and $\psi$, and
we can still use Lemma \ref{lem:infty-minimizer} to deduce: 
\begin{prop}
[OT LP update]The exact update of (\ref{eq:optstep}) for the OT
LP is
\begin{align*}
\phi_{t+1} & =\phi_{t}-\frac{\eta}{2}\|g_{t}^{(1)}\|\,\sign(g_{t}^{(1)}),\\
\psi_{t+1} & =\psi_{t}-\frac{\eta}{2}\|g_{t}^{(2)}\|\,\sign(g_{t}^{(2)}).
\end{align*}
\end{prop}

\subsection{Strong Perm-Synch SDP\label{app:algo-strong-ps}}

The optimization problem for the update in this case is 
\[
\min_{\Blam}\Big\{ f_{\beta}(\Blam_{t})+\sum_{i=1}^{n}\langle G_{t}^{(i)},\Lambda^{(i)}-\Lambda_{t}^{(i)}\rangle+\frac{1}{2\eta}\max_{i=1,...,N}\|\Lambda^{(i)}-\Lambda_{t}^{(i)}\|_{2}^{2}\Big\}.
\]

\begin{prop}
[Strong Perm-Synch SDP update]The exact update of (\ref{eq:optstep})
for the strong Perm-Synch SDP is
\[
\Lambda_{t+1}^{(i)}=\Lambda_{t}^{(i)}-\eta\left(\sum_{j=1}^{N}\Vert G_{t}^{(j)}\Vert_{\Tr}\right)\,\mathrm{sign}[G_{t}^{(i)}].
\]
\end{prop}

\begin{proof}
We can let $X^{(i)}=\Lambda^{(i)}-\Lambda_{t}^{(i)}$, $\mathbf{X}=(X^{(1)},\ldots,X^{(N)})$,
and seek to minimize $h(\mathbf{X})$ defined by 
\begin{align*}
h(\bX) & \coloneqq g_{\beta}(\Blam_{t})+\sum_{i=1}^{n}\langle G_{t}^{(i)},X^{(i)}\rangle+\frac{1}{2\eta}\,\max_{i=1,...,N}\left\{ \|X^{(i)}\|_{2}^{2}\right\} \\
 & \overset{(*)}{\geq}g_{\beta}(\Blam_{t})-\sum_{i=1}^{n}\|G_{t}^{(i)}\|_{\Tr}\,\|X^{(i)}\|_{2}+\frac{1}{2\eta}\,\max_{i=1,...,N}\left\{ \|X^{(i)}\|_{2}^{2}\right\} .
\end{align*}
If we just consider optimizing this lower bound over $(\|X^{(i)}\|_{2})_{i=1}^{n}$,
this reduces to the case for the Max-Cut SDP. Then by Lemma \ref{lem:infty-minimizer},
the optimal solution satisfies 
\[
\|X^{(i)}\|_{2}=\eta\bigg(\sum_{i=1}^{n}\|G_{t}^{(i)}\|_{\Tr}\text{\ensuremath{\bigg)}}\cdot\sign\Big(\|G_{t}^{(i)}\|_{\Tr}\Big)=\eta\bigg(\sum_{i=1}^{n}\|G_{t}^{(i)}\|_{\Tr}\text{\ensuremath{\bigg)}}.
\]
 Then if we can find $\mathbf{X}$ satisfying 
\[
\|X^{(i)}\|_{2}=\eta\bigg(\sum_{i=1}^{n}\|G_{t}^{(i)}\|_{\Tr}\text{\ensuremath{\bigg)}}
\]
 as well as 
\[
\sum_{i=1}^{n}\langle G_{t}^{(i)},X^{(i)}\rangle=-\sum_{i=1}^{n}\|G_{t}^{(i)}\|_{\Tr}\cdot\|X^{(i)}\|_{2},
\]
 so that the $(*)$ inequality holds with equality, we know that $\mathbf{X}$
must be an optimizer.

Note that if we take 
\[
X^{(i)}=-\eta\bigg(\sum_{i=1}^{n}\|G_{t}^{(i)}\|_{\Tr}\text{\ensuremath{\bigg)}}\mathrm{sign}(G_{t}^{(i)}),
\]
 for all $i,$ the desired equations are satisfied. This concludes
the proof.
\end{proof}

\subsection{Weak Perm-Synch SDP\label{app:algo-weak-ps}}

For the weak Perm-Synch SDP, we want to optimize (omitting the irrelevant
constant term):
\[
\min_{\blam,\bmu}\Big\{\langle g_{t},\blam-\blam_{t}\rangle+\left\langle h_{t},\bmu-\bmu_{t}\right\rangle +\frac{1}{\eta}\|\blam-\blam_{t}\|_{\infty}^{2}+\frac{1}{\eta}\|\bmu-\bmu_{t}\|_{\infty}^{2}\Big\}.
\]
Here 
\[
g_{t}=\mathrm{diag}[\hat{X}_{t}]-\frac{\mathbf{1}_{n}}{n},\qquad h_{t}^{(i)}=\frac{\mathbf{1}_{K}^{\top}\hat{X}_{t}^{(i,i)}\mathbf{1}_{K}}{K}-\frac{1}{n},\ \ i=1,\ldots,N,
\]
 defining $h_{t}\in\R^{N}$.

By the same argument as for the OT LP, we deduce: 
\begin{prop}
[Weak Perm-Synch SDP update]The exact update of (\ref{eq:optstep})
for the weak Perm-Synch SDP is
\[
\blam_{t+1}=\blam_{t}-\frac{\eta}{2}\Vert g_{t}\Vert_{1}\,\mathrm{sign}(g_{t}),
\]
\[
\bmu_{t+1}=\bmu_{t}-\frac{\eta}{2}\Vert h_{t}\Vert_{1}\,\mathrm{sign}(h_{t}).
\]
\end{prop}

\section{Proofs for gradient estimation \label{app:gradest}}
\begin{proof}
[Proof of Corollary \ref{cor:maxcutgradest}.] Let $\ve=\gamma/4\in(0,1/2)$.
Then the conclusion of Lemma (\ref{lem:concentration1}), together
with the union bound, implies in particular that 
\[
(1-\ve)\Tr[Y_{\beta,\blam_{t}}]\leq\Tr[\hat{Y}_{t}]\leq(1+\ve)\,\Tr[Y_{\beta,\blam_{t}}]
\]
 holds for all $t=0,\ldots,T-1$ with probability at least $1-\delta$.
Therefore, since $\hat{X}_{t}=\hat{Y}_{t}/\Tr[\hat{Y}_{t}]$ and $X_{\beta,\blam}=Y_{\beta,\blam}/\Tr[Y_{\beta,\blam}]$,
it follows that 
\[
\frac{1-\ve}{1+\ve}\,\mathrm{diag}[X_{\beta,\blam_{t}}]\leq\mathrm{diag}[\hat{X}_{t}]\leq\frac{1+\ve}{1-\ve}\,\mathrm{diag}[X_{\beta,\blam_{t}}]
\]
 holds for all $t=0,\ldots,T-1$ with probability at least $1-\delta$.

Now $\frac{1+\ve}{1-\ve}-1=\frac{2\ve}{1-\ve}$ and $1-\frac{1-\ve}{1+\ve}=\frac{2\ve}{1+\ve}\leq\frac{2\ve}{1-\ve}$,
so by subtracting $\mathrm{diag}[X_{\beta,\blam_{t}}]$ from all sides
we find that 
\[
\left|\mathrm{diag}[\hat{X}_{t}]-\mathrm{diag}[X_{\beta,\blam_{t}}]\right|\leq\frac{2\ve}{1-\ve}\,\mathrm{diag}[X_{\beta,\blam_{t}}].
\]
 Therefore 
\begin{align*}
\Vert g_{t}-\nabla f_{\beta}(\blam_{t})\Vert_{1} & =\Vert\mathrm{diag}[\hat{X}_{t}]-\mathrm{diag}[X_{\beta,\blam_{t}}]\Vert_{1}\\
 & \leq\frac{2\ve}{1-\ve}\Vert\mathrm{diag}[X_{\beta,\blam_{t}}]\Vert_{1}\\
 & =\frac{2\ve}{1-\ve}\\
 & \leq\gamma,
\end{align*}
 where we have used the fact that $\ve\in(0,1/2]$, so that $1-\ve\geq1/2$,
as well as the fact that $\gamma=4\ve$. This concludes the proof.
\end{proof}
\begin{center}------------------------------------------------------------------------\end{center}
\begin{proof}
[Proof of Corollary \ref{cor:strongpermsynchgradest}.] Let $\ve=\gamma/8\in(0,1/2)$.
Then the conclusion of Lemma (\ref{lem:concentration1}), together
with the union bound, implies in particular that 
\[
(1-\ve)\Tr[Y_{\beta,\blam_{t}}]\leq\Tr[\hat{Y}_{t}]\leq(1+\ve)\,\Tr[Y_{\beta,\blam_{t}}]
\]
holds for all $t=0,\ldots,T-1$ with probability at least $1-\delta$.
Therefore, since $\hat{X}_{t}=\hat{Y}_{t}/\Tr[\hat{Y}_{t}]$ and $X_{\beta,\blam}=Y_{\beta,\blam}/\Tr[Y_{\beta,\blam}]$,
it follows that 
\[
\frac{1-\ve}{1+\ve}\,X_{\beta,\blam_{t}}^{(i,i)}\preceq\hat{X}_{t}^{(i,i)}\preceq\frac{1+\ve}{1-\ve}\,X_{\beta,\blam_{t}}^{(i,i)}
\]
 holds for all $t=0,\ldots,T-1$ with probability at least $1-\delta$.

Then subtracting $X_{\beta,\blam_{t}}^{(i,i)}$ from all sides (with
similar reasoning as in the proof of Corollary \ref{cor:maxcutgradest}),
we deduce that 
\[
-\frac{2\ve}{1-\ve}X_{\beta,\blam_{t}}^{(i,i)}\preceq\hat{X}_{t}^{(i,i)}-X_{\beta,\blam_{t}}^{(i,i)}\preceq\frac{2\ve}{1-\ve}X_{\beta,\blam_{t}}^{(i,i)}.
\]

We claim that it is a general fact that if $-B\preceq A\preceq B$
where $B\succeq0$, then 
\[
\Vert A\Vert_{\Tr}\leq2\,\Tr[B].
\]
 To see this, note that we can write 
\begin{align*}
\Vert A\Vert_{\Tr} & =\left\langle A,\mathrm{sign}(A)\right\rangle \\
 & =\left\langle A,\Theta(A)\right\rangle +\left\langle -A,\Theta(-A)\right\rangle ,
\end{align*}
 where $\Theta$ denotes the Heaviside function and angle brackets
denote the Frobenius inner product. Then since $A\preceq B$ and $\Theta(A)\succeq0$,
it follows that 
\[
\left\langle A,\Theta(A)\right\rangle \leq\left\langle B,\Theta(A)\right\rangle \leq\Vert B\Vert_{\Tr}\,\Vert\Theta(A)\Vert_{2}\leq\Vert B\Vert_{\Tr},
\]
 where we have used the H\"{o}lder inequality for Schatten norms. Similarly,
since $-A\preceq B$ and $\Theta(-A)\succeq0$, we also have that
\[
\left\langle -A,\Theta(-A)\right\rangle \leq\Vert B\Vert_{\Tr},
\]
 and the claim follows, since $\Vert B\Vert_{\Tr}=\Tr[B]$.

Then by applying our general fact, we deduce that 
\[
\Vert\hat{X}_{t}^{(i,i)}-X_{\beta,\blam_{t}}^{(i,i)}\Vert_{\Tr}\leq\frac{4\ve}{1-\ve}\Tr[X_{\beta,\blam_{t}}^{(i,i)}].
\]
 Therefore 
\begin{align*}
\Vert g_{t}-\nabla f_{\beta}(\blam_{t})\Vert_{*} & =\sum_{i=1}^{N}\Vert\hat{X}_{t}^{(i,i)}-X_{\beta,\blam_{t}}^{(i,i)}\Vert_{\Tr}\\
 & \leq\frac{4\ve}{1-\ve}\sum_{i=1}^{N}\Tr[X_{\beta,\blam_{t}}^{(i,i)}]\\
 & =\frac{4\ve}{1-\ve}\Tr[X_{\beta,\blam_{t}}]\\
 & =\frac{4\ve}{1-\ve}\\
 & \leq\gamma,
\end{align*}
 where we have used the fact that $\ve\in(0,1/2]$, so that $1-\ve\geq1/2$,
as well as the fact that $\gamma=8\ve$. This concludes the proof.
\end{proof}
\begin{center}------------------------------------------------------------------------\end{center}
\begin{proof}
[Proof of Corollary \ref{cor:strongpermsynchgradest}.] Let $\ve=\gamma/4\in(0,1/2]$.
Following the proof of Corollary \ref{cor:maxcutgradest} and making
suitable adjustments, we again conclude that 
\[
\Vert g_{t}-\nabla_{\blam}f_{\beta}(\blam_{t},\bmu_{t})\Vert_{1}\leq\gamma,
\]
 and moreover that 
\[
\left|\frac{\mathbf{1}_{K}^{\top}\hat{X}_{t}^{(i,i)}\mathbf{1}_{K}}{K}-\frac{\mathbf{1}_{K}^{\top}X_{\beta,\blam_{t}}^{(i,i)}\mathbf{1}_{K}}{K}\right|\leq\frac{2\ve}{1-\ve}\,\frac{\mathbf{1}_{K}^{\top}X_{\beta,\blam_{t}}^{(i,i)}\mathbf{1}_{K}}{K}
\]
 for all $i=1,\ldots,N$ and all $t=0,\ldots,T-1$, with probability
at least $1-\delta$.

Using the facts that $\gamma=4\ve$ and $1-\ve\geq1/2$, as well as
the fact that 
\begin{align*}
\frac{\mathbf{1}_{K}^{\top}X_{\beta,\blam_{t}}\mathbf{1}_{K}}{K} & =\Tr\left[X_{\beta,\blam_{t}}\frac{\mathbf{1}_{K}\mathbf{1}_{K}^{\top}}{K}\right]\\
 & \leq\Vert X_{\beta,\blam_{t}}\Vert_{\Tr}\left\Vert \frac{\mathbf{1}_{K}\mathbf{1}_{K}^{\top}}{K}\right\Vert _{2}\\
 & =\Tr\left[X_{\beta,\blam_{t}}\right],
\end{align*}
 where we have used the H\"{o}lder inequality for Schatten norms, we conclude
that 
\[
\left|\frac{\mathbf{1}_{K}^{\top}\hat{X}_{t}^{(i,i)}\mathbf{1}_{K}}{K}-\frac{\mathbf{1}_{K}^{\top}X_{\beta,\blam_{t}}^{(i,i)}\mathbf{1}_{K}}{K}\right|\leq\gamma\,\Tr[X_{\beta,\blam_{t}}^{(i,i)}].
\]
 It follows that 
\begin{align*}
\Vert h_{t}-\nabla_{\bmu}f_{\beta}(\blam_{t},\bmu_{t})\Vert_{1} & \leq\gamma\sum_{i=1}^{N}\Tr[X_{\beta,\blam_{t}}^{(i,i)}]\\
 & =\gamma\,\Tr[X_{\beta,\blam_{t}}]\\
 & =\gamma.
\end{align*}
 Therefore 
\[
\Vert(g_{t},h_{t})-\nabla f_{\beta}(\blam_{t},\bmu_{t})\Vert_{*}=\max\left(\Vert g_{t}-\nabla_{\blam}f_{\beta}(\blam_{t},\bmu_{t})\Vert_{1},\Vert h_{t}-\nabla_{\bmu}f_{\beta}(\blam_{t},\bmu_{t})\Vert_{1}\right)\leq\gamma,
\]
 which concludes the proof.
\end{proof}

\section{Proofs for gradient convergence \label{app:gradconv}}

Before giving the proof of Theorem \ref{thm:gradconv}, we state and
prove two useful lemmas.
\begin{lem}
\label{lem:norm} For any $\alpha>0$ and $x\in\R^{n}$, we have 
\[
\min_{y\in\R^{n}}\left\{ x\cdot y+\frac{\alpha}{2}\Vert y\Vert^{2}\right\} =-\frac{1}{2\alpha}\Vert x\Vert_{*}^{2}.
\]
 Moreover, any $y$ attaining the minimum satisfies $\Vert y\Vert=\Vert x\Vert_{*}/\alpha$.
\end{lem}

\begin{proof}
This statement and proof are largely reproduced from Section 10.9
in \cite{beck2017first}. We only need a little extra argument to
recover the last part of the statement of the lemma, which does not
appear in \cite{beck2017first}.

Let $z\in\R^{n}$ be a vector such that $\Vert z\Vert\leq1$ and $x\cdot z=\Vert x\Vert_{*}$,
which we know exists by the defining property of the dual norm. Then
letting $y=-\frac{\Vert x\Vert_{*}}{\alpha}z$, it follows that $x\cdot y+\frac{\alpha}{2}\Vert y\Vert^{2}=-\frac{1}{2\alpha}\Vert y\Vert_{*}^{2}$.
This implies that 
\[
\min_{y\in\R^{n}}\left\{ x\cdot y+\frac{\alpha}{2}\Vert y\Vert^{2}\right\} \leq-\frac{1}{2\alpha}\Vert x\Vert_{*}^{2}.
\]
 Meanwhile, we have for any $y\in\R^{n}$: 
\begin{equation}
x\cdot y+\frac{\alpha}{2}\Vert y\Vert^{2}\geq-\Vert x\Vert_{*}\Vert y\Vert+\frac{\alpha}{2}\Vert y\Vert^{2}\geq\min_{t\in\R}\left\{ -t\Vert x\Vert_{*}+\frac{\alpha}{2}t^{2}\right\} =-\frac{1}{2\alpha}\Vert x\Vert_{*}^{2},\label{eq:normineq}
\end{equation}
 so 
\[
\min_{y\in\R^{n}}\left\{ x\cdot y+\frac{\alpha}{2}\Vert y\Vert^{2}\right\} \geq-\frac{1}{2\alpha}\Vert x\Vert_{*}^{2}.
\]

Moreover, for any $y$ that satisfies $x\cdot y+\frac{\alpha}{2}\Vert y\Vert^{2}=-\frac{1}{2\alpha}\Vert x\Vert_{*}^{2}$,
the chain of inequalities (\ref{eq:normineq}) implies that 
\[
\Vert y\Vert=\underset{t\in\R}{\text{argmin}}\left\{ -t\Vert x\Vert_{*}+\frac{\alpha}{2}t^{2}\right\} =\frac{\Vert x\Vert_{*}}{\alpha}.
\]
 This completes the proof.
\end{proof}
The second useful lemma bounds the error in the objective of our initial
guess.
\begin{lem}
\label{lem:guessobj}$f_{\beta}(\blam_{0})+p_{\star,\beta}\leq\Delta e(C)+\beta^{-1}\log n$.
\end{lem}

\begin{proof}
Recall that $p_{\star,\beta}$ is the optimal value of (\ref{eq:reg_sdp}).
Therefore 
\[
p_{\star,\beta}=\Tr[CX_{\star,\beta}]+\beta^{-1}S(X_{\star,\beta})\leq\Tr[CX_{\star,\beta}],
\]
 where we have used the fact that $S(X)\leq0$ for $X\in\mathcal{P}_{1}$.
Now by considering a spectral decomposition $X_{\star,\beta}=\sum_{i=1}^{n}x_{i}u_{i}u_{i}^{\top}$
where $\sum_{i=1}^{n}x_{i}=1$ and $x\geq0$, we can see that $\Tr[CX_{\star,\beta}]\leq e_{\max}(C)$.
Therefore 
\[
p_{\star,\beta}\leq e_{\max}(C).
\]

Meanwhile, since $\blam_{0}=0$, we have 
\[
f_{\beta}(\blam_{0})=\beta^{-1}\log\Tr\left[e^{-\beta C}\right].
\]
 Now 
\[
\Tr\left[e^{-\beta C}\right]\leq n e_{\max}\left[e^{-\beta C}\right]=ne^{-\beta e_{\min}(C)}
\]
 so 
\[
f_{\beta}(\blam_{0})\leq-e_{\min}(C)+\beta^{-1}\log n.
\]

The lemma then follows.
\end{proof}
Finally we proceed with the proof of Theorem \ref{thm:gradconv},
which we restate here.

\gradconvthm*
\begin{proof}
From (\ref{eq:optstep}) and Lemma \ref{lem:norm} we know that 
\begin{equation}
\Vert\blam_{t+1}-\blam_{t}\Vert\leq\beta^{-1}\Vert g_{t}\Vert_{*}.\label{eq:steplambda}
\end{equation}
 Next, compute: 
\begin{align*}
f_{\beta}(\blam_{t+1})\overset{\mathrm{(i)}}{\leq}\  & f_{\beta}(\blam_{t})+\nabla f_{\beta}(\blam_{t})\cdot(\blam_{t+1}-\blam_{t})+\frac{\beta}{2}\Vert\blam_{t+1}-\blam_{t}\Vert^{2}\\
=\,\  & f_{\beta}(\blam_{t})+g_{t}\cdot(\blam_{t+1}-\blam_{t})+\frac{\beta}{2}\Vert\blam_{t+1}-\blam_{t}\Vert^{2}+\left(\nabla f_{\beta}(\blam_{t})-g_{t}\right)\cdot(\blam_{t+1}-\blam_{t})\\
\overset{\mathrm{(ii)}}{\leq}\  & \min_{\blam}\left\{ f_{\beta}(\blam_{t})+g_{t}\cdot(\blam-\blam_{t})+\frac{\beta}{2}\Vert\blam-\blam_{t}\Vert^{2}\right\} +\Vert\nabla f_{\beta}(\blam_{t})-g_{t}\Vert_{*}\,\Vert\blam_{t+1}-\blam_{t}\Vert\\
\overset{\mathrm{(iii)}}{\leq}\  & \left[f_{\beta}(\blam_{t})-\frac{1}{2\beta}\Vert g_{t}\Vert_{*}^{2}\right]+\frac{\gamma}{\beta}\,\Vert g_{t}\Vert_{*}\\
\overset{\mathrm{(iv)}}{\leq}\  & \left[f_{\beta}(\blam_{t})-\frac{1}{2\beta}\Vert g_{t}\Vert_{*}^{2}\right]+\frac{\gamma^{2}}{\beta}+\frac{1}{4\beta}\Vert g_{t}\Vert_{*}^{2}\\
=\,\  & f_{\beta}(\blam_{t})-\frac{1}{4\beta}\Vert g_{t}\Vert_{*}^{2}+\frac{\gamma^{2}}{\beta}.
\end{align*}
 In (i), we have used strong smoothness. In (ii), we have used the
defining property of the update (\ref{eq:optstep}) and the Cauchy-Schwarz
inequality for general norms. In (iii), we have used Lemma \ref{lem:norm}
to bound the first term, as well as Assumption \ref{assump:gradientestimator}
and (\ref{eq:steplambda}) to bound the second term. In step (iv),
we have used Young's inequality $ab\leq a^{2}+\frac{1}{4}b^{2}$.

In summary, we have shown (after some rearrangement): 
\[
\Vert g_{t}\Vert_{*}^{2}\leq4\beta\left[f_{\beta}(\blam_{t})-f_{\beta}(\blam_{t+1})\right]+4\gamma^{2}.
\]
 Then sum both sides from $t=0,\ldots,T-1$ to obtain (via telescoping):
\[
\sum_{t=0}^{T-1}\Vert g_{t}\Vert_{*}^{2}\leq4\beta\left[f_{\beta}(\blam_{0})-f_{\beta}(\blam_{T})\right]+4T\gamma^{2}.
\]
 Of course, $-f_{\beta}(\blam_{T})\leq p_{\star,\beta}$ by strong
duality, so in turn: 
\[
\frac{1}{T}\sum_{t=0}^{T-1}\Vert g_{t}\Vert_{*}^{2}\leq\frac{4\beta}{T}\left[f_{\beta}(\blam_{0})+p_{\star,\beta}\right]+4\gamma^{2}.
\]

It follows that there must exist $t\in\{0,\ldots,T-1\}$ such that
\[
\Vert g_{t}\Vert_{*}\leq\sqrt{\frac{4\beta}{T}\left[f_{\beta}(\blam_{0})+p_{\star,\beta}\right]+4\gamma^{2}}.
\]
 But by Assumption \ref{assump:gradientestimator}, $\Vert g_{t}\Vert_{*}\geq\Vert\nabla f_{\beta}(\blam_{t})\Vert_{*}-\gamma$,
so it follows that 
\[
\Vert\nabla f_{\beta}(\blam_{t})\Vert_{*}\leq\gamma+\sqrt{\frac{4\beta}{T}\left[f_{\beta}(\blam_{0})+p_{\star,\beta}\right]+4\gamma^{2}}.
\]
 Then by Lemma \ref{lem:guessobj}, we have 
\[
\Vert\nabla f_{\beta}(\blam_{t})\Vert_{*}\leq\gamma+\sqrt{\frac{4\beta\,\Delta e(C)}{T}+\frac{4\,\log n}{T}+4\gamma^{2}}.
\]
Then using the general inequality $\sqrt{a+b+c}\leq\sqrt{a}+\sqrt{b}+\sqrt{c}$
for $a,b,c\geq0$, the proof of the theorem is completed.
\end{proof}

\section{Proofs for objective convergence \label{app:lossconv}}

Note that our proof of objective convergence with exact gradients
(Theorem \ref{thm:objconv}) is not original. It can be recovered
from Fact B.2 in \cite{allen2014linear} or Section 10 of \cite{beck2017first}.
On the other hand, our proof for inexact gradients (Theorem \ref{thm:objconv-1})
does not appear in any previous literature, to the best of our knowledge.
A key assumption is that the error in the gradient is bounded.

Before we present the proofs of Theorem \ref{thm:objconv} and Theorem
\ref{thm:objconv-1}, we state a key lemma. 
\begin{lem}
\label{lem:lossdecrease}Given a norm $\|\cdot\|$, let $f$ be convex
and $L$-smooth with respect to this norm, and suppose that $x_{\star}$
is a minimizer of $f$. Then 
\begin{equation}
\min_{y}\left\{ f(y)+\frac{L\|x-y\|^{2}}{2}\right\} \le f(x)-\frac{\big[f(x)-f(x_{\star})\big]^{2}}{2L\|x-x_{\star}\|^{2}}.\label{eq:minexp}
\end{equation}
\end{lem}

\begin{proof}
Let $z_{s}=x+s(x_{\star}-x)$ where $s\in[0,1]$. We will plug $z_{s}$
in the place of $y$ in the minimized expression on the left-hand
side of (\ref{eq:minexp}) to obtain an upper bound.

Compute: 
\begin{align*}
f(z_{s})+\frac{L\|x-z_{s}\|^{2}}{2}\ = & \ f((1-s)x+sx_{\star})+\frac{Ls^{2}\|x-x_{\star}\|^{2}}{2}\\
\leq & \ f(x)+s(f(x_{*})-f(x))+\frac{Ls^{2}\|x-x_{\star}\|^{2}}{2}\\
= & \ f(x)+\frac{L\|x-x_{\star}\|^{2}}{2}\Bigg[s-\frac{f(x)-f(x_{\star})}{L\|x-x_{\star}\|^{2}}\Bigg]^{2}-\frac{\big[f(x)-f(x_{\star})\big]^{2}}{2L\|x-x_{\star}\|^{2}}.
\end{align*}
Note that since $f$ is $L$-smooth and $x_{\star}$ is the minimizer,
we must have 
\[
0\le\frac{f(x)-f(x_{\star})}{L\|x-x_{\star}\|^{2}}\le\frac{1}{2}.
\]
Therefore, we can take $s=\frac{f(x)-f(x_{x})}{L\|x-x_{\star}\|^{2}}$,
yielding 
\[
f(z_{s})+\frac{L\|x-z_{s}\|^{2}}{2}=f(x)-\frac{\big[f(x)-f(x_{\star})\big]^{2}}{2L\|x-x_{\star}\|^{2}}.
\]
 By plugging $z_{s}$ in the place of $y$ in the minimized expression
on the left-hand side of (\ref{eq:minexp}), the proof is completed.
\end{proof}
Now we proceed with the proof of Theorem \ref{thm:objconv}, which
we restate here.

\objconvthm*
\begin{proof}
Recall that 
\begin{equation}
\blam_{t+1}=\underset{\blam}{\text{argmin}}\left\{ \langle\nabla f_{\beta}(\blam_{t}),\blam-\blam_{t}\rangle+\frac{\|\blam-\blam_{t}\|^{2}}{2\eta}\right\} .\label{eq:blamargmin}
\end{equation}
Now for any $y\in\R^{m}$, we can bound: 
\begin{align*}
f(\blam_{t+1})\ \leq & \ f_{\beta}(\blam_{t})+\langle\nabla f_{\beta}(\blam_{t}),\blam_{t+1}-\blam_{t}\rangle+\frac{\beta\|\blam_{t+1}-\blam_{t}\|^{2}}{2}\\
\le & \ f_{\beta}(\blam_{t})+\langle\nabla f_{\beta}(\blam_{t}),y-\blam_{t}\rangle+\frac{\|y-\blam_{t}\|^{2}}{2\eta}\\
\le & \ f_{\beta}(y)+\frac{\|\blam_{t}-y\|^{2}}{2\eta}.
\end{align*}
Here the first inequality is due to the $\beta$-smoothness of $f_{\beta}$,
the second inequality follows from (\ref{eq:blamargmin}), and the
third inequality follows from the convexity of $f_{\beta}$.

Now since $y$ was arbitrary, we have 
\[
f(\blam_{t+1})\leq\min_{y}\left\{ f_{\beta}(y)+\frac{\|\blam_{t}-y\|^{2}}{2\eta}\right\} .
\]
 Moreover, since $f_{\beta}$ is $\beta$-smooth and $\eta\leq1/\beta$,
then $f_{\beta}$ is also $(1/\eta)$-smooth. Therefore Lemma \ref{lem:lossdecrease}
implies that 
\[
f(\blam_{t+1})\leq f_{\beta}(\blam_{t})-\frac{\eta\big[f_{\beta}(\blam_{t})-f_{\beta}(\blam_{\star})\big]^{2}}{2\|\blam_{t}-\blam_{\star}\|^{2}}.
\]

Then if we define the objective error
\[
\epsilon_{t}\coloneqq f(\blam_{t})-f(\blam_{\star})
\]
 for all $t$, it follows that 
\[
\ve_{t+1}\leq\ve_{t}-\frac{\eta\ve_{t}^{2}}{2D}
\]
for all $t\geq0$.

In particular, the sequence $\{\epsilon_{t}\}_{t=0}^{\infty}$ is
monotonically decreasing, so $\ve_{t}^{2}\geq\ve_{t}\ve_{t+1}$, and
in turn 
\[
\epsilon_{t+1}\leq\epsilon_{t}-\frac{\eta\epsilon_{t}\epsilon_{t+1}}{2D}.
\]
Dividing both sides by $\epsilon_{t}\epsilon_{t+1}$, we obtain 
\[
\frac{1}{\epsilon_{t}}\le\frac{1}{\epsilon_{t+1}}-\frac{\eta}{2D}.
\]
 Using telescoping sums, it therefore follows that 
\[
\frac{1}{\epsilon_{t}}\ge\frac{1}{\epsilon_{0}}+\frac{t\eta}{2D}\geq\frac{t\eta}{2D}
\]
 for all $t\geq0$. In turn we deduce that 
\[
\epsilon_{t}\le\frac{2D}{t\eta}.
\]
This concludes the proof.
\end{proof}
Finally, we present the proof of Theorem \ref{thm:objconv-1}.

\noiseobjconvthm*
\begin{proof}
Recall that 
\begin{equation}
\blam_{t+1}=\underset{\blam}{\text{argmin}}\left\{ \langle g_{t},\blam-\blam_{t}\rangle+\frac{\|\blam-\blam_{t}\|^{2}}{2\eta}\right\} .\label{eq:blamargmin2}
\end{equation}
 Recall that by assumption $\eta\leq\frac{1}{2\beta}$. Then define
\[
s\coloneqq\eta^{-1}-\beta>\beta,
\]
 so that $\beta+s=\eta^{-1}$. Also define the gradient error vector
\[
\delta_{t}\coloneqq\nabla f_{\beta}(\blam_{t})-g_{t},
\]
 so $\Vert\delta_{t}\Vert_{*}\leq\gamma$.

Then for any $y$, we can bound: 
\begin{align*}
f_{\beta}(\blam_{t+1})\ \ \overset{\mathrm{(i)}}{\leq} & \ \ f(\blam_{t})+\langle\nabla f_{\beta}(\blam_{t}),\blam_{t+1}-\blam_{t}\rangle+\frac{\beta\|\blam_{t+1}-\blam_{t}\|^{2}}{2}\\
= & \ \ f(\blam_{t})+\langle g_{t},\blam_{t+1}-\blam_{t}\rangle+\frac{\beta\|\blam_{t+1}-\blam_{t}\|^{2}}{2}+\langle\delta_{t},\blam_{t+1}-\blam_{t}\rangle\\
\overset{\mathrm{(ii)}}{\leq} & \ \ f(\blam_{t})+\langle g_{t},\blam_{t+1}-\blam_{t}\rangle+\frac{\text{(}\beta+s)\|\blam_{t+1}-\blam_{t}\|^{2}}{2}+\frac{\gamma^{2}}{2s}\\
\overset{\mathrm{(iii)}}{\leq} & \ \ f(\blam_{t})+\langle g_{t},y-\blam_{t}\rangle+\frac{\text{(}\beta+s)\|y-\blam_{t}\|^{2}}{2}+\frac{\gamma^{2}}{2s}\\
= & \ \ f(\blam_{t})+\langle\nabla f_{\beta}(\blam_{t}),y-\blam_{t}\rangle+\frac{\text{(}\beta+s)\|y-\blam_{t}\|}{2}+\frac{\gamma^{2}}{2s}+\langle-\delta_{t},y-\blam_{t}\rangle\\
\overset{\mathrm{(iv)}}{\leq} & \ \ f(\blam_{t})+\langle\nabla f_{\beta}(\blam_{t}),y-\blam_{t}\rangle+\frac{\text{(}\beta+2s)\|y-\blam_{t}\|^{2}}{2}+\frac{\gamma^{2}}{s}\\
\overset{\mathrm{(v)}}{\leq} & \ \ f(y)+\frac{\text{(}\beta+2s)\|\blam_{t}-y\|^{2}}{2}+\frac{\gamma^{2}}{s}.
\end{align*}
Here (i) is due to the $L$-smoothness of $f$, (ii) follows from
the bound 
\[
\langle\delta_{t},\blam_{t+1}-\blam_{t}\rangle\le\|\delta_{t}\|_{*}\|\blam_{t+1}-\blam_{t}\|\leq\gamma\|\blam_{t+1}-\blam_{t}\|\le\frac{s}{2}\|\blam_{t+1}-\blam_{t}\|^{2}+\frac{\gamma^{2}}{2s},
\]
 (iii) follows from (\ref{eq:blamargmin2}) (noting that $\beta+s=\eta^{-1}$),
(iv) follows---similarly to (ii)---from the bound 
\[
\langle-\delta_{t},y-\blam_{t}\rangle\leq\frac{s}{2}\|y-\blam_{t}\|^{2}+\frac{\gamma^{2}}{2s},
\]
 and (v) is due to the convexity of $f_{\beta}$.

Now since $y$ was arbitrary, we have 
\[
f(\blam_{t+1})\leq\min_{y}\left\{ f(y)+\frac{\text{(}\beta+2s)\|\blam_{t}-y\|^{2}}{2}\right\} +\frac{\gamma^{2}}{s}.
\]
 Moreover, since $f_{\beta}$ is $\beta$-smooth, it is also $(\beta+2s)$-smooth,
and Lemma \ref{lem:lossdecrease} then implies that 
\begin{align*}
f_{\beta}(\blam_{t+1}) & \le f(\blam_{t})-\frac{\big[f(\blam_{t})-f(\blam_{\star})\big]^{2}}{2(\beta+2s)\|\blam_{t}-\blam_{\star}\|^{2}}+\frac{\gamma^{2}}{s}.
\end{align*}

Then if we define the objective error
\[
\epsilon_{t}\coloneqq f(\blam_{t})-f(\blam_{\star})
\]
 for all $t$, it follows that 
\[
\ve_{t+1}\leq\ve_{t}-\frac{\ve_{t}^{2}}{2(\beta+2s)D}+\frac{\gamma^{2}}{s}
\]
for all $t\geq0$.

To simplify our notation, define
\[
a^{2}\coloneqq2(\beta+2s)D,\qquad b^{2}\coloneqq\gamma^{2}/s,
\]
 so that 
\[
\epsilon_{t+1}\le\epsilon_{t}-\frac{\epsilon_{t}^{2}}{a^{2}}+b^{2}.
\]
 In fact, it is equivalent that 
\[
(\epsilon_{t+1}-ab)\le(1-2b/a)(\epsilon_{t}-ab)-\frac{(\epsilon_{t}-ab)^{2}}{a^{2}},
\]
 as can be verified by expanding the right-hand side and performing
elementary simplifications.

By dropping the negative term on the right-hand side, we obtain the
looser bound
\begin{equation}
(\epsilon_{t+1}-ab)\le(1-2b/a)(\epsilon_{t}-ab).\label{eq:decay}
\end{equation}
 Recall that $s>\beta$ (from which it follows that $a^{2}\geq6\beta D$)
and that by assumption $\gamma\leq\sqrt{3\beta^{2}D/2}$, so in turn
\[
2b=\frac{2\gamma}{\sqrt{s}}<\frac{2\gamma}{\sqrt{\beta}}\le\sqrt{6\beta D}\le a.
\]
 Therefore 
\[
0<(1-2b/a)<1,
\]
 i.e., the geometric factor (\ref{eq:decay}) is positive and it follows
from repeated application of (\ref{eq:decay}) that 
\[
(\epsilon_{t}-ab)\le(1-2b/a)^{t}(\epsilon_{0}-ab).
\]
 In turn it follows that 
\[
\epsilon_{t}\le(1-2b/a)^{t}\epsilon_{0}+ab.
\]
Replacing $\epsilon_{t},a$, and $b$ with $f_{\beta}(\blam_{t})-f_{\beta}(\blam_{\star}),\sqrt{2(\beta+2s)D},$
and $\gamma/\sqrt{s}$, respectively, and in turn substituting $s=\eta^{-1}-\beta$,
we get the desired bound, concluding the proof.
\end{proof}

\section{Proofs for OT convergence \label{app:pf-ot}}

Before we present the proof of Theorem \ref{eq:otupdate}, we state
the key lemma. 
\begin{lem}
\label{lem:ot-diff-bound}Let $(\phi_{t},\psi_{t})$, $t=0,\ldots,T$,
be furnished by the algorithm (\ref{eq:otupdate}) with exact gradient
estimation and step size $\eta\leq1/\beta$. Let $M,s>0$ such that
$\min_{i\in[m]}\mu_{i}\geq s$, $\min_{j\in[n]}\nu_{j}\geq s$ and
$\max_{ij}|c_{ij}|\leq M$. Then 
\[
[\phi_{t}]_{i}-[\phi_{t}]_{k}\le2M+\frac{\log(1/s)+2}{\beta}\ \ \ \mathrm{for\,\,all}\,\ i,k\in[m],
\]
and
\[
[\psi_{t}]_{j}-[\psi_{t}]_{l}\le2M+\frac{\log(1/s)+2}{\beta},\ \ \mathrm{for\,\,all}\,\ j,l\in[n].
\]
\end{lem}

\begin{proof}
Without loss of generality, we will only present the proof for $\phi_{t}$.
Then fix $i,k\in[m]$. Our intuition is based on studying what happens
if $[\phi_{t}]_{i}-[\phi_{t}]_{k}$becomes `large enough.' The proof
is divided into two steps.
\begin{enumerate}
\item In Step 1, we show that if $[\phi_{t}]_{i}-[\phi_{t}]_{k}\ge2M+\frac{\log(1/s)}{\beta}$,
then $[\phi_{t+1}]_{i}-[\phi_{t+1}]_{k}\le[\phi_{t}]_{i}-[\phi_{t}]_{k}$.
\item In Step 2, we show that the desired result follows (using the fact
that $\phi_{0}=0$ and $\psi_{0}=0$).
\end{enumerate}
\paragraph{Step 1.} Consider the marginal $\mu_{t}=\pi_{\beta,\blam_{t}}\mathbf{1}_{n}$
appearing in (\ref{eq:otupdate}) via the difference $\mu_{t}-\mu$
(which is the dual gradient with respect to $\phi$). We can compute
elementwise: 
\begin{align}
\left[\mu_{t}-\mu\right]_{k}\ \ = & \ \ \frac{\sum_{j=1}^{n}\exp\left[\beta\Big([\phi_{t}]_{k}+[\psi_{t}]_{j}-c_{kj}\Big)\right]}{\sum_{i'=1}^{m}\sum_{j=1}^{n}\exp\left[\beta\Big([\phi_{t}]_{i'}+[\psi_{t}]_{j}-c_{i'j}\Big)\right]}-\mu_{k}\nonumber \\
\overset{\mathrm{(i)}}{<} & \ \ \frac{\sum_{j=1}^{n}\exp\left[\beta\Big([\phi_{t}]_{k}+[\psi_{t}]_{j}-c_{kj}\Big)\right]}{\sum_{j=1}^{n}\exp\left[\beta\Big([\phi_{t}]_{i}+[\psi_{t}]_{j}-c_{ij}\Big)\right]}-\mu_{k}\nonumber \\
\overset{\mathrm{(ii)}}{\leq} & \ \ \max_{j=1,\ldots,n}\left\{ \frac{\exp\left[\beta\Big([\phi_{t}]_{k}+[\psi_{t}]_{j}-c_{kj}\Big)\right]}{\exp\left[\beta\Big([\phi_{t}]_{i}+[\psi_{t}]_{j}-c_{ij}\Big)\right]}\right\} -\mu_{k}\nonumber \\
= & \ \ \max_{j=1,\ldots,n}\left\{ \exp\left[\beta\Big([\phi_{t}]_{k}-[\phi_{t}]_{i}-[c_{kj}-c_{ij}]\Big)\right]\right\} -\mu_{k}\nonumber \\
\leq & \ \ \exp\left[\beta\Big([\phi_{t}]_{k}-[\phi_{t}]_{i}+2M\Big)\right]-s.\label{eq:gradineq}
\end{align}
Here, (i) follows from discarding positive terms in the denominator,
and (ii) follows from the general inequality 
\[
\frac{\sum_{j=1}^{n}a_{j}}{\sum_{j=1}^{n}b_{j}}\le\max_{j=1,\ldots,n}\left\{ \frac{a_{j}}{b_{j}}\right\} 
\]
 which holds for all $a_{j},b_{j}>0$, $j=1,\ldots,n$. (To prove
this general inequality, let $R$ denote the right-hand side, so $a_{j}/b_{j}\leq R$
for all $j$. Then in turn $a_{j}\leq Rb_{j}$ for all $j$. Summing
over $j$ proves the desired inequality.)

Therefore, (\ref{eq:gradineq}) implies that if $[\phi_{t}]_{i}-[\phi_{t}]_{k}\ge2M+\frac{\log(1/s)}{\beta}$,
then 
\[
\left[\mu_{t}-\mu\right]_{k}<0.
\]
As a consequence, from the algorithm (\ref{eq:otupdate}), we have
\[
[\tilde{\phi}_{t+1}]_{k}=[\phi_{t}]_{k}+\frac{\eta}{2}\|\mu_{t}-\mu\|_{1}.
\]
 Meanwhile, 
\[
[\tilde{\phi}_{t+1}]_{i}\le[\phi_{t}]_{i}+\frac{\eta}{2}\|\mu_{t}-\mu\|_{1}.
\]
Combining these facts, we deduce that 
\[
[\phi_{t+1}]_{i}-[\phi_{t+1}]_{k}=[\tilde{\phi}_{t+1}]_{i}-[\tilde{\phi}_{t+1}]_{k}\le[\phi_{t}]_{i}-[\phi_{t}]_{k},
\]
 as was to be shown.

\paragraph{Step 2.} We will complete the proof by induction on $t$.
The base case $t=0$ is true by the choice of initial condition. Then
assume the inductive hypothesis that $[\phi_{t}]_{i}-[\phi_{t}]_{k}\le2M+\frac{\log(1/s)+2}{\beta}$.
We want to show that likewise $[\phi_{t+1}]_{i}-[\phi_{t+1}]_{k}\le2M+\frac{\log(1/s)+2}{\beta}$.
There are two cases to consider.

\vspace{4mm}

\uline{Case 1}: $[\phi_{t}]_{i}-[\phi_{t}]_{k}\le2M+\frac{\log(1/s)}{\beta}$.

\vspace{2mm}

Note that 
\[
\Vert\mu-\mu_{t}\Vert_{1}\leq\Vert\mu\Vert_{1}+\Vert\mu_{t}\Vert_{1}=2
\]
 for all $t$. Then we can bound (using the update rule (\ref{eq:otupdate})):
\begin{align*}
[\phi_{t+1}]_{i}-[\phi_{t+1}]_{k} & =\underbrace{[\phi_{t}]_{i}-[\phi_{t}]_{k}}_{\leq\ 2M+\frac{\log(1/s)}{\beta}}\ +\ \frac{\eta}{2}\underbrace{\|\mu_{t}-\mu\|_{1}}_{\leq\ 2}\,\underbrace{\left[\sign\left([\mu-\mu_{t}]_{i}\right)-\sign\left([\mu-\mu_{t}]_{k}\right)\right]}_{\leq\ 2},
\end{align*}
 so 
\[
[\phi_{t+1}]_{i}-[\phi_{t+1}]_{k}\leq2M+\frac{\log(1/s)}{\beta}+2\eta.
\]
 By assumption, $\eta\le1/\beta$, so 
\[
[\phi_{t+1}]_{i}-[\phi_{t+1}]_{k}\leq2M+\frac{\log(1/s)+2}{\beta},
\]
 which completes the induction in this case.

\vspace{4mm}

\uline{Case }2: $[\phi_{t}]_{i}-[\phi_{t}]_{k}>2M+\frac{\log(1/s)}{\beta}$.

\vspace{2mm}

In this case, Step 1 implies that 
\[
[\phi_{t+1}]_{i}-[\phi_{t+1}]_{k}\le[\phi_{t}]_{i}-[\phi_{t}]_{k}\le2M+\frac{\log(1/s)+2}{\beta}.
\]
This concludes the proof.
\end{proof}
Now we proceed with the proof of Theorem \ref{thm:otconv}, which
we restate here.

\otconvthm*
\begin{proof}
Note that by the shifting step in (\ref{eq:otupdate}), we have ensured
that $\mathbf{1}_{m}^{\top}\phi_{t}=0$ for all $t$. Therefore, there
exist $i_{1,}i_{2}\in[m]$, such that $[\phi_{t}]_{i_{1}}\ge0$ and
$[\phi_{t}]_{i_{2}}\le0$.

For any $i,$ we have by Lemma \ref{lem:ot-diff-bound} that 
\[
[\phi_{t}]_{i}-[\phi_{t}]_{i_{2}}\le2M+\frac{\log(1/s)+2}{\beta},
\]
 hence 
\[
[\phi_{t}]_{i}\leq2M+\frac{\log(1/s)+2}{\beta}.
\]
 Similarly, we can prove using the index $i_{1}$ that 
\[
-[\phi_{t}]_{i}\leq2M+\frac{\log(1/s)+2}{\beta}.
\]
 Therefore we conclude that 
\[
\|\phi_{t}\|_{\infty},\,\|\psi_{t}\|_{\infty}\le2M+\frac{\log(1/s)+2}{\beta},
\]
 where similar arguments establish the claim for $\psi_{t}$. 

Let $\blam_{\star}=(\phi_{\star},\psi_{\star})$ be a dual optimizer,
and assume that $\mathbf{1}_{m}^{\top}\phi_{\star}=\mathbf{1}_{n}^{\top}\psi_{\star}=0$
by shifting if necessary, which does not change the dual objective.
We can bound $\|\phi_{\star}\|_{\infty}$ and $\|\psi_{\star}\|_{\infty}$
by similar arguments as follows. The dual optimality condition implies
that $\pi_{\beta,(\phi_{\star},\psi_{\star})}\mathbf{1}_{n}=\mu$.
Hence using the same reasoning as in the justification of (\ref{eq:gradineq}),
we deduce that 
\[
0\leq\exp\left[\beta\Big([\phi_{\star}]_{k}-[\phi_{\star}]_{i}+2M\Big)\right]-s,
\]
 for all $i,k\in[m]$, from which it follows that 
\[
[\phi_{\star}]_{k}-[\phi_{\star}]_{i}\leq2M+\frac{\log(1/s)}{\beta}
\]
 for all $i,k\in[m]$. Since $\mathbf{1}_{m}^{\top}\phi_{\star}=0$,
it follows using the same argument as above that 
\[
\Vert\phi_{\star}\Vert_{\infty},\,\|\psi_{\star}\|_{\infty}\le2M+\frac{\log(1/s)}{\beta},
\]
 where similar arguments establish the claim for $\psi_{\star}$.

Then by the triangle inequality, 
\[
\Vert\phi_{\star}-\phi_{t}\Vert_{\infty},\,\|\psi_{t}-\psi_{\star}\|_{\infty}\leq2\left(2M+\frac{\log(1/s)+1}{\beta}\right).
\]
 Then finally we can bound $D$ as follows: 
\begin{align*}
D\ = & \ \sup_{t=0,1,2,\ldots}\left\{ \|(\phi_{t},\psi_{t})-(\phi_{\star},\psi_{\star})\|^{2}\right\} \\
= & \ \sup_{t=0,1,2,\ldots}\left\{ 2\|\phi_{t}-\phi_{\star}\|_{\infty}^{2}+2\|\psi_{t}-\psi_{\star}\|_{\infty}^{2}\right\} \\
 & \le16\left[2M+\frac{\log(1/s)+1}{\beta}\right]^{2}.
\end{align*}
Then the proof is concluded by invoking Theorem \ref{thm:objconv}.
(Note that although the update (\ref{eq:otupdate}) for the OT problem
differs slightly from the general update (\ref{eq:optstep}), due
to the shifts, the proof of Theorem \ref{thm:objconv} still applies
with minor modifications since these shifts do not change the value
of the objective.)
\end{proof}
Now we proceed with the proof of Theorem \ref{thm:otgrad}, which
we restate here.

\otgradthm*
\begin{proof}
Again let $\blam_{\star}=(\phi_{\star},\psi_{\star})$ be a dual optimizer,
and assume that $\mathbf{1}_{m}^{\top}\phi_{\star}=\mathbf{1}_{n}^{\top}\psi_{\star}=0$
by shifting if necessary, which does not change the dual objective.

Let $J:=\lceil T/2\rceil\geq T/2$. Then by Theorem \ref{thm:otconv},
we have that 
\begin{align}
f_{\beta}(\blam_{J})-f_{\beta}(\blam_{\star})\ \leq & \ \ \frac{32\left[2M+\frac{\log(1/s)+1}{\beta}\right]^{2}}{J\eta}\nonumber \\
\leq & \ \ \frac{64\left[2M+\frac{\log(1/s)+1}{\beta}\right]^{2}}{T\eta}.\label{eq:Jbound}
\end{align}

Then we consider the same strategy as in the proof for Theorem \ref{thm:gradconv}
(noting that here $\gamma=0$), which establishes that 
\[
\|\nabla f_{\beta}(\blam_{t})\|_{*}^{2}\leq2\eta^{-1}\left[f_{\beta}(\blam_{t})-f_{\beta}(\blam_{t+1})\right]
\]
 for all $t=0,1,\ldots$. Then via telescoping we deduce that 
\begin{align*}
\sum_{t=J}^{T}\|\nabla f_{\beta}(\blam_{t})\|_{*}^{2}\ \ \leq & \ \ 2\eta^{-1}\left[f_{\beta}(\blam_{J})-f_{\beta}(\blam_{T+1})\right]\\
\leq & \ \ 2\eta^{-1}\left[f_{\beta}(\blam_{J})-f_{\beta}(\blam_{\star})\right]\\
\leq & \ \ \frac{128\left[2M+\frac{\log(1/s)+1}{\beta}\right]^{2}}{T\eta^{2}},
\end{align*}
 where in the last step we have used (\ref{eq:Jbound}).

Since there are $T-\lceil T/2\rceil\geq\frac{T-1}{2}$ terms in the
summation on the left-hand side, it follows that there exists $t\in\{J,....,T\}$
such that 
\begin{align*}
\|\nabla f_{\beta}(\blam_{t})\|_{*}^{2}\ \ \leq & \ \ \frac{1}{T-\lceil T/2\rceil}\cdot\frac{128\left[2M+\frac{\log(1/s)+1}{\beta}\right]^{2}}{T\eta^{2}}\\
\leq & \ \ \frac{256\left[2M+\frac{\log(1/s)+1}{\beta}\right]^{2}}{(T-1)^{2}\,\eta^{2}}.
\end{align*}
By taking square roots we complete the proof.
\end{proof}
Here we pause to define a useful norm $\|\cdot\|_{\S}$ over matrices:
\[
\|A\|_{\S}:=\sum_{ij}|A_{ij}|.
\]
Given this definition, we have tht for any $A,B\in\R^{m\times n}$,
\[
\langle A,B\rangle\le\max_{i,j}\left\{ |B_{ij}|\right\} \,\|A\|_{\S}.
\]
We also need the following technical lemma (which appears as Lemma
7 in \cite{altschuler2017near}), which controls the procedure of
rounding to a primal-feasible solution in terms of the errors of the
marginals.
\begin{lem}
\label{lem:round-ot} Given any $\mu\in\Delta_{m}$, $\nu\in\Delta_{n}$,
and $\pi\in\R_{+}^{m\times n}$, Algorithm 2 of \cite{altschuler2017near}
takes $O(mn)$ operations to output $\hat{\pi}$ such that:
\[
\hat{\pi}\mathbf{1}_{n}=\mu,\quad\hat{\pi}^{\top}\mathbf{1}_{m}=\nu,
\]
and
\[
\|\hat{\pi}-\pi\|_{\S}\le2\left[\|\pi\mathbf{1}_{n}-\mu\|_{1}+\|\pi^{\top}\mathbf{1}_{m}-\nu\|_{1}\right].
\]
\end{lem}

The proof is given in Appendix A.4 of \cite{altschuler2017near}.

Now we proceed with the proof of Theorem \ref{thm:otcomplexity},
which we restate here.

\otcomthm*
\begin{proof}
In this proof, we use $\pi_{\star}$ to denote an optimal solution
for the unregularized OT problem (\ref{eq:ot_lp}). Recall the construction
of $\tilde{\pi}=\pi_{\beta,\blam_{\tilde{t}}}$ from the discussion
preceding the statement of Theorem \ref{thm:otcomplexity} in the
main text, and recall that $\hat{\pi}$ is furnished by applying the
rounding procedure (Algorithm 2 of \cite{altschuler2017near}) to
$\tilde{\pi}$. Define marginals $\tilde{\mu}:=\tilde{\pi}\mathbf{1}_{n}$
and $\tilde{\nu}:=\tilde{\pi}^{\top}\mathbf{1}_{m}$ for $\tilde{\pi}$.

In order to control the rounding procedure, Lemma \ref{lem:round-ot}
motivates us to bound $\Vert\tilde{\mu}-\mu\Vert_{1}+\Vert\tilde{\nu}-\nu\Vert_{1}$,
which relates to the dual gradient $\nabla f_{\beta}(\blam_{\tilde{t}})$,
which is in turn bounded via Theorem \ref{thm:otgrad}.

Indeed, by construction, 
\begin{align*}
\Vert\nabla f_{\beta}(\blam_{\tilde{t}})\Vert_{*}\ \ = & \ \ \sqrt{\frac{1}{2}\left(\Vert\tilde{\mu}-\mu\Vert_{1}^{2}+\Vert\tilde{\mu}-\nu\Vert_{1}^{2}\right)}\\
\geq & \ \ \frac{1}{2}\left(\Vert\tilde{\mu}-\mu\Vert_{1}+\Vert\tilde{\nu}-\nu\Vert_{1}\right)
\end{align*}
 and by Theorem \ref{thm:otgrad}, 
\[
\Vert\nabla f_{\beta}(\blam_{\tilde{t}})\Vert_{*}\leq\frac{16\left[2M+\frac{\log(1/s)+1}{\beta}\right]}{(T-1)\eta}.
\]
 Therefore 
\begin{equation}
\Vert\tilde{\mu}-\mu\Vert_{1}+\Vert\tilde{\nu}-\nu\Vert_{1}\leq\frac{32\left[2M+\frac{\log(1/s)+1}{\beta}\right]}{(T-1)\eta}.\label{eq:marbound}
\end{equation}
Then by Lemma \ref{lem:round-ot}, together with (\ref{eq:marbound}),
we have that 
\[
\|\hat{\pi}-\tilde{\pi}\|_{\S}\leq\frac{64\left[2M+\frac{\log(1/s)+1}{\beta}\right]}{(T-1)\eta},
\]
 which in turn implies that 
\begin{equation}
\left|\left\langle c,\hat{\pi}\right\rangle -\left\langle c,\tilde{\pi}\right\rangle \right|\leq\xi:=\frac{64M\left[2M+\frac{\log(1/s)+1}{\beta}\right]}{(T-1)\eta}.\label{eq:xi1}
\end{equation}

Meanwhile, observe that $\tilde{\pi}$ is the \emph{exact} solution
of the entropically regularized problem 
\begin{align}
\underset{\pi\in\R^{m\times n}}{\text{minimize}}\ \  & \left\langle c,\pi\right\rangle +\beta^{-1}S(\pi)\label{eq:altered}\\
\text{subject to}\ \  & \pi\mathbf{1}_{n}=\tilde{\mu},\nonumber \\
 & \pi^{\top}\mathbf{1}_{m}=\tilde{\nu},\nonumber \\
 & \left\langle \pi,\mathbf{1}_{m\times n}\right\rangle =1\nonumber 
\end{align}
 with slightly altered marginal constraints. Let $\pi'$ be the result
of applying the rounding procedure (Algorithm 2 of \cite{altschuler2017near})
to $\pi_{\star}$ with marginal constraints specified by $\tilde{\mu}$
and $\tilde{\nu}$. Then again by Lemma \ref{lem:round-ot}, together
with (\ref{eq:marbound}), we deduce that 
\begin{equation}
\left|\left\langle c,\pi'\right\rangle -p_{\star}\right|=\left|\left\langle c,\pi'\right\rangle -\left\langle c,\pi_{\star}\right\rangle \right|\leq\xi.\label{eq:xi2}
\end{equation}
 Now since $\pi'$ is feasible for the altered problem (\ref{eq:altered}),
we know that 
\[
\left\langle c,\pi'\right\rangle +\beta^{-1}S(\pi')\geq\left\langle c,\tilde{\pi}\right\rangle +\beta^{-1}S(\tilde{\pi}).
\]
 Then by applying (\ref{eq:xi1}) and (\ref{eq:xi2}) and using the
fact that the entropic diameter of the set of feasible solutions is
bounded by $\log(mn)$, we deduce that 
\[
\left\langle c,\hat{\pi}\right\rangle \leq p_{\star}+2\xi+\beta^{-1}\log n,
\]
 which completes the proof of the first statement of the theorem.
The remaining statements follow from elementary calculations.
\end{proof}

\section{Proofs for Max-Cut SDP rounding \label{app:maxcutrounding}}

Before giving the proof of Theorem \ref{thm:maxcut}, we state and
prove the key lemma that enables the rounding argument.

\maxcutroundinglem* 
\begin{proof}
First we will define $\tilde{X}\succeq0$ in terms of $X$ by `deflating'
the diagonal values that are too large to ensure that $\mathrm{diag}(\tilde{X})\leq\mathbf{1}_{n}$.
Then we will add a nonnegative diagonal component to obtain $\tilde{X}$,
ensuring that $\mathrm{diag}(X')=\mathbf{1}_{n}$ holds exactly.

Specifically, let 
\[
\mathcal{S}=\{i\in\{1,\ldots,n\}\,:\,X_{ii}\geq1\}
\]
 denote the set of indices that we need to deflate. Note that 
\begin{equation}
\sum_{i\in\mathcal{S}}(X_{ii}-1)\leq\Vert\mathrm{diag}(X)-\mathbf{1}_{n}\Vert_{1}\leq\delta.\label{eq:Xiidelta}
\end{equation}

Then define a diagonal matrix $D=\mathrm{diag}(d)$, where 
\[
d_{i}=\begin{cases}
\frac{1}{\sqrt{X_{ii}}}, & i\in\mathcal{S},\\
1, & \text{otherwise},
\end{cases}
\]
 and define $\tilde{X}=DXD$, so $\tilde{X}\succeq0$ by construction.
Moreover we have that 
\[
\tilde{X}_{ii}=\begin{cases}
1, & i\in\mathcal{S},\\
X_{ii}, & \text{otherwise}.
\end{cases}
\]

Let $X=V^{\top}V$ denote an arbitrary Cholesky factorization of $X$,
and let $v_{1},\ldots,v_{n}$ denote the columns of $V$, so that
$X_{ij}=v_{i}\cdot v_{j}$. In particular $X_{ii}=\Vert v_{i}\Vert^{2}$,
and for $i\in\mathcal{S}$ we have $\Vert v_{i}\Vert^{2}\geq1$. In
terms of these vectors (\ref{eq:Xiidelta}) can be written as 
\begin{equation}
\sum_{i\in\mathcal{S}}(\Vert v_{i}\Vert^{2}-1)\leq\delta.\label{eq:videlta}
\end{equation}

Then we can similarly write $\tilde{X}=\tilde{V}^{\top}\tilde{V}$,
i.e., $\tilde{X}_{ij}=\tilde{v}_{i}\cdot\tilde{v}_{j}$ where $\tilde{V}$
has columns $\tilde{v}_{1},\ldots,\tilde{v}_{n}$ defined by 
\[
\tilde{v}_{i}=\begin{cases}
\frac{v_{i}}{\Vert v_{i}\Vert}, & i\in\mathcal{S},\\
v_{i}, & \text{otherwise}.
\end{cases}
\]

Then for $i,j\in\mathcal{S}$, we have 
\begin{align*}
\vert X_{ij}-\tilde{X}_{ij}\vert & =\left|v_{i}\cdot v_{j}-\tilde{v}_{i}\cdot\tilde{v}_{j}\right|\\
 & =\left|v_{i}\cdot v_{j}-\frac{v_{i}\cdot v_{j}}{\Vert v_{i}\Vert\,\Vert v_{j}\Vert}\right|\\
 & =\vert v_{i}\cdot v_{j}\vert\left|1-\frac{1}{\Vert v_{i}\Vert\,\Vert v_{j}\Vert}\right|\\
 & \overset{\textrm{(C-S)}}{\leq}\Vert v_{i}\Vert\,\Vert v_{j}\Vert\,\left|1-\frac{1}{\Vert v_{i}\Vert\,\Vert v_{j}\Vert}\right|\\
 & =\left|\Vert v_{i}\Vert\,\Vert v_{j}\Vert-1\right|.
\end{align*}
 Here `C-S' indicates an application of the Cauchy-Schwarz inequality.

Now since $i,j\in\mathcal{S}$, we have $\Vert v_{i}\Vert\geq1$ and
$\Vert v_{j}\Vert\geq1$, so we can actually remove the absolute value
bars in the last expression to obtain 
\[
\vert X_{ij}-\tilde{X}_{ij}\vert\leq\Vert v_{i}\Vert\Vert v_{j}\Vert-1.
\]
 Then apply Cauchy's inequality $xy\leq\frac{1}{2}x^{2}+\frac{1}{2}y^{2}$
to obtain 
\[
\vert X_{ij}-\tilde{X}_{ij}\vert\leq\frac{1}{2}(\Vert v_{i}\Vert^{2}-1)+\frac{1}{2}(\Vert v_{j}\Vert^{2}-1),\quad\text{for all }i,j\in\mathcal{S}.
\]
 Moreover, note that 
\[
X_{ij}=\tilde{X}_{ij},\quad\text{for all }i,j\notin\mathcal{S}.
\]

Next consider the situation $i\in\mathcal{S}$, $j\notin\mathcal{S}$,
so $\Vert v_{i}\Vert\geq1$ and $\Vert v_{j}\Vert\leq1$. Then 
\begin{align*}
\vert X_{ij}-\tilde{X}_{ij}\vert & =\left|v_{i}\cdot v_{j}-\tilde{v}_{i}\cdot\tilde{v}_{j}\right|\\
 & =\left|v_{i}\cdot v_{j}-\frac{v_{i}\cdot v_{j}}{\Vert v_{i}\Vert}\right|\\
 & =\vert v_{i}\cdot v_{j}\vert\,\left|1-\frac{1}{\Vert v_{i}\Vert}\right|\\
 & \leq\Vert v_{i}\Vert\,\Vert v_{j}\Vert\,\left|1-\frac{1}{\Vert v_{i}\Vert}\right|\\
 & =\,\Vert v_{j}\Vert\,\left|\Vert v_{i}\Vert-1\right|\\
 & \leq\Vert v_{i}\Vert-1\\
 & \leq\Vert v_{i}\Vert^{2}-1.
\end{align*}
 where we have used the facts that $\Vert v_{i}\Vert\geq1$ and $\Vert v_{j}\Vert\leq1$.

In summary, we have shown that 
\[
\vert X_{ij}-\tilde{X}_{ij}\vert\leq\begin{cases}
\frac{1}{2}(\Vert v_{i}\Vert^{2}-1)+\frac{1}{2}(\Vert v_{j}\Vert^{2}-1), & i,j\in\mathcal{S}\\
(\Vert v_{i}\Vert^{2}-1) & i\in\mathcal{S},\,j\notin\mathcal{S}\\
(\Vert v_{j}\Vert^{2}-1) & i\notin\mathcal{S},\,j\in\mathcal{S}\\
0, & i,j\notin\mathcal{S}.
\end{cases}
\]

Therefore, for arbitrary symmetric $A$, we have
\begin{align*}
\Tr[A(X-\tilde{X})] & \leq\sum_{i,j}\vert A_{ij}\vert\,\vert X_{ij}-\tilde{X}_{ij}\vert\\
 & =\sum_{i,j\in\mathcal{S}}\vert A_{ij}\vert\,\vert X_{ij}-\tilde{X}_{ij}\vert+\sum_{i\in\mathcal{S},j\notin\mathcal{S}}\vert A_{ij}\vert\,\vert X_{ij}-\tilde{X}_{ij}\vert\\
 & \quad\quad+\ \sum_{i\notin\mathcal{S},j\in\mathcal{S}}\vert A_{ij}\vert\,\vert X_{ij}-\tilde{X}_{ij}\vert+\sum_{i,j\notin\mathcal{S}}\vert A_{ij}\vert\,\vert X_{ij}-\tilde{X}_{ij}\vert\\
 & \leq\sum_{i,j\in\mathcal{S}}\vert A_{ij}\vert\,\left[\frac{1}{2}(\Vert v_{i}\Vert^{2}-1)+\frac{1}{2}(\Vert v_{j}\Vert^{2}-1)\right]+2\sum_{i\in\mathcal{S},j\notin\mathcal{S}}\vert A_{ij}\vert\,\left[(\Vert v_{i}\Vert^{2}-1)\right]\\
 & \leq2\sum_{i,j\in\mathcal{S}}\vert A_{ij}\vert\,(\Vert v_{i}\Vert^{2}-1)+2\sum_{i\in\mathcal{S},j\notin\mathcal{S}}\vert A_{ij}\vert\,\left[(\Vert v_{i}\Vert^{2}-1)\right]\\
 & =2\sum_{i\in\mathcal{S}}\left(\sum_{j\in[n]}\vert A_{ij}\vert\right)\,(\Vert v_{i}\Vert^{2}-1)\\
 & \leq2\Vert A\Vert_{\infty}\sum_{i\in\mathcal{S}}(\Vert v_{i}\Vert^{2}-1)\\
 & \leq2\delta\Vert A\Vert_{\infty},
\end{align*}
 where we have used (\ref{eq:videlta}) in the last line.

Finally we define 
\[
X'=\tilde{X}+\underbrace{\mathrm{diag}\left[\mathbf{1}-\mathrm{diag}(\tilde{X})\right]}_{=:\,Y}.
\]
 By construction we have $\mathbf{1}-\mathrm{diag}(\tilde{X})\geq0$
and $\tilde{X}\succeq0$, so $X'\succeq0$ automatically. Moreover,
\begin{align*}
\Tr[A(X-X')] & =\Tr[A(X-\tilde{X})]-\Tr[AY]\\
 & \leq2\delta\Vert A\Vert_{\infty}+\sum_{i}\vert A_{ii}\vert\,\vert Y_{ii}\vert.
\end{align*}
 Now $\vert A_{ii}\vert\le\Vert A\Vert_{\infty}$, and $Y_{ii}=1-\tilde{X}_{ii}\leq\vert1-X_{ii}\vert$,
hence $\sum_{i}\vert Y_{ii}\vert\leq\Vert\mathrm{diag}(X)-\mathbf{1}\Vert_{1}\leq\delta$.
The lemma follows.
\end{proof}
Then we can upgrade the preceding lemma to allow for general $\mathbf{b}$
not proportional to $\mathbf{1}_{n}$. 
\begin{cor}
\label{cor:round}Let\textbf{ $\mathbf{b}=(b_{i})\geq0$} and let
$\kappa=\frac{\max_{i}b_{i}}{\min_{i}b_{i}}$. Suppose that $X\succeq0$
with $\Vert\mathrm{diag}(X)-\mathbf{b}\Vert_{1}\leq\delta$. Then
there exists $X'\succeq0$ with $\mathrm{diag}(X')=\mathbf{b}$ such
that for any symmetric $A$, we have 
\[
\Tr[A(X-X')]\leq3\kappa\delta\Vert A\Vert_{\infty}.
\]
\end{cor}

\begin{proof}
Let $D$ be the diagonal matrix with entries $D_{ij}=\sqrt{b_{i}}\delta_{ij}$.
Then define 
\[
Y=D^{-1}XD^{-1},
\]
 so $Y\succeq0$ by construction and $Y_{ii}=\frac{X_{ii}}{b_{i}}$
for all $i$. Then compute: 
\[
\Vert\mathrm{diag}(Y)-\mathbf{1}\Vert_{1}=\sum_{i}\left|\frac{X_{ii}}{b_{i}}-1\right|\leq\sum_{i}\frac{1}{b_{i}}\vert X_{ii}-b_{i}\vert.
\]
 Note that we have $\frac{1}{b_{i}}\leq\frac{\kappa}{\Vert b\Vert_{\infty}}$
for all $i$, so we can continue our calculation: 
\[
\Vert\mathrm{diag}(Y)-\mathbf{1}\Vert_{1}\leq\frac{\kappa}{\Vert b\Vert_{\ell^{\infty}}}\sum_{i}\vert X_{ii}-b_{i}\vert\leq\frac{\kappa\delta}{\Vert b\Vert_{\ell^{\infty}}}=:\tilde{\delta}.
\]

Now apply the preceding lemma with $Y$ in the place of $X$ and $\tilde{\delta}$
in place of $\delta$ and $\tilde{A}=DAD$ in place of $A$ to conclude
that there exists $Y'\succeq0$ with $\mathrm{diag}(Y)=\mathbf{1}$
such that 
\[
\Tr[\tilde{A}(Y-Y')]\leq3\tilde{\delta}\Vert\tilde{A}\Vert_{\infty}.
\]
 Then we define 
\[
X':=DY'D
\]
 so $X'\succeq0$ and $\mathrm{diag}(X)=\mathbf{b}$ as desired, and
moreover 
\begin{align*}
\Tr[\tilde{A}(Y-Y')] & =\Tr[DAD(Y-Y')]\\
 & =\Tr[AD(Y-Y')D]\\
 & =\Tr[A(DYD-DY'D)]\\
 & =\Tr[A(X-X')].
\end{align*}
 Therefore we have shown 
\[
\Tr[A(X-X')]\leq3\tilde{\delta}\Vert\tilde{A}\Vert_{\infty}.
\]
 But 
\begin{align*}
\Vert\tilde{A}\Vert_{\infty} & =\max_{j}\sum_{i}\sqrt{b_{i}b_{j}}\vert A_{ij}\vert\leq\Vert b\Vert_{\infty}\Vert A\Vert_{\infty}.
\end{align*}
 In turn it follows that $\Tr[A(X-X')]\leq3\kappa\delta\Vert A\Vert_{\infty}$,
as was to be shown.
\end{proof}
Now we finally turn to the proof of Theorem \ref{thm:maxcut}, which
we restate here:

\maxcutroundingthm*
\begin{proof}
Let $X=X_{\beta,\blam}$ be as in the statement of the theorem. Note
that $\nabla f_{\beta}(\blam)=\mathrm{diag}(X)-\mathbf{b}$, so we
know that $\Vert\mathrm{diag}(X)-\mathbf{b}\Vert_{1}\leq\delta$.
Define $\tilde{\mathbf{b}}=\mathrm{diag}(X)$, so $\Vert\tilde{\mathbf{b}}-\mathbf{b}\Vert_{1}\leq\delta$.

Observe that $X$ is the \emph{exact}\textbf{ }solution of the entropically
regularized problem 
\begin{align}
\underset{X\in\R^{n\times n}}{\text{minimize}}\quad\quad & \Tr[CX]+\beta^{-1}S(X)\label{eq:perturbedSDP}\\
\text{subject to}\quad\quad & \mathrm{diag}(X)=\tilde{\mathbf{b}}.\nonumber 
\end{align}
 with slightly altered diagonal constraint.

Let $X_{\star}$ denote the optimizer of the unregularized SDP (\ref{eq:maxcut_sdp}).
In particular $\mathrm{diag}(X_{\star})=\mathbf{b}$. Then by applying
Corollary \ref{cor:round}, we can find $X_{\star}'\succeq0$ with
$\mathrm{diag}(X_{\star}')=\tilde{\mathbf{b}}$ (hence feasible for
the altered regularized problem (\ref{eq:perturbedSDP})) satisfying
\[
|\Tr[C(X_{\star}-X_{\star}')]|\leq3\kappa\delta\Vert C\Vert_{\infty},
\]
 hence 
\begin{align}
\Tr[CX_{\star}'] & \geq\Tr[CX_{\star}]-3\kappa\delta\Vert C\Vert_{\infty}\nonumber \\
 & =p_{\star}-3\kappa\delta\Vert C\Vert_{\infty}.\label{eq:sub1}
\end{align}

By the optimality of $X_{\blam}$ for (\ref{eq:perturbedSDP}), we
have 
\[
\Tr[CX_{\blam}]+\beta^{-1}S(X_{\blam})\leq\Tr[CX_{\star}']+\beta^{-1}S(X_{\star}'),
\]
 or 
\[
\Tr[CX_{\blam}]-\Tr[CX_{\star}']\leq\beta^{-1}\left[S(X_{\star}')-S(X_{\blam})\right].
\]
Then it is helpful to recall that $S(X)\in[0,-\log n]$ for all $X\in\mathcal{P}_{1}$.
Therefore 
\[
\Tr[CX_{\blam}]-\Tr[CX_{\star}']\leq\beta^{-1}\log n.
\]

Then combining with (\ref{eq:sub1}), we deduce: 
\begin{equation}
\Tr[CX_{\blam}]-p_{\star}\leq3\kappa\delta\Vert C\Vert_{\infty}+\beta^{-1}\log n.\label{eq:pstarupper}
\end{equation}
 This is our favorable upper bound for $\Tr[CX_{\blam}]$.

To get a similar lower bound, use Corollary \ref{cor:round} again
to construct $X'\succeq0$ with $\mathrm{diag}(X')=\mathbf{b}$ (hence
feasible for the unregularized SDP (\ref{eq:maxcut_sdp})) satisfying
\[
|\Tr[C(X_{\blam}-X')]|\leq3\kappa\delta\Vert C\Vert_{\infty},
\]
 hence 
\begin{equation}
\Tr[CX']\leq\Tr[CX_{\blam}]+3\kappa\delta\Vert C\Vert_{\infty}.\label{eq:sub2}
\end{equation}

Moreover, since $X'$ is feasible for (\ref{eq:maxcut_sdp}), which
has optimal value $p_{\star}$, we have 
\[
p_{\star}\leq\Tr[CX']\leq\Tr[CX_{\blam}]+3\kappa\delta\Vert C\Vert_{\infty}.
\]
 Therefore, combining with (\ref{eq:pstarupper}), we have
\[
\vert\Tr[CX_{\blam}]-p^{*}\vert\leq3\kappa\delta\Vert C\Vert_{\infty}+\beta^{-1}\log n
\]
 as desired.

Moreover, (\ref{eq:sub2}) and (\ref{eq:pstarupper}) together imply
\[
\Tr[CX']-p^{*}\leq6\kappa\delta\Vert C\Vert_{\infty}+\beta^{-1}\log n,
\]
 which completes the proof.
\end{proof}

\section{Proofs for Strong Perm-Synch SDP rounding \label{app:strongpermsynchrounding}}

Before giving the proof of Theorem \ref{thm:strongpermsynch}, we
state and prove the key lemma that enables the rounding argument.

\strongpermsynchroundinglem*
\begin{proof}
First factorize 
\[
X=Z^{\top}Z
\]
 where 
\[
Z=\left(\begin{array}{ccc}
Z^{(1)} & \cdots & Z^{(N)}\end{array}\right),
\]
 where the blocks $Z^{(i)}$ are $NK\times K$, $i=1,\ldots,N$. Therefore
we can write 
\[
X^{(i,j)}=Z^{(i)\top}Z^{(j)}.
\]

Now for each block we can perform a thin SVD $Z^{(i)}=U^{(i)}\Sigma^{(i)}V^{(i)\top}$,
where $U^{(i)}$ is $NK\times K$ and $\Sigma^{(i)}$ has diagonal
entries $\sigma_{k}^{(i)}$. Consider the modification $\tilde{\Sigma}^{(i)}$
of $\Sigma^{(i)}$ which thresholds all singular values at $1$, and
let the diagonal entries be denoted $\tilde{\sigma}_{k}^{(i)}=\min(1,\sigma_{k}^{(i)})$.
Then define $\tilde{Z}^{(i)}=U^{(i)}\tilde{\Sigma}^{(i)}V^{(i)^{\top}}$,
and in turn define 
\[
\tilde{Z}=\left(\begin{array}{ccc}
\tilde{Z}^{(1)} & \cdots & \tilde{Z}^{(N)}\end{array}\right)
\]
 and 
\[
\tilde{X}=\tilde{Z}^{\top}\tilde{Z},
\]
 the `deflated' version of $X$ (which we will shift upward later
to achieve feasibility).

Let $A$ be a symmetric $NK\times NK$ matrix. Then (using angle brackets
always to denote Frobenius innter products): 
\begin{align}
\Tr\left[A(X-\tilde{X})\right] & =\sum_{i,j=1}^{N}\left\langle A^{(i,j)},X^{(i,j)}-\tilde{X}^{(i,j)}\right\rangle ,\nonumber \\
 & =\sum_{i,j=1}^{N}\left\langle \tilde{A}^{(i,j)},\Sigma^{(i)}B^{(i,j)}\Sigma^{(j)}-\tilde{\Sigma}^{(i)}B^{(i,j)}\tilde{\Sigma}^{(j)}\right\rangle \nonumber \\
 & =\sum_{i,j=1}^{N}\left\langle \tilde{A}^{(i,j)},\Delta^{(i,j)}\right\rangle \nonumber \\
 & =\left\langle \tilde{A},\Delta\right\rangle ,\label{eq:AtildeDelta}
\end{align}
 where 
\[
\tilde{A}^{(i,j)}:=V^{(i)\top}A^{(i,j)}V^{(j)},
\]
\[
B^{(i,j)}:=U^{(i)\top}U^{(j)},
\]
\[
\Delta^{(i,j)}:=\Sigma^{(i)}B^{(i,j)}\Sigma^{(j)}-\tilde{\Sigma}^{(i)}B^{(i,j)}\tilde{\Sigma}^{(j)},
\]
 and $\tilde{A}$ and $\Delta$ are defined suitably in terms of the
appropriate blocks.

Now we can write the entries of $\Delta^{(i,j)}$ as 
\[
\Delta_{kl}^{(i,j)}=B_{kl}^{(i,j)}(\sigma_{k}^{(i)}\sigma_{l}^{(j)}-\tilde{\sigma}_{k}^{(i)}\tilde{\sigma}_{l}^{(j)}).
\]

Similar to the proof of Lemma \ref{lem:maxcutround}, let $\mathcal{S}$
denote the set of indices $(i,k)$ such that $\tilde{\sigma}_{k}^{(i)}\neq\sigma_{k}^{(i)}$,
so in particular $\sigma_{k}^{(i)}>1$ and $\tilde{\sigma}_{k}^{(i)}=1$
for $(i,k)\in\mathcal{S}$, while $\sigma_{k}^{(i)}=\tilde{\sigma}_{k}^{(i)}\leq1$
for $(i,k)\notin\mathcal{S}$.

Then if $(i,k),(j,l)\in\mathcal{S}$, we have 
\begin{align*}
\sigma_{k}^{(i)}\sigma_{l}^{(j)}-\tilde{\sigma}_{k}^{(i)}\tilde{\sigma}_{l}^{(j)} & =\sigma_{k}^{(i)}\sigma_{l}^{(j)}-1\\
 & \leq\frac{1}{2}\left([\sigma_{k}^{(i)}]^{2}-1\right)+\frac{1}{2}\left([\sigma_{l}^{(j)}]^{2}-1\right).
\end{align*}
 Meanwhile, if $(i,k)\in\mathcal{S}$ but $(j,l)\notin\mathcal{S}$,
we have 
\begin{align*}
\sigma_{k}^{(i)}\sigma_{l}^{(j)}-\tilde{\sigma}_{k}^{(i)}\tilde{\sigma}_{l}^{(j)} & =(\sigma_{k}^{(i)}-1)\sigma_{l}^{(j)}\\
 & \leq\sigma_{k}^{(i)}-1\\
 & \leq[\sigma_{k}^{(i)}]^{2}-1.
\end{align*}
 Similarly, if $(i,k)\notin\mathcal{S}$ but $(j,l)\in\mathcal{S}$,
we have 
\[
\sigma_{k}^{(i)}\sigma_{l}^{(j)}-\tilde{\sigma}_{k}^{(i)}\tilde{\sigma}_{l}^{(j)}\leq[\sigma_{l}^{(j)}]^{2}-1.
\]
 And finally, if $(i,k),(j,l)\notin\mathcal{S}$, we have 
\[
\sigma_{k}^{(i)}\sigma_{l}^{(j)}-\tilde{\sigma}_{k}^{(i)}\tilde{\sigma}_{l}^{(j)}=0.
\]

Moreover, $\vert B_{kl}^{(i,j)}\vert\leq1$ for all $i,j,k,l$ by
the Cauchy-Schwarz inequality, since $B^{(i,j)}=U^{(i)\top}U^{(j)}$,
hence all entries of $B^{(i,j)}$ can be viewed as inner products
of unit vectors. Therefore it follows that 
\[
\vert\Delta_{kl}^{(i,j)}\vert\leq\begin{cases}
\frac{1}{2}\left([\sigma_{k}^{(i)}]^{2}-1\right)+\frac{1}{2}\left([\sigma_{l}^{(j)}]^{2}-1\right), & (i,k),(j,l)\in\mathcal{S},\\
\frac{1}{2}\left([\sigma_{k}^{(i)}]^{2}-1\right), & (i,k)\in\mathcal{S},\,(j,l)\notin\mathcal{S},\\
\frac{1}{2}\left([\sigma_{l}^{(j)}]^{2}-1\right), & (i,k)\notin\mathcal{S},\,(j,l)\in\mathcal{S},\\
0. & (i,k),(j,l)\notin\mathcal{S}
\end{cases}
\]

Then compute, starting from (\ref{eq:AtildeDelta}): 
\begin{align}
\Tr\left[A(X-\tilde{X})\right] & =\left\langle \tilde{A},\Delta\right\rangle \nonumber \\
 & =\sum_{i,j=1}^{N}\sum_{k,l=1}^{K}\tilde{A}_{kl}^{(i,j)}\Delta_{kl}^{(i,j)}\nonumber \\
 & \leq\sum_{(i,k),(j,l)\in\mathcal{S}}\vert\tilde{A}_{kl}^{(i,j)}\vert\,\vert\Delta_{kl}^{(i,j)}\vert+\sum_{(i,k)\in\mathcal{S},(j,l)\notin\mathcal{S}}\vert\tilde{A}_{kl}^{(i,j)}\vert\,\vert\Delta_{kl}^{(i,j)}\vert+\nonumber \\
 & \quad\quad\ +\ \sum_{(i,k)\notin\mathcal{S},(j,l)\in\mathcal{S}}\vert\tilde{A}_{kl}^{(i,j)}\vert\,\vert\Delta_{kl}^{(i,j)}\vert+\sum_{(i,k),(j,l)\notin\mathcal{S}}\vert\tilde{A}_{kl}^{(i,j)}\vert\,\vert\Delta_{kl}^{(i,j)}\vert\nonumber \\
 & \leq\sum_{(i,k),(j,l)\in\mathcal{S}}\vert\tilde{A}_{kl}^{(i,j)}\vert\,\left[\frac{1}{2}\left([\sigma_{k}^{(i)}]^{2}-1\right)+\frac{1}{2}\left([\sigma_{l}^{(j)}]^{2}-1\right)\right]\nonumber \\
 & \quad\quad\ +2\sum_{(i,k)\in\mathcal{S},(j,l)\notin\mathcal{S}}\vert\tilde{A}_{kl}^{(i,j)}\vert\,\left([\sigma_{k}^{(i)}]^{2}-1\right)\nonumber \\
 & \leq2\sum_{(i,k),(j,l)\in\mathcal{S}}\vert\tilde{A}_{kl}^{(i,j)}\vert\,\left([\sigma_{k}^{(i)}]^{2}-1\right)+2\sum_{(i,k)\in\mathcal{S},(j,l)\notin\mathcal{S}}\vert\tilde{A}_{kl}^{(i,j)}\vert\,\left([\sigma_{k}^{(i)}]^{2}-1\right)\nonumber \\
 & =2\sum_{(i,k)\in\mathcal{S}}\left(\sum_{j=1}^{N}\sum_{l=1}^{K}\vert\tilde{A}_{kl}^{(i,j)}\vert\right)\left([\sigma_{k}^{(i)}]^{2}-1\right)\nonumber \\
 & \leq2\times\max_{i,k\in\mathcal{S}}\left\{ \sum_{j=1}^{N}\sum_{l=1}^{K}\vert\tilde{A}_{kl}^{(i,j)}\vert\right\} \times\sum_{(i,k)\in\mathcal{S}}\left([\sigma_{k}^{(i)}]^{2}-1\right).\label{eq:ADelta}
\end{align}

In the last expression (\ref{eq:ADelta}), consider the last factor.
Note that $[\sigma_{k}^{(i)}]^{2}$, $k=1,\ldots,K$, are the eigenvalues
of the diagonal block $X^{(i,i)}=Z^{(i)\top}Z^{(i)}$. The nuclear
norm error of this block can be rewritten in terms of the eigenvalues
as 
\[
\Vert X^{(i,i)}-\mathbf{I}_{K}\Vert_{\Tr}=\sum_{k=1}^{K}\left|[\sigma_{k}^{(i)}]^{2}-1\right|.
\]
 Therefore 
\[
\sum_{(i,k)\in\mathcal{S}}\left([\sigma_{k}^{(i)}]^{2}-1\right)\leq\sum_{i=1}^{N}\sum_{k=1}^{K}\left|[\sigma_{k}^{(i)}]^{2}-1\right|=\sum_{i=1}^{N}\Vert X^{(i,i)}-\mathbf{I}_{K}\Vert_{\Tr}\leq\delta,
\]
 where the last inequality follows by hypothesis.

Then consider the penultimate factor in the last expression of (\ref{eq:ADelta}).
We have 
\begin{align*}
\vert\tilde{A}_{kl}^{(i,j)}\vert & =\vert[V^{(i)\top}A^{(i,j)}V^{(j)}]_{kl}\vert\\
 & \leq\vert v_{k}^{(i)\top}A^{(i,j)}v_{l}^{(j)}\vert\\
 & \leq\Vert A^{(i,j)}\Vert_{2},
\end{align*}
 where we use $v_{k}^{(i)}$ to denote the $k$-th column of $V^{(i)}$
for all $i,k$, and we have used the fact that these columns are unit
vectors (since the $V^{(i)}$ are all unitary). Then it follows that
\begin{equation}
\max_{i,k\in\mathcal{S}}\left\{ \sum_{j=1}^{N}\sum_{l=1}^{K}\vert\tilde{A}_{kl}^{(i,j)}\vert\right\} \leq K\max_{i=1,\ldots,N}\left\{ \sum_{j=1}^{N}\Vert A^{(i,j)}\Vert_{2}\right\} =K\,\vert\vert\vert A\vert\vert\vert.\label{eq:annoying}
\end{equation}

Combining our bounds with (\ref{eq:ADelta}) we conclude that 
\begin{equation}
\Tr\left[A(X-\tilde{X})\right]\leq2K\delta\,\vert\vert\vert A\vert\vert\vert.\label{eq:AXXtildetriple}
\end{equation}

Finally, we need to shift up the deflated matrix $\tilde{X}$ to ensure
exact feasibility. Now the diagonal blocks 
\[
\tilde{X}^{(i,i)}=\tilde{Z}^{(i)\top}\tilde{Z}^{(i)}=V^{(i)}[\tilde{\Sigma}^{(i)}]^{2}V^{(i)\top}
\]
 satisfy $\tilde{X}^{(i,i)}\preceq\mathbf{I}_{K}$, since $[\tilde{\Sigma}^{(i)}]^{2}\preceq\mathbf{I}_{K}$
by construction. Therefore we can define 
\[
Y^{(i)}:=\mathbf{I}_{K}-\tilde{X}^{(i,i)}\succeq0
\]
 for all $i=1,\ldots,N$, and in turn 
\[
X':=\tilde{X}+\bigoplus_{i=1}^{N}Y^{(i)},
\]
 which is exactly feasible.

To complete the proof, given (\ref{eq:AXXtildetriple}), it suffices
to bound 
\begin{equation}
\left|\left\langle A,\bigoplus_{i=1}^{N}Y^{(i)}\right\rangle \right|\leq\delta\,\vert\vert\vert A\vert\vert\vert.\label{eq:lastclaim}
\end{equation}
 To justify this bound, compute 
\begin{align*}
\left|\left\langle A,\bigoplus_{i=1}^{N}Y^{(i)}\right\rangle \right| & \leq\sum_{i=1}^{N}\left\langle A^{(i,i)},Y^{(i)}\right\rangle \\
 & \leq\sum_{i=1}^{N}\Vert A^{(i,i)}\Vert_{2}\Vert Y^{(i)}\Vert_{\Tr}\\
 & \leq\vert\vert\vert A\vert\vert\vert\,\sum_{i=1}^{N}\Tr[Y^{(i)}]\\
 & =\vert\vert\vert A\vert\vert\vert\,\sum_{i=1}^{N}\sum_{k=1}^{K}(1-[\tilde{\sigma}_{k}^{(i)}]^{2})\\
 & \leq\vert\vert\vert A\vert\vert\vert\,\sum_{i=1}^{N}\sum_{k=1}^{K}\vert1-[\sigma_{k}^{(i)}]^{2}\vert\\
 & =\vert\vert\vert A\vert\vert\vert\,\sum_{i=1}^{N}\Vert X^{(i,i)}-\mathbf{I}_{K}\Vert_{\Tr}\\
 & \leq\delta\,\vert\vert\vert A\vert\vert\vert,
\end{align*}
 where the last inequality again follows by hypothesis. This establishes
(\ref{eq:lastclaim}), which completes the proof of the lemma.
\end{proof}
Now we finally turn to the proof of Theorem \ref{thm:strongpermsynch},
which we restate here:

\strongpermsynchroundingthm*
\begin{proof}
The proof is completely analogous to the proof of Theorem \ref{thm:maxcut},
using Lemma \ref{lem:strongpermsynchround} analogously to Lemma \ref{lem:maxcutround}.
\end{proof}

\end{document}